\documentclass[mnsc,nonblindrev]{informs3-YCA}

\OneAndAHalfSpacedXI %

\usepackage{endnotes}
\let\footnote=\endnote

\usepackage{url}
\usepackage{enumitem}
\usepackage{amsmath}
\usepackage{amssymb}
\usepackage{natbib}
\bibpunct[, ]{(}{)}{,}{a}{}{,}%
\def\bibfont{\small}%
\def\bibsep{\smallskipamount}%
\usepackage{tikz}
\usepackage{pgfplots}
\usetikzlibrary{arrows}
\usetikzlibrary{calc,arrows.meta,positioning}
\usepackage{ctable}
\usepackage{booktabs}
\usepackage{multirow}
\usepackage{algorithm}
\usepackage[noend]{algorithmic}
\usepackage{color}
\usepackage{graphicx}
\usepackage{enumitem}
\usepackage{cancel}

\allowdisplaybreaks

\usepackage{mathtools}
\usepackage[hidelinks]{hyperref}

\usepackage{xspace}

\usepackage{pgfplotstable}
\usetikzlibrary{arrows.meta}
\usepgfplotslibrary{groupplots}

\usepackage[title]{appendix}
\usepackage{titletoc}
\def\Pbb{\mathbb{P}}

\def\Ebb{\mathbb{E}}

\def\Dcal{\mathcal{D}}
\def\Ncal{\mathcal{N}}

\def\Acal{\mathcal{A}}
\def\Ucal{\mathcal{U}}

\def\Scal{\mathcal{S}}
\def\Gcal{\mathcal{G}}
\def\Rcal{\mathcal{R}}
\def\Ical{\mathcal{I}}

\def\Pcal{\mathcal{P}}

\usepackage{accents}
\newcommand{\ubar}[1]{\underaccent{\bar}{#1}}
\newcommand{\pr}[1]{{\rm Pr} \left[ #1 \right]}

\newcommand{\ex}[1]{{\mathbb E} \left[ #1 \right]}

\newcommand{\eps}{\epsilon}
\newcommand{\rev}{{\cal R}}

\def\Ocal{\mathcal{O}}
\def\Fcal{\mathcal{F}}

\def\Ccal{\mathcal{C}}
\def\Vcal{\mathcal{V}}
\newcommand{\opt}{\mathrm{OPT}}
\newcommand{\even}{\mathrm{EVEN}}
\newcommand{\odd}{\mathrm{ODD}}

\def\FP{\textsc{FP}}
\def\JLP{\textsc{JLP}}
\def\REV{\textsf{Rev}}

\def\thetab{\boldsymbol{\theta}}

\def\nopurchase{\texttt{0}}

\DeclarePairedDelimiter\floor{\lfloor}{\rfloor}

\usepackage{titlesec}
\titleformat{\paragraph}[runin]{\normalfont\bfseries}{\theparagraph}{1em}{}

\usepackage{natbib}
 \bibpunct[, ]{(}{)}{,}{a}{}{,}%
 \def\bibfont{\small}%
 \def\bibsep{\smallskipamount}%

\TheoremsNumberedThrough     %
\ECRepeatTheorems

\EquationsNumberedThrough    %

\newcounter{jakecounter}

\newcommand{\YC}[1]{{\color{black}#1}}

\begin{document}

\RUNAUTHOR{Akchen, {\c C}ak{\i}ro{\u g}lu, Caro, and Feldman}

\RUNTITLE{Assortment and Inventory Optimization under the CFTC Choice Model}

\TITLE{\large Do You Have My Size In Stock? Assortment and Inventory Optimization Under the Consider-Fit-Then-Choose Choice Model}

\ARTICLEAUTHORS{%
	\AUTHOR{Yi-Chun Akchen$^1$, Kaan {\c C}ak{\i}ro{\u g}lu$^1$, Felipe Caro$^2$, Jacob Feldman$^3$}
	\AFF{$^1$School of Management, University College London, London E14 5AB, United Kingdom}
	\AFF{$^2$Anderson School of Management, University of California, Los Angeles, California 90095, United States}
	\AFF{$^3$Olin Business School, Washington University in St. Louis, St. Louis 63130, Missouri, United States}
	\AFF{yi-chun.akchen@ucl.ac.uk, kaan.cakiroglu.25@ucl.ac.uk, felipe.caro@anderson.ucla.edu, jbfeldman@wustl.edu}
}

\ABSTRACT{In apparel retail and other applications, when a customer's preferred size is unavailable, demand may shift to nearby sizes. This substitution creates new assortment and inventory optimization challenges because product availability---and, consequently, demand---becomes coupled across sizes. To address these challenges, we introduce the \emph{consider-fit-then-choose} (CFTC) model, a simple yet rich framework for capturing size-dependent choice behavior. The model applies to settings in which each product may be offered in multiple sizes and the set of available sizes directly affects both customer preferences and consideration sets through a notion of fit, measured by the distance between an available size and the customer's ideal size. Under this model, we study both assortment optimization and the \emph{show-all inventory selection problem}, in which the retailer chooses initial inventory levels before the selling horizon and subsequently offers every in-stock product to each arriving customer.

We first show that assortment optimization under the CFTC model is NP-hard and develop a polynomial-time approximation scheme (PTAS) when customers are willing to deviate by at most $O(1)$ sizes from their ideal size. Combined with the recent black-box framework of~\cite{fu2026joint}, this result yields a nearly $0.272$-approximation for the corresponding show-all inventory selection problem. We then ask whether stronger guarantees can be obtained by exploiting the specific choice dynamics of the CFTC model. Our approach to both the fluid and stochastic versions of the inventory problem begins by stocking only every other size, thereby decoupling demand across the stocked sizes. Under this restriction, the CFTC model reduces to a special class of mixed multinomial logit models, which we prove satisfies the convex chain decomposition (CCD) property introduced by~\cite{goyal2023pricing}. This property allows us to translate solutions to a classical choice-based linear program into feasible show-all inventory policies.

Building on this structure, for the fluid problem we develop a polynomial-time $(\frac{1}{2}-\eps)$-approximation algorithm under adjacent-size substitution and a mild condition on the preference weights. For the stochastic problem, we establish an asymptotic $\frac{1}{2}$-approximation guarantee using a new coupling argument that connects the stochastic inventory process to its fluid counterpart. Numerical experiments on instances calibrated using footwear data demonstrate that our inventory solutions achieve small optimality gaps across a broad range of substitution patterns and problem settings.
}%

\HISTORY{This version: August 14, 2026.}

\maketitle

\vspace{-1.0cm}

\section{Introduction}

Apparel retailing has long been an important application area for operations management because of the industry's short product life cycles, uncertain demand, and extensive product variety. These characteristics have motivated a rich stream of research on analytical models for improving operational decisions, including pricing \citep{caro2012clearance}, initial inventory allocation \citep{gallien2015initial}, and buy-back decisions for rental fashion products \citep{apaolaza2025rented}. Among these decisions, inventory management remains one of the most fundamental, as retailers must determine how much inventory to procure before the selling season begins despite substantial uncertainty about future demand. Prior work has developed a variety of inventory-planning models for apparel products, including policies that account for seasonal demand patterns \citep{smith1998clearance}, stochastic dynamic programs for initial and replenishment ordering \citep{fisher2001optimizing}, and mixed-integer optimization models for inventory allocation under practical merchandising constraints \citep{caro2010inventory}.

Unlike in many other retail settings, apparel inventory management features a particularly strong interaction between inventory availability and customer choice, with stockouts inducing substantial demand substitution \citep{ergin2022empirical,li2022estimating}. When a customer's preferred product is unavailable, the customer may purchase an alternative rather than leave without making a purchase \citep[Chapter~4]{gallego2019revenue}. Ignoring such substitution can lead to biased demand estimates and suboptimal operational decisions, thereby reducing the profitability of inventory policies \citep{campo2000towards,musalem2010structural,che2012investigating}. Although substitution across product styles has received considerable attention in the literature \citep{boada2020estimating}, apparel products exhibit an additional and distinctive form of substitution across sizes. Unlike style, which captures intrinsic product attributes such as design and brand, size primarily determines how well a product physically fits a customer. Consequently, size substitution has a localized and ordered structure: customers may be willing to substitute to nearby sizes but are generally unwilling to purchase sizes that differ substantially from their ideal fit. Indeed, recent empirical evidence suggests that such localized size substitution can be economically significant. Using transaction data from one of the largest sports-footwear retailers in China, \citet{li2022estimating} estimate that nearly one quarter of the unmet demand for an out-of-stock size spills over to adjacent sizes of the same style. This finding suggests that inventory decisions are coupled not only across styles but also across sizes, because the inventory available in one size directly affects demand for neighboring sizes.

Despite its practical importance, the operations literature contains relatively few models that explicitly account for size substitution through a notion of fit. A notable recent exception is \citet{akchen2023size}, who propose a two-stage consider-then-choose framework for apparel purchases. In their model, customers first determine which styles are acceptable based on the availability of sizes near their ideal size and then choose among the considered products according to a multinomial logit (MNL) model. Using a proprietary footwear dataset, they estimate a magnitude of size substitution similar to that reported by \citet{li2022estimating}.

\paragraph{Our main focus.}
In this paper, we propose a generalization of the model considered in~\cite{akchen2023size}, which we term the \emph{consider-fit-then-choose} (CFTC) model, named playfully for its connection to the well-known consider-then-choose framework for modeling choice behavior~\citep{farias2013nonparametric,aouad2021assortment,akchen2025consider}. As fully formalized in Section~\ref{sec:model}, the CFTC model captures two-dimensional demand substitution across both product bases (e.g., styles or brands) and sizes. For each product base, a customer first identifies the available size closest to her ideal size, provided that this size lies within her acceptable range. These product--size combinations collectively form the customer's consideration set. The customer then makes a final choice according to a modified MNL model in which the preference weight of a product is discounted whenever its available size differs from the customer's ideal size. Although motivated by apparel retailing, the CFTC model also applies to a broader range of settings (Section~\ref{subsec:model_comparison}).

Under this choice model, we study a canonical inventory optimization problem, which we refer to as the \emph{show-all inventory selection problem}. Before the start of a finite selling horizon, the retailer chooses integral initial inventory levels for all products, subject to a cardinality constraint on total inventory. Thereafter, every arriving customer chooses from all products that remain in stock---hence the \emph{show-all} distinction in the name of the problem---so that changes in the offered assortment arise only through random stockout events. The retailer's objective is to select a feasible initial inventory vector that maximizes expected revenue over the selling horizon. As discussed further in our literature review in Section~\ref{subsec:lit_review}, this problem has primarily been studied under rank-based choice models~\citep{goyal2016near,goldstein2023dynamic} and the Multinomial Logit (MNL) model~\citep{aouad2018greedy,aouad2023stability,sun2024unified}. Very recent work by~\cite{fu2026joint} elegantly provides a black-box $0.272\alpha$-approximation for any choice model within the random-utility-maximization framework, a broad class that includes the CFTC model. Here, $\alpha$ denotes the best available approximation guarantee for the corresponding cardinality-constrained assortment optimization problem under the choice model at hand.

\subsection{Main contributions}

In what follows, we summarize our main contributions, in particular focusing on how our results exploit and build upon the result of~\cite{fu2026joint}.

\paragraph{Assortment optimization (Section~\ref{sec:assortment_optimization} and Appendix~\ref{app-sec:AO}).}
So as to exploit the result of~\cite{fu2026joint}, we first consider cardinality-constrained assortment optimization under the most general form of the CFTC model. In this problem, the retailer selects an assortment to maximize expected revenue when customer choice is governed by the CFTC model, subject to an upper bound on the number of products that can be offered. We show that even the unconstrained version of this problem is NP-hard when each product base is offered in only two sizes. Despite this hardness, we develop a polynomial-time approximation scheme (PTAS) that computes a $(1-\epsilon)$-optimal solution when customers consider only products whose sizes deviate from their ideal size by at most $O(1)$ size levels. Our approach is based on a carefully designed approximate dynamic program that identifies a small number of unprofitable sizes to exclude, thereby decoupling the problem across blocks of contiguous sizes. Combined with the recent black-box result of~\cite{fu2026joint}, this PTAS yields a $(0.272-\epsilon)$-approximation guarantee for the show-all inventory selection problem.

We then ask whether this nearly $0.272$ guarantee can be improved in certain settings by developing algorithms tailored specifically to the choice dynamics of the CFTC model. Interestingly, as summarized next, pursuing this question uncovers a collection of new algorithmic and structural insights that are entirely distinct from those underlying~\cite{fu2026joint}. Moreover, several of these techniques extend beyond the CFTC model and may prove useful for inventory optimization under broader classes of choice models. For tractability, we consider a simplified variant of the CFTC model in which, most notably, a customer expands her consideration set to include only one of the two adjacent sizes when her ideal size is out of stock.

\paragraph{A new CCD choice model (Section~\ref{sec:even_odd}).}
Our approaches for tackling the show-all inventory selection problem under the CFTC model begin by restricting attention to solutions that stock only every other size. This restriction loses at most one half of the optimal value but, critically, decouples demand across the stocked sizes, since each customer will deviate only to neighbors of their ideal size. In doing so, the CFTC model reduces to what we call the \emph{mixed-NP-MNL} model, a special class of mixed-MNL models in which customer types differ only in their no-purchase weights. Thus, our original inventory problem effectively reduces to the show-all inventory selection problem under the mixed-NP-MNL model. From here, we prove that the mixed-NP-MNL model satisfies the \emph{convex chain decomposition} (CCD) property introduced by~\citet{goyal2023pricing}, whose previously known members included only the MNL and Markov chain choice models. This property allows us to formulate the fluid version of our inventory problem as a simple adaptation of the classical choice-based deterministic linear program~\citep{gallego2004managing,liu2008choice}. In general, this linear program does not enforce that the displayed assortment consists of all in-stock products. The CCD property provides precisely the bridge needed to overcome this issue: it allows us to show that an optimal solution to the linear program can be implemented as a show-all policy, thereby establishing the validity of our approach.

\paragraph{Results for the fluid setting (Section~\ref{sec:fluid_setting}).} We start by considering the fluid version of the show-all inventory problem, in which stochastic demand is replaced by its mean. In this setting, we develop a $(\frac{1}{2}-\epsilon)$-approximation algorithm whose running time is polynomial when the ratio between the maximum and minimum base preference weights is $O(1)$. As discussed above, our approach first restricts attention to stocking either the even- or odd-indexed sizes, losing at most one half of the optimal revenue while decoupling demand across the retained sizes. We then partition products into weight classes and round their preference weights within each class, which reveals a useful structure in the optimal inventory vector: within each class, products can be partitioned into \emph{stocking groups} according to their inventory levels, with higher-revenue products receiving weakly more inventory. After efficiently guessing these groups and the inventories of low-stock products, the remaining high-stock inventory levels are determined through a fluid linear program. Crucially, the CCD property of the resulting mixed-NP-MNL model allows us to show that the solution of this linear program can be implemented as a feasible show-all policy, completing the $(\frac{1}{2}-\epsilon)$ guarantee.

\paragraph{Results for the stochastic setting (Section~\ref{sec:asym_one_half}).}
We next study the stochastic version of the inventory problem, in which inventory is depleted one unit at a time according to customers' discrete choices. Unlike in the fluid setting, these choices generate random inventory depletion and stockout times, so the sequence of offered assortments depends on the realized purchase outcomes. We develop an asymptotic $\frac{1}{2}$-approximation algorithm for this stochastic inventory optimization problem. Our approach again restricts attention to either even- or odd-indexed sizes, thereby reducing the problem to independent mixed-NP-MNL instances while losing at most one half of the optimal benchmark. We then solve a joint fluid linear program that simultaneously determines inventory levels and assortment offerings, and round down its inventory solution to obtain feasible initial inventories. The CCD property of the mixed-NP-MNL model allows us to show that, in the fluid setting, the resulting show-all inventory process asymptotically achieves the value prescribed by this linear program.

The remaining challenge is to transfer this guarantee to the stochastic setting, where random purchases lead to random stockout trajectories. We overcome this difficulty through a novel coupling argument. Specifically, we construct an auxiliary stochastic process that follows the stockout trajectory prescribed by the fluid solution while preserving randomness in customer choices. We show that the expected revenue of this auxiliary process asymptotically approaches the fluid revenue and is no greater than the expected revenue of the original stochastic inventory process. This completes the bridge from the fluid linear program to the stochastic $\frac{1}{2}$-approximation guarantee. We believe that this coupling framework may be of independent interest for analyzing stochastic inventory systems under more general forms of dynamic demand substitution. Moreover, as discussed at the end of this section, this asymptotic result extends trivially when there are ordering or stocking costs associated with each products, in addition to the upper bound on the total number of units that can stocked.  To the best of knowledge, the result of~\cite{fu2026joint} cannot accommodate these sorts of costs.

\paragraph{Numerical experiments (Section~\ref{sec:numerical_experiments}).} We complement our theoretical results with numerical experiments based on instances calibrated from real-world footwear data. The experiments demonstrate that the proposed inventory policies achieve consistently small optimality gaps relative to the fluid upper bound while remaining effective across a wide range of problem settings and substitution levels. These results suggest that the structural insights developed in the paper also translate into strong practical performance.

\subsection{Related literature}\label{subsec:lit_review}

\paragraph{Choice models with consideration sets.} Our CFTC model generalizes the choice model proposed by~\cite{akchen2023size}; as discussed in Section~\ref{subsec:model_comparison}, this extension better captures several important features of real-world size-substitution behavior. More broadly, our work contributes to the growing literature on assortment optimization under consideration set-based choice models~\citep{feldman2018capacitated,feldman2019assortment,aouad2021assortment,gallego2024random,aouad2025click,akchen2025consider,farzaneh2026feature}. A key distinction of the CFTC model is that customers' consideration sets are \emph{endogenous}: rather than being specified independently of the retailer's decisions, they depend on the offered assortment through customers' size-substitution behavior. This endogeneity creates new challenges for assortment and inventory optimization because changes in product availability affect not only customers' final choices but also the products they consider. These challenges become particularly pronounced in the inventory setting, where stockouts dynamically reshape consideration sets and, consequently, future demand.

The CFTC model also contributes to the growing literature on choice models tailored to apparel products. In addition to~\cite{akchen2023size}, our work is related to~\cite{boada2020estimating} and~\cite{alavi2024designing}. In~\cite{boada2020estimating}, all sizes associated with the same product base are aggregated into a single product. Motivated by the \emph{broken assortment effect}, whereby demand for an aggregated product decreases as more of its sizes become unavailable, the authors develop an inventory-dependent choice model and optimize initial inventory levels through an integer programming formulation. In contrast, both our work and~\cite{alavi2024designing} model products at the base--size level. The latter study proposes a two-stage choice model in which customers first choose a product base and then select among the available sizes within that base. Building on this model, the authors study an assortment optimization problem that balances revenue with size fairness across customers with different body sizes. Our model differs by allowing size availability to determine which product bases enter a customer's consideration set and by permitting substitution across both bases and sizes.

\paragraph{The show-all inventory selection problem.}
Our paper contributes to the literature on inventory optimization under demand substitution. This problem is challenging because product availability evolves endogenously as customers arrive sequentially and make stochastic purchase decisions. Early work by~\cite{mahajan2001stocking} shows that the resulting profit function need not be quasi-concave in the initial inventory levels. More recent work has largely proceeded along two directions. The first studies unconstrained inventory optimization with explicit inventory costs and develops policies whose optimality gaps grow sublinearly with the length of the selling horizon~\citep{honhon2010assortment,honhon2013fixed,zhang2025leveraging,mouchtaki2026joint}.

The second direction, to which our work belongs, studies capacitated inventory optimization, in which the initial inventory vector is subject to a cardinality constraint and the objective is to develop approximation algorithms with multiplicative performance guarantees~\citep{goyal2016near,aouad2018greedy,aouad2019approximation,aouad2023stability,sun2024unified}. Our work is closely related to~\cite{aouad2023stability,sun2024unified,fu2026joint}. Having already summarized the results of the latter, we focus on the former two papers, which both consider MNL choice. Specifically,~\cite{aouad2023stability} develop a $(1-\epsilon)$-approximation for stochastic show-all inventory optimization; like our fluid algorithm, their approximation scheme runs in polynomial time when the ratio $\Delta$ between the largest and smallest preference weights is $O(1)$. More recently,~\cite{sun2024unified} obtain polynomial-time guarantees of $0.474-\epsilon$ by also exploiting the CCD property of the MNL model.

\section{The Consider-Fit-Then-Choose Model and the Show-All Inventory Selection Problem}
\label{sec:model}

To start, we formalize the dynamics of the CFTC model, which captures customer purchasing behavior in settings where each product may be offered in multiple sizes and where the set of available sizes directly affects customer preferences and induced consideration sets through a notion of fit.  From here, we how CFTC model naturally extends to a broader class of applications beyond apparel in Section~\ref{subsec:other-applications}. Furthermore, in Section~\ref{subsec:model_comparison},  we explicitly discuss how the CFTC model extends the style–size choice framework of \citet{akchen2023size} in several important dimensions. We conclude this section with a formal description of the stochastic and fluid versions of our show-all inventory selection problem.

\subsection{Preliminaries and high-level overview of the CFTC model}
\label{subsec:model_set_up}

In what follows, we introduce preliminary notation and provide a high-level overview of our two-stage consider-then-choose framework. The two subsections that follow are then devoted to detailing each stage in turn, thereby fully formalizing the dynamics of the model.

\paragraph{Product preliminaries.} We consider a universe $\Ucal$ of substitutable products, where each product is represented as a \emph{base–size} pair. Formally, let $\Ncal = \{1,2,\ldots,n\}$ denote the set of $n$ substitutable bases, which capture all product attributes except size.  We also introduce a “size-less” no-purchase option, indexed by $0$, which represents the ever-present choice of leaving the store without making a purchase. Additionally, let $\Scal = \{1,\ldots,m\}$ denote the set of all possible sizes, indexed from smallest to largest according to the product-category-specific notion of size; for example, in footwear these indices correspond to numerical shoe sizes, whereas in apparel they may represent labels such as S, M, L, XL, and so on. We also introduce two “dummy” sizes, $0$ and $m+1$, to simplify the notation used to formalize the consideration set formation process. We assume that these dummy sizes are implicitly offered for every base in every assortment. For each base $i \in \Ncal$, we assume that only sizes in a subset $\Scal_i \subseteq \Scal$ can be offered as base-size combinations. Such restrictions may reflect practical considerations, including manufacturers’ product-line decisions or inventory limitations arising from upstream supply conditions. The resulting product universe is $\Ucal = \{ (i,s) \mid i \in \Ncal, s \in \Scal_i \}$, where $r_{i,s}$ denotes the revenue earned from selling a single unit of product $(i,s) \in \Ucal$.

\paragraph{Customer-type preliminaries.} We model the customer population as a collection of types $\Gcal$, where each type $g \in \Gcal$ comprises a fraction $\mu_g \ge 0$ of the population, and is characterized by a triple $(s_g, \tau_g, \ell_g) \in \Scal \times \{\uparrow, \downarrow\} \times [\ell_{\max}]$. %
The first element, $s_g \in \Scal$, denotes the customer's \emph{most preferred size}, determined by factors such as body size in apparel settings. In other applications, it can instead represent the customer's ideal value of a one-dimensional ordered attribute (e.g., the restaurant reservation time slot that best fits one's schedule; see Section~\ref{subsec:other-applications}). 
The second element, $\tau_g \in \{\uparrow, \downarrow\}$, captures the customer’s \emph{size tendency}, which describes how the customer expands her consideration set when her preferred size is unavailable. Specifically, $\tau_g = \uparrow$ denotes a willingness to consider larger sizes, whereas $\tau_g = \downarrow$ denotes a willingness to consider smaller sizes.  The final element, $\ell_g \in [\ell_{\max}]$ represents the maximum number of size levels by which the customer is willing to deviate from her preferred size $s_g$. The precise interpretation of $\tau_g$ and $\ell_g$ are provided in Section~\ref{subsec:consideration_set}, where we describe the initial consideration set formation process. Throughout the paper, we use the expressions ``a customer of type $g$'' and ``a customer of type $(s_g,\tau_g, \ell_g)$'' interchangeably.

\paragraph{Consider-then-choose framework.} The choice process of a type-$g$ customer unfolds over two sequential stages. In the first stage, the customer observes the set of available products and forms a consideration set consisting of, for each base, the available size that is “closest” to $s_g$, the customer’s ideal size. The notion of “closeness” is governed by the value of $\tau_g$: if $\tau_g = \uparrow$, the consideration set includes only products offered in sizes $s_g$ or larger, whereas if $\tau_g = \downarrow$, only products offered in sizes $s_g$ or smaller are considered. Having formed her consideration set, the customer proceeds to the second stage, in which she makes a purchase decision from among the considered products. This choice follows a standard MNL-type random utility maximization framework, with the key modification that the utility of products whose size differs from $s_g$ is discounted to reflect mismatch in fit. The formal details of these two stages are provided in the following two sections.

\subsection{Stage one: consideration set formation}\label{subsec:consideration_set}

In this section, we formalize the endogenized consideration set formation process for a type-$g$ customer. Specifically, we describe how a customer of type $(s_g,\tau_g, \ell_g)$ constructs her consideration set when faced with a particular assortment.

\paragraph{The size-based consideration set formation process.}  We begin with an informal description for the case $\tau_g = \uparrow$; the case $\tau_g = \downarrow$ is entirely analogous, with the only difference being that the customer’s consideration set is restricted to products with sizes at most $s_g$. For each base $i \in \Ncal$, if the customer’s most preferred size $s_g$ is available—i.e., if $(i,s_g) \in A$—then the corresponding product $(i,s_g)$ is included in the consideration set. If this size is not available, the customer sequentially considers larger adjacent sizes of the same base, as indicated by $\tau_g = \uparrow$. In particular, if $(i,s_g+1) \in A$, then that product is included; otherwise, the customer proceeds to $(i,s_g+2)$, and so on, up to a maximum size deviation of $\ell_g$.  If none of these sizes is available, then no product of base $i$ enters the consideration set. Conceptually, a type-$(s_g,\uparrow)$ customer exhibits the following size preference order for each base $i$:
\[
(i, s_g) \succ (i, s_g + 1) \succ (i, s_g + 2) \succ \cdots \succ (i, s_g + \ell_g),
\]
meaning that the customer first prefers $(i,s_g)$ whenever available, and otherwise considers progressively larger adjacent sizes up to size $s_g + \ell_g$. Formally, for each base $i \in \Ncal$, we define
\begin{equation*}
\label{eq:best-size-in-base-uptype}
s_g^*(i,A) =
\begin{cases}
\min\left\{ s \in [s_g, s_g+\ell_g] : (i,s) \in A \right\} 
& \text{if } \tau_g = \uparrow, \\[0.5em]
\max\left\{ s \in [s_g-\ell_g, s_g] : (i,s) \in A \right\} 
& \text{if } \tau_g = \downarrow,
\end{cases}
\end{equation*}
which represents the best-fitting available size for base $i$ among those that deviate from $s_g$ by at most $\ell_g$. Recalling that we implicitly assume sizes $0$ and $m+1$ are offered for every base $i$, the quantity $s_g^*(i,A)$ is always well defined. 
Accordingly, the consideration set of a type-$g$ customer facing assortment $A$ is defined as
\begin{equation*}
\label{eq:consideration-set}
C_g(A)
=
\left\{ (i, s_g^*(i,A)) : i \in \Ncal,\; s_g^*(i,A) \in [1,m] \right\} \cup \{0\},
\end{equation*}
and thus, if $s_g^*(i,A) \in \{0,m+1\}$, base $i$ is not considered by the customer. 
By construction, observe that at most one size of any base $i \in \Ncal$ can appear in the consideration set, and that customers always consider the no-purchase option.

\begin{example}[Consideration set]
Suppose $\Ncal=\{1,2,3\}$ and consider a customer of type
$(s_g,\tau_g,\ell_g) = (4,\uparrow,2)$. Let
\[
A=\{(1,4),(1,5),(2,6),(3,2),(3,3)\}.
\]
For base $1$, the preferred size is available, so $s_g^*(1,A)=4$. For base $2$, the preferred size and the next larger size are unavailable, but the second larger size is available, so $s_g^*(2,A)=6$. For base $3$, no acceptable size within two levels above the preferred size is available, so base $3$ is excluded. Thus, $
C_g(A)=\{(1,4),(2,6),0\}$.
\end{example}

\subsection{Stage two: purchase decision}\label{subsec:MNL_purchase}

Once a type-$g$ customer has formed her consideration set $C_g(A)$, she proceeds to an MNL-inspired purchase decision. Accordingly, we first formalize the random utility maximization (RUM) specification that governs customer choice and then explicitly derive closed-form expressions for the resulting choice probabilities.

\paragraph{The size-based RUM framework.}  Assume that a type-$g$ customer associates a random utility
\[
U_{i,s,g} = v_{i,g} - \alpha_{|s_g-s|,g} + \xi_{i,s,g}
\]
with product $(i,s) \in \Ucal$. In this specification, $v_{i,g}$ is a base-specific deterministic component, while $\alpha_{|s_g-s|,g}$ is a size-specific deterministic disutility term that captures mismatches in fit between the offered size $s$ and the customer’s ideal size $s_g$. The term $\xi_{i,s,g} \sim \mathrm{Gumbel}(0,1)$ represents an i.i.d.\ random shock. The random utility of the no-purchase option is $U_{0,g} =  \xi_{0,g}$.  We impose the following natural assumption on the disutility parameters, which ensures that, holding all else equal, sizes closer to $s_g$ are weakly preferred to sizes further away.
\begin{assumption}
	\label{assumption:disutility_is_monotonic}
	For each customer type $g \in \Gcal$, we assume that $\alpha_{0,g} \leq \alpha_{1,g} \leq \cdots \leq \alpha_{m-1,g}$ for every $i \in \Ncal$.
\end{assumption}
It is worth noting that our algorithmic approaches will not require non-negativity of the disutility parameters, allowing the model to capture the possibility of a utility boost when a customer finds her ideal size in stock.

\paragraph{Choice probabilities.} A type-$g$ customer facing assortment $A$ will purchase product $(i,s) \in C_g(A)$ with probability
\begin{eqnarray*}
\pi_g\!\left((i,s),A\right)
&=&
\pr{ U_{i,s,g} = \max_{(j,\sigma) \in C_g(A)} U_{j,\sigma,g} } 
=\frac{ e^{v_{i,g} - \alpha_{|s-s_g|,g}} }
{ 1 + \displaystyle \sum_{(j,\sigma) \in C_g(A)} e^{v_{j,g} - \alpha_{|\sigma-s_g|,g}} },
\end{eqnarray*}
where the rational expression above follows directly from the classical MNL model~\citep{train2009discrete}.  Naturally, if $(i,s) \notin C_g(A)$, then we have that $ \pi_g\left(\left(i,g\right),A\right) =0$. For notational convenience moving forward, we define weights $w_{i,g} =e^{v_{i,g}}$ and size-adjusted weights $w_{i,s,g} = w_{i,g}\cdot \beta_{|s-s_g|,g}$ for each base $i \in \Ncal \cup \{0\}$, customer type $g \in \Gcal$, and size  $s \in \Scal$, where  $\beta_{d,g} =e^{-\alpha_{d,g}}$.
Using this notation, the purchase probability of product $(i,s) \in C_g(A)$ can be written compactly as
\begin{equation}\label{eqn:CFTC_cp}
\pi_g\left((i,s),A\right)
=
\frac{ w_{i,s,g} }
{ 1 + w_g\left(C_g(A)\right) },
\end{equation}
where $w_g(A) = \sum_{(i,s) \in A} w_{i,s,g}$ denotes the total size-adjusted weight of the products in $A$ for a customer of type $g$. Although seemingly similar, the rational expression in~\eqref{eqn:CFTC_cp} differs in several important ways from the structure of choice probabilities under the classical MNL model. First, under the standard MNL framework, it is implicitly assumed that all offered products are considered by the customer. By contrast, the CFTC model incorporates a nuanced, endogenized narrowing of the offered assortment to the subset of products that are actually considered. Second, the effective weight $w_{i,g}\cdot \beta_{|s-s_g|,g}$ that a type-$g$ customer associates with base $i$ is itself endogenized through the size-dependent disutility factor, whereas under the traditional MNL model the underlying weights associated with each product are assortment independent. Together, these two features of the CFTC model significantly complicate the operational problems studied in subsequent sections.

\subsection{Additional applications of the CFTC model beyond apparel}
\label{subsec:other-applications}

\begin{figure}
	\centering
	\includegraphics[height=6cm]{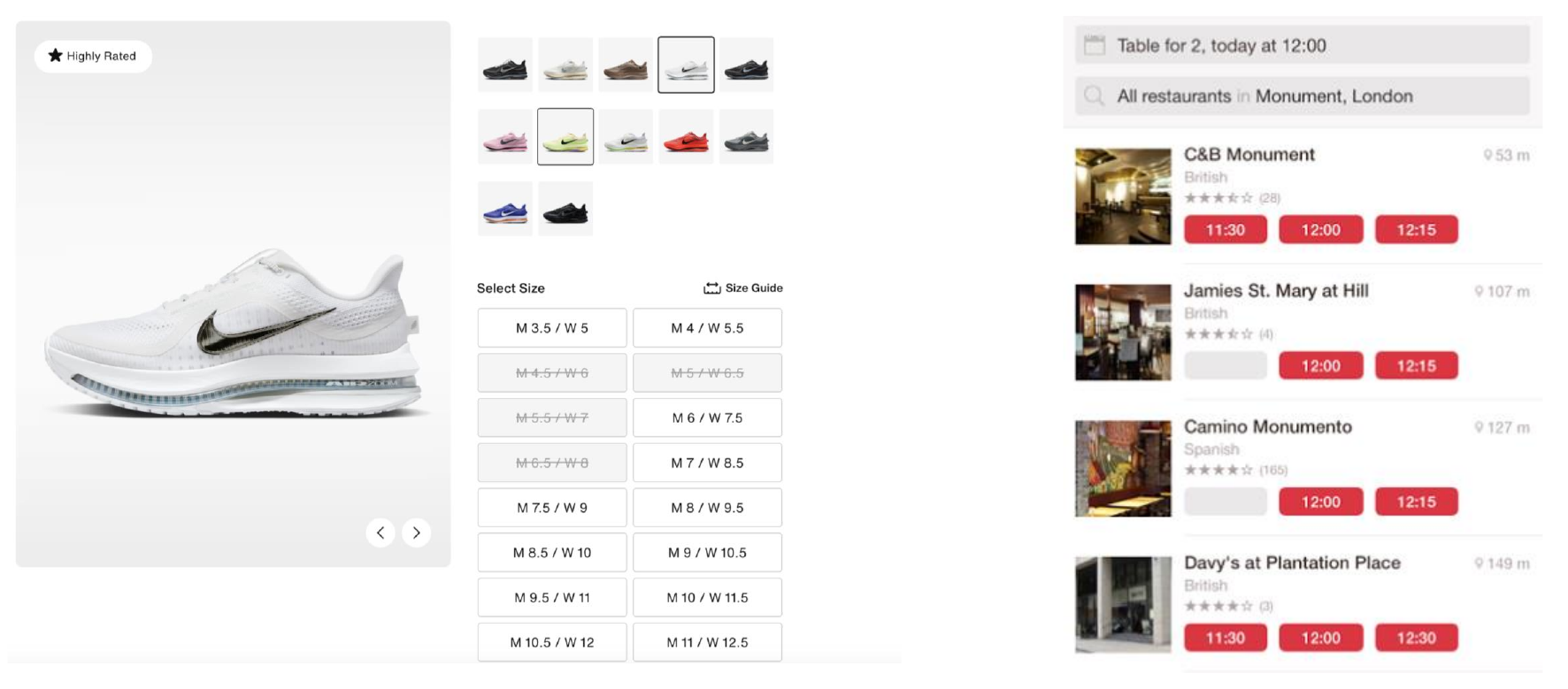}	
	\caption{Examples of products and services that exhibit the consider-fit-then-choose structure.} \label{fig:applications}
\end{figure}

Although we motivate the CFTC model by size substitution in apparel retailing, the framework applies more broadly to settings in which products are characterized by both an intrinsic product attribute and a customer-specific fit attribute. Figure~\ref{fig:applications} illustrates representative applications of this framework. Besides apparel products such as footwear and clothing, where the fit attribute corresponds to shoe or garment size, a similar structure arises on restaurant reservation platforms such as OpenTable. Customers typically have an ideal reservation time but may consider nearby time slots when their preferred time is unavailable. Among the restaurants offering acceptable time slots, they then choose based on attributes such as cuisine, quality, price, and location. 

The framework also applies to the scheduling of fitness classes. A customer may prefer a particular class time but be willing to attend a nearby session when that time is full or unavailable. Among classes offered at acceptable times, the customer may then choose based on the instructor, class format, intensity, studio location, or price. From the fitness provider's perspective, capacity decisions across neighboring class times therefore interact through time-based substitution, much as inventory decisions across neighboring sizes interact in apparel retailing.

More broadly, the fit attribute may represent any one-dimensional characteristic for which customers have an ideal value but are willing to consider nearby alternatives, such as the appointment time for a doctor visit or salon visit, the departure time of a flight or train, or a delivery window.

\subsection{Comparison to the style-size model of \cite{akchen2023size}.} \label{subsec:model_comparison}

The CFTC model extends the style–size choice framework of \citet{akchen2023size} in three distinct ways, which we summarize below. Collectively, these extensions broaden the applicability of the framework, allowing it to capture choice behavior in a wider range of settings beyond apparel and fashion. At the same time, these extensions render the algorithmic insights of \citet{akchen2023size} inapplicable to the classical assortment and inventory selection problems we later study.

\paragraph{Extended consideration set flexibility.} First, the model of \citet{akchen2023size} assumes that $\ell_g = 1$ for all customer types, so that customers who do not find their ideal size expand their consideration sets to include only adjacent sizes. This assumption is natural for apparel and footwear and is also adopted in empirical work such as \citet{li2022estimating}. However, the adjacent-substitution assumption may not be appropriate in all applications. For example, a customer who cannot obtain her preferred restaurant reservation time may be willing to consider several later time slots rather than only the immediately adjacent ones. 

\paragraph{General revenues and weights.} Second, \citet{akchen2023size} assume uniform prices across sizes within each base, i.e., $r_{i,s}=r_i$, and uniform base preference weights across customer types, i.e., $w_{i,g}=w_i$. Both restrictions are natural in apparel settings but need not hold more generally. For example, different sizes of the same mattress model are typically priced differently, while even in apparel and footwear, certain brands may appeal differently to customers with different size preferences.

\paragraph{General size availabilities.} Finally, \citet{akchen2023size} assume that all sizes are available for every base; formally, $\Scal_i = \Scal$ for all $i \in \Ncal$. Together with uniform prices and preference weights, this assumption yields an optimal assortment that offers the same set of bases at every size. The problem of selecting which bases to offer then reduces to a standard unconstrained assortment optimization problem under an MNL model, which is well known to admit an optimal revenue-ordered assortment. This structure aligns well with certain retail planning practices, in which assortments are managed at the style level by aggregating all sizes of the same base. In practice, however, size ranges often vary across bases; for example, some clothing styles may be offered only in S--XL, whereas others span a wider range of sizes. Allowing $\Scal_i$ to vary across bases eliminates this structure and substantially complicates the assortment problem.

\subsection{The show-all inventory selection problem}

In this section, we consider the fundamental problem of determining initial stocking levels for a collection of products that are subsequently consumed over a finite selling horizon. We study the problem of choosing starting inventory levels for each product $(i,s) \in \Ucal$, subject to an upper bound $\Ccal \in \mathbb{Z}_+$ on the total number of units that can be stocked. The selling horizon consists of~$T$ periods, each of which welcomes the arrival of a single customer who selects among all in-stock products according to a  CFTC model. Let
$
\Fcal = \{c \in \mathbb{Z}^{n \times m}_+ : \|c\|_1 \leq  \Ccal\}
$
be the corresponding set of feasible starting inventory vectors.

\paragraph{The stochastic setting.} Let $\rev(c)$ denote the random revenue earned over the $T$ periods when starting from an initial inventory vector $c \in \Fcal$. We consider the following inventory optimization
\begin{align}
	\label{eqn:Show_all}
	\tag{SA}
    \max_{c \in \Fcal} \ex{\Rcal(c)}, 
\end{align}
where the expectation is taken with respect to the random demand induced by the CFTC model. The label \ref{eqn:Show_all} stands for ``show-all.''

\paragraph{The fluid setting.} It is convenient to represent the selling horizon as the continuous interval $[0,T]$, and we refer to any $\tau \in [0,T]$ as a \emph{moment}. Under inventory vector $c \in \Fcal$, the first displayed assortment is
\[
A^{(1)} = \{(i,s) \in \Ucal : c_{i,s} > 0\},
\]
which is offered continuously until the first stockout occurs. Specifically, each product's inventory will be consumed at a rate equal to its choice probability, and thus the first product to stock out is
\[
(i^{(1)}, s^{(1)}) 
=
\argmin_{(i,s) \in A^{(1)}} 
\frac{c_{i,s}}{\pi((i,s),A^{(1)})},
\]
which occurs at moment
\[
\tau^{(1)} 
=
\min_{(i,s) \in A^{(1)}} 
\frac{c_{i,s}}{\pi((i,s),A^{(1)})}.
\]
After time $\tau^{(1)}$, the displayed assortment updates to
\[
A^{(2)} = A^{(1)} \setminus \{(i^{(1)}, s^{(1)})\},
\]
and the process repeats until the end of the horizon or until all products stock out.  We use $x_{(i,s)}(c,\tau)$ to denote the fluid sales of product $(i,s)$ up to moment $\tau$, when the initial stocking levels of each product are given by the vector $c$. As a shorthand, we write $x_{(i,s)}(c) = x_{(i,s)}(c,T)$ to denote the total fluid sales of product $(i,s)$ over the entire selling horizon and, abusing notation, let $\Rcal(c) = \sum_{(i,s) \in \Ucal} r_i \cdot x_{(i,s)}(c)$ denote the total accrued fluid revenue. Accordingly, the fluid version of our problem of interest can be concisely written as
\begin{equation}
    \label{eqn:Show_all_Fluid}
    \tag{SA-FLUID}
    \max_{c \in \Fcal} \Rcal(c).
\end{equation}
We use $\opt$ to denote the optimal objective value of either \ref{eqn:Show_all} or \ref{eqn:Show_all_Fluid}, with its meaning clear from the context.

 We note that even though demand is deterministic in our fluid setting,~\ref{eqn:Show_all} remains highly non-trivial due to the complex form of $\Rcal(c)$, which makes it difficult to extract preliminary insights into what profitable starting inventories should look like.

\section{Assortment Optimization under the CFTC Model}
\label{sec:assortment_optimization}

In this section, we tackle the cardinality-constrained assortment optimization problem under the CFTC model. In this now-classic framework, a retailer seeks to identify the revenue-maximizing subset of products to make available for purchase when the offered assortment can include at most $C \in \mathbb{Z}$ products. In what follows, we formally define this problem and establish that even its most basic variant is NP-hard. We then state our main algorithmic result: an approximation scheme that returns an $\epsilon$-optimal assortment in strongly polynomial time when $\ell_{\max} = O(1)$. After presenting this result as a formal theorem, we outline the two key steps underlying its proof. Appendix~\ref{app-sec:AO} subsequently provides the detailed analysis for each of these steps. We note again that our approximation scheme implies a $(0.272 -\eps)$-guarantee for \ref{eqn:Show_all} via the blackbox approach of~\cite{fu2026joint}.

\subsection{Problem formulation and results}\label{subsec:formulation_hardness}

Under assortment $A \subseteq \Ucal$, let
\[
R_g(A) = \sum_{(i,s) \in C_g(A)} r_{i,s}\cdot \pi_g\bigl((i,s),A\bigr)
\]
denote the expected revenue generated by a type-$g$ arrival, and let $
\Rcal(A) = \sum_{g \in \Gcal} \mu_g \cdot R_g(A)$ denote the total expected revenue obtained from assortment $A$. The assortment optimization problem under the CFTC model can therefore be concisely written as
\begin{equation}
\tag{AO-CFTC}
\label{eqn:AO_CFTC}
\max_{\substack{A \subseteq \Ucal: \\ |A| \leq C}} \, \Rcal(A).
\end{equation}

\paragraph{Hardness.}  The following theorem establishes that~\ref{eqn:AO_CFTC} is NP-Hard when each base comes in at most two size. Its proof is provided in Appendix~\ref{appendix-subsec:proof_of_AO_hardness}.
\begin{theorem}
    \label{thm:AO_hardness}
    The unconstrained variant of \ref{eqn:AO_CFTC} is NP-Hard, even when $m=2$.
\end{theorem}
As noted earlier, our model generalizes the framework of \citet{akchen2023size} along three key dimensions: (i) allowing the maximum deviation from the ideal size, $\ell_g$, to exceed one; (ii) permitting product revenues  to vary across sizes and the weights to vary by customer type; and (iii) accommodating heterogeneity in the size-availability sets $\Scal_i$ across bases. In proving Theorem~\ref{thm:AO_hardness}, we relax only the third of these assumptions relative to \citet{akchen2023size}. Remarkably, this single relaxation is sufficient to render the problem NP-hard, whereas under the original model it is solvable in polynomial time and  admits the simple structural characterization of the optimal assortment described in Section~\ref{subsec:model_comparison}.

\paragraph{Main result.} Our main algorithmic result for~\ref{eqn:AO_CFTC} is stated in the following theorem.
\begin{theorem}
    \label{thm:main_theorem_assort}
    For any $\eps>0$, \ref{eqn:AO_CFTC} can be approximated within factor $1-\eps$ of optimal.  The running time of this algorithm is $O((\frac{mn\ell_{\max}}{\eps})^{O(\frac{\ell_{\max}^2}{\eps})})$.
\end{theorem}
Observe that when $\ell_{\max} = O(1)$, the running time of our approach is strongly polynomial in the input size, and thus the method constitutes a polynomial-time approximation scheme (PTAS). Below, we outline the two key steps required to prove Theorem~\ref{thm:main_theorem_assort}, which are developed in full in Appendix~\ref{app-sec:AO}. 

\paragraph{Step 1: The size-skipping dynamic program (Appendix~\ref{subsec:step_1_AO}).} In the first step, we assume oracle access to an $\alpha$-approximate solution to~\ref{eqn:AO_CFTC}, for some $\alpha \in [0,1]$, which can be applied only to instances in which $m = O\left(\frac{\ell_{\max}}{\epsilon}\right)$. As discussed shortly, in the second step we show that such an oracle indeed exists with $\alpha = 1-\epsilon$. Given access to this oracle, we reduce an arbitrary instance of~\ref{eqn:AO_CFTC} to a collection of independent sub-instances, each involving $O\left(\frac{\ell_{\max}}{\epsilon}\right)$ sizes that form an uninterrupted interval of consecutive sizes. To accomplish this, we employ a dynamic program that carefully selects blocks of $\ell_{\max}$ consecutive sizes to skip, meaning that no products are offered at sizes within these “skipped” size blocks. Because no customer deviates more than $\ell_{\max}$ sizes from her ideal size, these skipped blocks naturally partition the product universe, ensuring that assortment decisions for products whose sizes lie between skipped blocks can be made independently.   Moreover, the dynamic program is designed so that the skipped size blocks contribute at most an $\epsilon$-fraction of the optimal expected revenue, while leaving at most $\frac{2\ell_{\max}}{\eps}$ sizes between skipped blocks, thereby enabling the oracle to be applied to each resulting sub-instance.

\paragraph{Step 2: The assortment oracle (Appendix~\ref{subsec:oracle}).} In the second step, we formalize the black-box oracle assumed in Step~1. Specifically, we provide a PTAS for instances of~\eqref{eqn:AO_CFTC} in which $m = O\left(\frac{\ell_{\max}}{\epsilon}\right)$ and $\ell_{\max} = O(1)$. This PTAS consists of several sub-steps, which are described in full detail in Appendix~\ref{app:assort_Oracle_thm}. The approach combines efficient enumeration with a dynamic program that processes the product universe base by base, selecting the optimal sizes to offer while approximately tracking the accrued expected revenue after each decision.

\section{The Even/Odd-Size Decomposition} \label{sec:even_odd}

In this section, we develop the demand-decoupling strategy used to tackle both the fluid and stochastic versions of our inventory problem. The key idea is to restrict attention to starting inventory vectors that stock exclusively even- or odd-indexed sizes. This restriction decouples demand across stocked sizes, thereby reducing the CFTC model to the mixed-NP-MNL model. We subsequently show that this special class of mixed-MNL models satisfies the CCD property, which will play a central role in our analysis.

For the remainder of the paper, we assume that customers choose according to a CFTC model with the following modifications, introduced to ensure tractability of the resulting inventory stocking problem:
\begin{itemize}
\item Most notably, we assume that $\ell_{\max}=1$, and thus, for every customer type $g \in \Gcal$, we have $\ell_g \in \{0,1 \}$. Consequently, each customer is willing to deviate by at most one size from her ideal size when forming her consideration set.

\item We assume that the base weights satisfy $w_{i,g}=w_i$ for each $i\in\Ncal$ and $g\in\Gcal$. Thus, customer type does not affect the valuation of a product base, although it may continue to affect the fit-mismatch penalty through the $\beta$-parameters. Throughout, we use $w_{\max}=\max_{i\in\Ncal}w_i$ and $w_{\min}=\min_{i\in\Ncal}w_i$ to denote the maximum and minimum base weights, respectively, and define $\Delta=\frac{w_{\max}}{w_{\min}}$.
\end{itemize}
Additionally, independent of the choice process, we assume that $r_{i,s}=r_i$, so that the same revenue is earned from a given base regardless of its size. All of our results continue to hold under the more general revenue structure considered in Section~\ref{sec:assortment_optimization}; this assumption is therefore made purely for notational convenience.

\paragraph{The even/odd-sized decomposition - stochastic setting.} Let $\Scal_{\even}$ denote the set of even-indexed sizes, and let $\Ucal_{\even}$ denote the corresponding universe of products with even-indexed sizes. We define $\Scal_{\odd}$ and $\Ucal_{\odd}$ analogously for odd-indexed sizes. In this first step, we approximate Problem~\ref{eqn:Show_all} by restricting attention to starting inventory vectors that stock exclusively either even- or odd-indexed sizes. To this end, let
\[
\Fcal_{\even}
=\left\{c \in \mathbb{Z}^{n \times m}_+ :
\sum_{(i,s) \in \Ucal_{\even}} c_{i,s} \leq \Ccal,\;
\sum_{(i,s) \in \Ucal_{\odd}} c_{i,s} = 0
\right\}
\]
denote the set of feasible starting inventory vectors that stock only even-indexed products, and let $\Fcal_{\odd}$ denote the analogous set for odd-indexed products. Moreover, define
\[
c_{\even} = \argmax_{c \in \Fcal_{\even}} \ex{\rev(c)}
\qquad\text{and}\qquad
c_{\odd} = \argmax_{c \in \Fcal_{\odd}} \ex{\rev(c)}
\]
to be the optimal starting inventory vectors restricted to even- and odd-indexed products, respectively. Using the fact that the CFTC model is a regular choice model, meaning that removing products from an assortment can only increase the purchase probabilities of the products that remain (for further discussion, see Appendix~\ref{appendix-sec:regularity_property}), it is straightforward to verify that
$
\max\left\{\ex{\rev(c_{\even})},\ex{\rev(c_{\odd})}\right\} \geq \frac{\opt}{2}.
$
 Thus, a $\frac{1}{2}$-approximation can be achieved provided that the even- and odd-restricted problems can be solved optimally. Moving forward, we assume without loss of generality that
$
\ex{\rev(c_{\even})} \geq \ex{\rev(c_{\odd})},
$
and therefore focus on the problem
\begin{equation}
\label{eqn:opt_inv_even}
\tag{SA-EVEN}
\max_{c \in \Fcal_{\even}} \ex{\rev(c)},
\end{equation}

\paragraph{The even/odd-sized decomposition - fluid setting.} For the fluid setting, we replace the expected revenue with the fluid revenue, noting that via identical logic, the optimal solution to the even-size-based decomposing also provides a  $\frac{1}{2}$-approximation to the original problem.  Formally, we consider the problem
\begin{equation}
\label{eqn:opt_inv_even_fluid}
\tag{SA-FLUID-EVEN}
\max_{c \in \Fcal_{\even}} \rev(c).
\end{equation}
We will similarly refer to the optimal objective values of both problems as $\opt_{\even}$, with the intended problem clear from the context.

In both settings,  a key benefit of this decomposition is that inventory decisions across the retained even-indexed sizes are coupled only through the overarching capacity constraint. Specifically, because $\ell_{\max}=1$, stockouts at one stocked size cannot affect demand at any other stocked size: each retained size is separated from the next by a size that is not stocked, preventing cross-size substitution. Consequently, demand decouples across the even-indexed sizes. As formalized next, the result of this decoupling is that the CFTC reduces to  a special class of mixed-MNL model.

\paragraph{A special class of mixed-MNL model.} By stocking only even-indexed products, we are guaranteed that at any moment, the displayed assortment satisfies $A \subseteq \Ucal_{\even}$. Furthermore, since $\ell_{\max}=1$, the structure of the consideration sets simplifies considerably. For any type $g \in \Gcal$ such that $s_g \in \Scal_{\even}$, we have
$
C_g(A) = A(s_g),
$
where for any size $s \in \Scal$, we define
$
A(s) = \{(i,\sigma) \in A : \sigma = s\}
$
to be the subset of products in $A$ of size $s$. Alternatively, if $s_g \in \Scal_{\odd}$, then
\[
C_g(A) =
\begin{cases}
    A(s_g-1) & \text{if } \tau_g = \downarrow,\\
    A(s_g+1) & \text{if } \tau_g = \uparrow,
\end{cases}
\]
so that each customer type's consideration set consists of products at a single even size. Building off this notion, for each $s \in \Scal_{\even}$, define
$
\Gcal(s) = \{g \in \Gcal : C_g(A) = A(s)\}
$
to be the set of customer types whose consideration set is $A(s)$. For any $(i,s) \in A \subseteq \Ucal_{\even}$, we then have
\begin{equation*}
\pi((i,s),A) =
\sum_{g \in \Gcal(s)} \lambda_g \cdot \frac{w_{i,s,g}}{1 + w_g(A(s))} = 
\sum_{g \in \Gcal(s)} \lambda_g \cdot \frac{w_i}{w_{0,g} + w(A(s))},
\end{equation*}
where $w_{0,g} = \frac{1}{\beta_{|s-s_g|,g}}$ denotes the type-$g$ no-purchase weight and
$
w(A(s)) = \sum_{(i,s) \in A(s)} w_i
$
is the total base weight of assortment $A(s)$. The above expression corresponds to a mixed-MNL model in which customer types differ only in the weight of the no-purchase option. We refer to this special class of models as the \emph{mixed-no-purchase MNL} (mixed-NP-MNL) model.
    
\paragraph{A new CCD choice model.} To conclude this section, we establish that the mixed-NP-MNL choice model is a convex chain decomposable (CCD) choice model. This structural property will prove to be the critical ingredient in both of the algorithms we develop. As the notion of a CCD model was only recently introduced in~\cite{goyal2023pricing}, we restate its definition in a form tailored to our setting before establishing that the mixed-NP-MNL model is indeed CCD. In what follows, we temporarily revert to a standard choice-based setting with products indexed by $\Ncal = \{1,\ldots,n\}$, and we use $\pi(i,A)$ to denote the choice probability of product $i$ under assortment $A \subseteq \Ncal$.
\begin{definition}[CCD choice model]
    Consider an arbitrary scaled convex combination of assortments, specified by weights $\{h(A)\}_{A \subseteq \Ncal}$ satisfying
$
\sum_{A \subseteq \Ncal} h(A) = T.
$
Let $\phi \in \mathbb{R}_+^n$ denote the resulting vector of purchase rates, defined by
$
\phi_i = \sum_{A \subseteq \Ncal} h(A)\cdot \pi(i,A).
$
For each $i \in \Ncal$, let $\tau_i(\phi)$ denote the stockout time of product $i$ under the show-all policy starting from inventory vector $\phi$ over an infinite selling horizon. A choice model is CCD if, for every such vector $\phi$, we have $\tau_i(\phi) \in [0,T]$ for all $i \in \Ncal$; that is, every product stocks out by time $T$.
\end{definition}

    The following lemma shows that the mixed-NP-MNL model is CCD. To the best of our knowledge, it was previously known only that the classical MNL and Markov chain choice models satisfy the CCD property. Thus, the inclusion of the mixed-NP-MNL model is a genuinely new result. Moreover, to cement this notion, in Appendix~\ref{app:MC_not_NP_MNL}, we show that there are instances of the mixed-NP-MNL model that cannot be represented as Markov chain choice models.
\begin{lemma}
    \label{lem:CCD}
    The mixed-NP-MNL model is a CCD choice model.
\end{lemma}
We prove the lemma above in Appendix~\ref{app:proof_CCD} by invoking the necessary and sufficient condition for CCD established in~\cite{goyal2023pricing}. Importantly, even with this characterization in hand, verifying whether a given choice model satisfies the condition remains highly non-trivial.

\section{The Fluid Show-All Inventory Stocking Problem}\label{sec:fluid_setting}

In this section, we consider~\ref{eqn:Show_all_Fluid}, wherein at all moments throughout the selling horizon, the inventory of each products is consumed at an instantaneous rates equal to its induced choice probability. Our main result for this setting in formalized in the following theorem.  
\begin{theorem}
    \label{thm:one_half_fluid}
    For any $\eps>0$, \ref{eqn:Show_all_Fluid} can be approximated within a factor of $\frac{1}{2}-\eps$ of optimal. The running time of this algorithm is $O(m\cdot (\frac{n\Ccal \Delta}{\eps^2})^{O(\frac{1}{\eps^2}\cdot \log \Delta)})$, where we recall that $\Delta = \frac{w_{\max}}{w_{\min}}$.
\end{theorem}
As discussed in Section~\ref{sec:even_odd}, the factor-$\frac{1}{2}$ loss in Theorem~\ref{thm:one_half_fluid} arises from restricting attention to even sizes; for the resulting problem~\ref{eqn:opt_inv_even_fluid}, our algorithm achieves a $(1-\eps)$-approximation. 
Moreover, a close inspection of the running time reveals that our algorithm runs in polynomial time when $\Delta = O(1)$, and in quasi-polynomial time when $\Delta = O(\mathrm{poly}(n,m,\frac{1}{\eps}))$. We note that this same dependence on $\Delta$ appears in several related works, including~\cite{aouad2023stability} and~\cite{derakhshan2022product}. Moreover, in the experiments presented in Section~\ref{sec:numerical_experiments}, where we fit the CFTC model to real shoe sales data, we find that $\Delta$ ranges from 1.99 to 3.98, across all fitted instances, suggesting that this assumption is reasonable in practice.

\subsection{Technical Overview}\label{subsec:technical_overview_opt_inv}

Our approach to solving~\ref{eqn:opt_inv_even_fluid} involves several interdependent components that must be carefully aligned. For this reason, we begin with an overview of the four key steps required to establish Theorem~\ref{thm:one_half_fluid}. Each of these steps is developed in its own subsection. 

\paragraph{Step 1: The stitching dynamic program (Section~\ref{subsec:step_1_INV}). } For~\ref{eqn:opt_inv_even_fluid}, the stocking decisions across the even sizes become effectively decoupled, aside from the global capacity constraint.  This observation leads to a stitching dynamic program that proceeds sequentially over the retained sizes. For each size, the program allocates a portion of the total capacity $\Ccal$, then optimally distributes this allocation across the bases at that size. As noted above, this base-specific allocation decision reduces to a special case of~\ref{eqn:Show_all_Fluid} in which only a single size is stocked and where customer choose according to the mixed-NP-MNL model. In the next three steps, we develop an approximation scheme for this special case.

\paragraph{Step 2: Rounded weights and the proxy problem (Section~\ref{subsec:step_2_INV}).} In the second step, we begin by partitioning the products into $O(\log \Delta)$ so-called \emph{weight classes}, such that the base weights of any two products within the same class differ by at most a $(1+\eps^2)$-factor. We then formulate a proxy version of~\ref{eqn:Show_all_Fluid} in which the base weight of each product is rounded down to the minimum weight within its corresponding class. The main result of this step shows that this rounding incurs only a negligible loss: specifically, the optimal initial inventory vector for the proxy problem achieves at least $(1-O(\eps))\cdot \opt$ in the original problem. Consequently, to establish Theorem~\ref{thm:one_half_fluid}, it suffices to focus on solving the proxy problem. Steps 3 and 4 are devoted to this goal, where we show that the proxy problem admits a $(1-O(\eps))$-approximation scheme.

\paragraph{Step 3: Stocking groups (Section~\ref{subsec:step_3_INV}).} To begin, for each weight class, we introduce the notion of its associated \emph{stocking groups}, which correspond to subsets of products within the class that are initially stocked at each level between $1$ and $\frac{1}{\eps}$ under $\hat{c}$, where $\hat{c}$ denotes the optimal solution to the proxy problem introduced in Step~2. Because, within each weight class, products differ only in their revenues after rounding the base weights, it is straightforward to show that optimal stocking levels are non-decreasing in revenue. Leveraging this structural property, we can efficiently guess via complete enumeration the composition of each stocking group, along with a final catch-all group that contains all so-called \emph{heavy} products whose initial stocking levels exceed $\frac{1}{\eps}$ under $\hat{c}$.

\paragraph{Step 4: The fluid linear program (Section~\ref{subsec:step_4_INV}).} After guessing the stocking groups for each weight class in Step~3, we fully recover the products that are stocked at levels of at most $\frac{1}{\eps}$ under $\hat{c}$, and we also identify the set of products whose initial inventory exceeds $\frac{1}{\eps}$. For this latter class of heavy products, however, we do not know their exact stocking levels under $\hat{c}$; we only know that they exceed $\frac{1}{\eps}$. To determine the initial inventory levels for these heavy products, we formulate a carefully designed fluid linear program (LP) that treats their inventory levels as decision variables. The key challenge is to ensure that this LP accurately captures the assortment dynamics induced by the show-all policy, under which products are removed from the assortment only upon stockout. To address this, we exploit the fact that the special class of mixed-MNL models under consideration satisfies the \emph{convex chain decomposition} (CCD) property.  This result, as well as an overview of the CCD property is established in Section~\ref{sec:even_odd}, which serves as precursor to these four algorithmic steps.  In a nutshell, this property implies that the choice probabilities generated by any convex combination of assortments can be equivalently represented by a nested sequence of assortments. Leveraging the CCD structure, we are able to align stockout events with product removals in this nested sequence within the optimal solution of our fluid LP.

\subsection{Step 1: The stitching dynamic program}\label{subsec:step_1_INV}

Our dynamic program will make use of the following oracle, which is given shape in Steps 2-4.

\paragraph{The black-box oracle.} Within our dynamic program, we make stocking decisions sequentially across the products available at each even size. Specifically, for each even size $s \in \Scal_{\even}$ and integer $k \in [\Ccal]_0$, we consider the inventory problem restricted to the products
$
\Ucal(s) = \{(i,\sigma) \in \Ucal : \sigma = s\},
$
which denotes the collection of base products offered in size $s$. The corresponding optimization problem is
\begin{equation}
    \label{eqn:opt_inv_oracle}
    \tag{SA-FLUID$(s,k)$}
    \opt_{s,k} = \max_{c \in \Fcal(s,k)} \rev(c),
\end{equation}
where
\[
\Fcal(s,k)
=
\left\{
c \in \mathbb{Z}^{n \times m}_+ :
\sum_{(i,s) \in \Ucal(s)} c_{i,s} \leq k,\;
\sum_{(i,\sigma) \in \Ucal \setminus \Ucal(s)} c_{i,\sigma} = 0
\right\}
\]
denotes the set of starting inventory vectors that allocate at most $k$ units exclusively to products of size $s$. In this first step, we assume oracle access to an inventory vector $\hat{c}_{s,k} \in \Fcal(s,k)$ whose fluid revenue satisfies
$
\rev(\hat{c}_{s,k}) \geq \alpha \cdot \opt_{s,k}
$
for some $\alpha \in (0,1]$. In Steps 2--4, we construct such an oracle with $\alpha = 1-O(\eps)$.

\paragraph{State description.} Each state $(s,\kappa)$ of our dynamic program consists of two components:
\begin{itemize}
    \item An even size $s \in \Scal_{\even}$ for which stocking decisions for products in $\Ucal(s)$ are to be made.
    \item An integer $\kappa \in [\Ccal]_0$ denoting the number of units of the initial capacity $\Ccal$ that remain to be allocated; ; thus, $\Ccal-\kappa$ units have been stocked across sizes $2,4,\ldots,s-2$.
\end{itemize}
Consequently, the dynamic program has $O(m\Ccal)$ states in total.

\paragraph{Value function.} For an oracle with $\alpha=1$, the value function $\Vcal(s,\kappa)$ represents the maximum fluid revenue that can be accrued from products with even sizes $s, s+2, \ldots, m$ by allocating at most $\kappa$ units across these sizes. Formally, we define
\begin{equation}
\label{eqn:DP_even}
\Vcal(s,\kappa)
=
\max_{k \in [\kappa]_0}
\left\{
\rev(\hat{c} 
_{s,k})
+
\Vcal(s+2, \kappa-k)
\right\},
\end{equation}
The base cases of the dynamic program are given by
\[
\Vcal(m+2, \kappa) =
\begin{cases}
0 & \text{if } \kappa \geq 0,\\
-\infty & \text{otherwise},
\end{cases}
\]
which ensures that the total allocated inventory does not exceed the capacity $\Ccal$.

\paragraph{Analysis.} Let
\[
(2,\Ccal)
\quad \xrightarrow{(k^{(2)}, c^{(2)})} \quad
(4,\Ccal-k^{(2)})
\quad \xrightarrow{(k^{(4)}, c^{(4)})} \quad
\cdots
\quad \xrightarrow{(k^{(m)}, c^{(m)})} \quad
(m,\Ccal-\sum_{s \in \Scal_{\even}} k^{(s)})
\]
denote the sequence of states and actions obtained by following the dynamic program in~\eqref{eqn:DP_even} starting from the initial state $(2,\Ccal)$, where for each $s \in \Scal_{\even}$, the vector $c^{(s)} = \hat{c}_{s,k^{(s)}}$ is the inventory decision returned by the oracle for size $s$ and capacity level $k^{(s)}$. Let
$
\hat{c}_{\even} = \sum_{s \in \Scal_{\even}} c^{(s)}
$
denote the starting inventory vector obtained by aggregating the size-specific stocking decisions. It is immediate that $\hat{c}_{\even} \in \Fcal_{\even}$, and thus the following claim, whose proof appears in Appendix~\ref{app:proof_dp_even_works}, establishes that $\hat{c}_{\even}$ is $\alpha$-optimal for~\ref{eqn:opt_inv_even_fluid}.
\begin{claim}
    \label{claim:dp_even_works}
    $
     \rev(\hat{c}_{\even}) \geq \alpha \cdot \opt_{\even}.
    $
\end{claim}
To prove the claim above, we exploit the assumption that $\ell_{\max}=1$ together with the fact that $\hat{c}_{\even}$ skips every other size. As a result, demand and revenue decompose across the even sizes that are stocked. Therefore, the remaining---and highly non-trivial---question in proving Theorem~\ref{thm:one_half_fluid} is whether a near-optimal oracle for \ref{eqn:opt_inv_oracle} exists. We now turn to this subproblem and, in particular, to its connection with the mixed-NP-MNL model.

\paragraph{The mixed-NP-MNL subproblem.}  The maximization in~\eqref{eqn:DP_even} can be solved via complete enumeration, and thus our attention turns to giving shape to the black-box oracle. Since we restrict attention to a single size, it follows from the discussion in Section~\ref{sec:even_odd} that the CFTC model reduces to an instance of the mixed-NP-MNL model. Accordingly, in the remaining steps, we develop an approximation scheme for~\ref{eqn:Show_all_Fluid} under a general mixed-NP-MNL model, which generalizes~\ref{eqn:opt_inv_oracle}. To streamline the exposition, we suppress the size index and focus on a mixed-NP-MNL model defined on $n$ products with customer types $\Gcal$, arrival probabilities $\{\lambda_g\}_{g \in \Gcal}$, and no-purchase weights $\{w_{0,g}\}_{g \in \Gcal}$. For any assortment $A \subseteq \Ncal$, we denote the choice probability of product $i \in A$ by $\pi(i,A)$. Accordingly, for the remainder of this section, we focus on the problem
\begin{equation}
    \label{eqn:opt_inv_np}
    \tag{SA-FLUID-NP}
    \opt_{w} = \max_{c \in \Fcal} \rev_w(c),
\end{equation}
where, with a slight abuse of notation, we recycle
$
\Fcal = \{c \in \mathbb{Z}^n_+ : \|c\|_1 \leq \Ccal\},
$
and let $\rev_w(c)$ denote the fluid revenue of the show-all process under the mixed-NP-MNL model with product weights $w = (w_1,\ldots,w_n)$, which correspond to the base weights in the original CFTC model. We use $c^*$ to denote an optimal solution to~\ref{eqn:opt_inv_np}.

\subsection{Step 2: Rounded weights and the proxy problem}\label{subsec:step_2_INV}

This second step begins by partitioning the products into so-called weight classes, which in turn leads to the formulation of a proxy version of~\ref{eqn:opt_inv_np}. In this proxy problem, the weights of all products within the same weight class are rounded down to a common value, which endows the problem with critical structure that will be exploited in Step 3.

\paragraph{Weight classes.} In what follows, we partition the products into 
$Q = O\!\left(\frac{1}{\eps^2}\log \Delta\right)$ weight classes, where 
$\Delta = \frac{w_{\max}}{w_{\min}}$ denotes the ratio between the maximum and minimum weights. Specifically, for each $q \in [Q]$, define
\[
\Ncal_q = \left\{ i \in \Ncal : 
w_i \in \left[ w_{\min}\cdot (1+\eps^2)^{q-1},\; w_{\min}\cdot (1+\eps^2)^q \right) \right\},
\]
so that the weights of any two products within the same class differ by at most a factor of $1+\eps^2$. Moreover, for each class $q \in [Q]$, we relabel the products as $\Ncal_q = \{1,\ldots,n_q\}$, where $n_q = |\Ncal_q|$, and index them in non-increasing order of revenue. Although each class contains a product indexed by $1$, this relabeling is local to each class and does not introduce ambiguity in the subsequent analysis. Finally, as a shorthand, we will often refer to $\Ncal_q$ as the ``class-$q$'' products.

\paragraph{The rounded proxy problem.}  From here, we construct a ``rounded'' instance in which the weight of each product $i \in \Ncal_q$, $q \in [Q]$, is rounded down to
$
\hat{w}_i = w_{\min}\cdot (1+\eps^2)^{q-1},
$
and for $g \in \Gcal$, we set $\hat{w}_{0,g} = \frac{1}{1+\eps^2}\cdot w_{0,g}$. Let $\hat{\pi}(i,A)$, $\hat{x}_i(c)$ and $\rev_{\hat{w}}(c)$ denote the choice probabilities, fluid sales and total revenue under the rounded weights $\hat{w}$, respectively. This leads to the following proxy problem:
\begin{equation}
    \label{eqn:opt_inv_rounded}
    \tag{SA-FLUID-PROXY}
      \opt_{\hat{w}} = \max_{c \in \Fcal} \rev_{\hat{w}}(c),
\end{equation}
and let $\hat{c}$ denote an optimal solution. Among all optimal solutions, we select $\hat{c}$ to be one that minimizes $\|\hat{c}\|_1$. This choice implies that the induced fluid sales satisfy
$
\hat{x}_i(\hat{c}) \in [\hat{c}_i - 1, \hat{c}_i]
$
for each $i \in \Ncal$. The following lemma, whose proof appears in Appendix~\ref{app:proof_proxy_near_op} shows that $O(\eps)$-optimal solutions to the proxy problem imply $O(\eps)$-optimal solutions to the original problem.
\begin{lemma}
    \label{lem:proxy_near_op}
    For any  $c \in \Fcal$ and $\eps \in (0,\frac{1}{10}]$ such that $\rev_{\hat{w}}(c) \geq (1-5\eps)\cdot \opt_{\hat{w}}$, we have that $$\rev_w(c) \geq (1 - 15\eps)\cdot \opt_{w}.$$
\end{lemma}
To prove the lemma, we show that $\rev_{\hat{w}}(c)$ provides a good approximation to $\rev_w(c)$ for inventory vectors that are already $O(1)$-optimal for~\ref{eqn:opt_inv_np}. Notably, such stability does not hold uniformly over all feasible inventory vectors. Establishing this approximation is non-trivial, as the transition from $w$ to $\hat{w}$ perturbs choice probabilities in a non-uniform manner, and consequently the fluid sales of individual products may either increase or decrease under rounding.

\subsection{Step 3: Stocking groups}\label{subsec:step_3_INV}

Given Lemma~\ref{lem:proxy_near_op}, our focus in the remaining two steps is to develop an approximation scheme for~\ref{eqn:opt_inv_rounded}. To this end, we introduce the notion of \emph{stocking groups}, which partition the products within each weight class according to their initial stocking levels under $\hat{c}$, the optimal solution to~\ref{eqn:opt_inv_rounded}. We then establish a simple structural property of these stocking groups that allows them to be recovered exactly via guessing. These groups will play a central role in the fluid LP formulated in Step~4.

\paragraph{Stocking groups.} For each stocking level $k \in [\frac{1}{\eps}]$, let $\Ncal_q(k) = \{i \in \Ncal_q:\hat{c}_i=k\}$ denote the class-$q$ products whose initial inventory is $k$ under $\hat{c}$, and let $k(q) = \{k \in [\frac{1}{\eps}]: \Ncal_q(k) \neq \emptyset\}$ denote the set of distinct stocking levels not exceeding $\frac{1}{\eps}$ used within class-$q$ under $\hat{c}$. For each $q \in [Q]$, we assume that $k(q) = \{k^{(1)}_q, \ldots, k^{(m_q)}_q\}$, where $m_q = |k(q)|$, and where lower superscripts correspond to larger stocking levels, meaning that $k^{(1)}_q > k^{(2)}_q > \cdots > k^{(m_q)}_q$. Finally, let $\Ncal_{q,H} = \{i \in \Ncal_q: \hat{c}_i > \frac{1}{\eps}\}$ denote the so-called \emph{heavy} products, whose initial stocking levels all exceed $\frac{1}{\eps}$. The following claim shows that within each weight class, the optimal stocking levels must be non-decreasing in revenue. The proof of this result appears in Appendix~\ref{app:proof_stock_non_dec_rev}.
\begin{claim}
    \label{claim:stock_non_dec_rev}
    For any $q \in [Q]$, consider distinct products $i,j \in \Ncal_q$ such that $r_i>r_j$. We have $\hat{c}_i \geq \hat{c}_j$.
\end{claim}
The above claim will prove critical in the guessing steps to come. Moreover, it provides further justification for indexing the products within each weight class in non-increasing order of revenue.

\paragraph{Heavy/light weight classes and guessing.} Let $Q_H = \{q \in [Q] : \Ncal_{q,H} \neq \emptyset\}$ denote the \emph{heavy} weight classes, namely those that contain at least one heavy product. Moreover, let $Q_L = [Q] \setminus Q_H$ denote the \emph{light} weight classes, for which every constituent product is stocked at a level of at most $\frac{1}{\eps}$ under $\hat{c}$. By Claim~\ref{claim:stock_non_dec_rev}, for any heavy class $q \in Q_H$, there exists an increasing sequence of class-$q$ product indices $i_q^{(0)}, i_q^{(1)}, \ldots, i_q^{(m_q)}$ such that $\Ncal_{q,H} = \{1,\ldots,i_q^{(0)}\}$ and, for each $\ell \in [m_q]$, we have
$
\Ncal_q(k_q^{(\ell)}) = \{i_q^{(\ell-1)}+1,\ldots,i_q^{(\ell)}\}.
$
Similarly, for any light class $q \in Q_L$, there exists an increasing sequence of product indices such that
$
\Ncal_q(k_q^{(\ell)}) = \{i_q^{(\ell-1)}+1,\ldots,i_q^{(\ell)}\}
$
for each $\ell \in [m_q]$, where in this case we take $i_q^{(0)}=0$. In this way, the stocking groups within any weight class are fully characterized by their $m_q+1$ product endpoints.

In what follows, we use guessing to recover the identity of all stocking groups, as well as the maximum stocking level in each heavy class. These guesses will play a critical role in establishing the correctness of the fluid LP formulated in Step~4.

\begin{itemize}
    \item \emph{Guessing heavy/light classes.} We begin by guessing whether each class $q \in [Q]$ is heavy or light. This can be done by enumerating over all possible subsets of $[Q]$, which requires time
    $
    O(2^Q) = O\!\left(\Delta^{\,O(\frac{1}{\eps^2})}\right).
    $

    \item \emph{Guessing stocking groups.} For each class $q \in [Q]$, we guess the identities of the endpoints $i_q^{(0)}, i_q^{(1)}, \ldots, i_q^{(m_q)}$. Since each endpoint has at most $n$ possible values and $m_q \le \frac{1}{\eps}$, this can be done by complete enumeration in time
    $
    O\!\left(n^{\,O(\frac{1}{\eps}\cdot Q)}\right)
    =
    O\!\left(n^{\,O(\frac{1}{\eps^3}\log \Delta)}\right).
    $

    \item \emph{Guessing maximum stocking levels.} For each heavy class $q \in Q_H$, we guess
    $
    \hat{c}_{q,\max} = \max_{i \in \Ncal_q} \hat{c}_i,
    $
    the maximum stocking level of any class-$q$ product under $\hat{c}$. Since $\hat{c}_{q,\max} \in [\Ccal]$, this can be done in time
    $
    O(\Ccal^Q)
    =
    O\!\left(\Ccal^{\,O(\frac{1}{\eps^2}\log \Delta)}\right).
    $
\end{itemize}

At this point, it is worth noting that if $Q_H = \emptyset$, then the first two guessing steps described above exactly recover $\hat{c}$, and the proof is complete. However, if $Q_H \neq \emptyset$, then we must still determine the initial inventories of the heavy products, knowing only that under $\hat{c}$ they lie somewhere between $\frac{1}{\eps}+1$ and $\hat{c}_{q,\max}$. The fluid LP introduced in Step~4 is designed precisely for this purpose.

\subsection{Step 4: The fluid linear program}\label{subsec:step_4_INV}

Before formulating the LP, we introduce a further partition of the light classes into those for which all constituent products stock out and those that contain at least one product that does not stock out. This leads to one final guessing step, in which we approximately guess the fluid sales of the light products that do not stock out.

\paragraph{Light stock-out products and guessing.}
Define
$
Q_{L,SO} = \{q \in Q_L : \hat{x}_i(\hat{c}) = \hat{c}_i \ \forall i \in \Ncal_q\}
$
to be the set of light classes for which every product stocks out, and define
$
Q_{L,<} = Q_L \setminus Q_{L,SO}
$
to be the set of light classes for which at least one product does not stock out. The following claim shows that if a product in a light class does not stock out, then it must belong to the highest-stocked group. The proof appears in Appendix~\ref{app:proof_stock_out_group}.
\begin{claim}
    \label{claim:stock_out_group}
    For each $q \in Q_L$ and each class-$q$ product $i > i_q^{(1)}$, we have $\hat{x}_i(\hat{c}) = \hat{c}_i$.
\end{claim}
The key implication of the claim above is that for any light class $q \in Q_{L,<}$, every product $i \in \Ncal_q(k_q^{(1)}) = \{1,\ldots,i_q^{(1)}\}$ must not stock out under $\hat{c}$ and is therefore offered over the entire horizon. Moreover, because all of these products belong to the same weight class, they share the same rounded weight, and hence each experiences the same total fluid demand. We now describe the final two guessing steps.

\begin{itemize}
    \item \emph{Guessing light stock-out classes.} For each $q \in Q_L$, we guess whether all products in class $q$ stock out. This can be done by enumerating over all subsets of $Q_L$, which requires time
    $
    O(2^Q) = O\!\left(\Delta^{\,O(\frac{1}{\eps^2})}\right).
    $
    After this step, the sets $Q_{L,SO}$ and $Q_{L,<}$ are fully determined.

    \item \emph{Guessing non-stock-out fluid sales.} For each light class $q \in Q_{L,<}$ and each product $i \in \Ncal_q(k_q^{(1)})$, we guess an estimate $\ubar{x}_i$ of $\hat{x}_i(\hat{c})$ satisfying
    $
    \ubar{x}_i \leq \hat{x}_i(\hat{c}) \leq (1+\eps)\cdot \ubar{x}_i.
    $
    To obtain such a guess, we use the fact that product $i$ is offered throughout the entire horizon, since it does not stock out. Therefore, %
    \[
    \hat{x}_i(\hat{c}) \in \left[ \sum_{g \in \Gcal} \lambda_g \cdot \frac{T w_{\min}}{ w_{0,g} +n w_{\max}},\; \sum_{g \in \Gcal} \lambda_g \cdot \frac{T w_{\max}}{ w_{0,g} +w_{\max}}\right].
    \]
    Notice that the ratio between the upper and the lower bounds is at most $n \Delta$, since $\left. \frac{w_{\max}}{ w_{0,g} + w_{\max} } \right/  \frac{w_{\min}}{ w_{0,g} + n \cdot w_{\max} } \leq n \Delta $ for each $g \in \Gcal$. It therefore suffices to enumerate over guesses of the form
    $
    \left( \sum_{g \in \Gcal}  \frac{ \lambda_g \cdot T w_{\min}}{ w_{0,g} +n w_{\max}} \right) \cdot (1+\eps)^q$, for
    $q = 0,1,\ldots,
    \left\lceil
    \log_{1+\eps}\!\left(
    n \Delta
    \right)
    \right\rceil.
    $
    We also add  $k^{(2)}_q$ to this list of guesses so that the guessed fluid sales is not below the fluid sales of any other product in the same class that stocks out. Consequently, guessing these fluid-sales estimates across all classes in $Q_{L,<}$ requires time
    $
    O\!\left((\frac{\log n\Delta}{\eps^2})^Q\right)
    =
    O\!\left((\frac{\log n\Delta}{\eps^2})^{\,O(\frac{1}{\eps^2}\log \Delta)}\right).
    $ 
\end{itemize}

\paragraph{The fluid LP.} The fluid LP that we present below is very close in spirit to the classical choice-based deterministic LP~\citep{gallego2004managing, liu2008choice}, with a few important modifications. Along this line, let $\Acal_{\mathrm{RO}}$ denote the set of within-class revenue-ordered assortments, meaning that for any assortment $A \in \Acal_{\mathrm{RO}}$ and class $q \in [Q]$, there exists a product $j \in \Ncal_q$ such that $A \cap \Ncal_q = \{1,\ldots,j\}$ corresponds to a class-$q$ revenue-ordered assortment. Our fluid LP uses decision variables $\{h(A)\}_{A \in \Acal_{\mathrm{RO}}}$, where $h(A)$ represents the fraction of time over the horizon that assortment $A$ is offered. Additionally, for each heavy product $i \in \biguplus_{q \in Q_H}\Ncal_{q,H}$, we introduce a decision variable $u_i$ to represent its fluid sales. These $u$-variables are auxiliary, but they lead to a more interpretable and compact formulation.  In the formulation below, recall that from the preceding guessing steps, we have recovered $\hat{c}_i$ for each $i \in \biguplus_{q \in [Q]\setminus Q_{L,<}} \biguplus_{k \in k(q)}\Ncal_q(k)$ and $  i \in \biguplus_{q \in Q_{L,<}} \biguplus_{ 2 \leq \ell \leq m_q } \Ncal_q( k_q^{(\ell)} ) $, as well as estimates $\ubar{x}_i$ for each $i \in \biguplus_{q \in Q_{L,<}} \Ncal_q(k_q^{(1)})$ satisfying $\ubar{x}_i \leq \hat{x}_i(\hat{c}) \leq (1+\eps)\cdot \ubar{x}_i$. 
\begin{align*}
\label{eqn:fluid_LP}
\tag{FLUID-LP}
Z^*_{\mathrm{LP}} \
= \ \max_{h,u} \quad 
& \sum_{A \in \Acal_{\mathrm{RO}}} h(A)\; \sum_{i \in A} r_i \cdot \hat{\pi}(i,A) \\[0.3em]
\text{s.t.} \quad
& (1)~~ \sum_{A \in \Acal_{\mathrm{RO}}} h(A) = T & \\[0.3em]
& (2)~~ \sum_{\substack{A \in \Acal_{\mathrm{RO}} :\\ i \in A}} h(A)\, \hat{\pi}(i,A)
=
\begin{cases}
\hat{c}_i 
& \text{if } \displaystyle i \in \biguplus_{q \in [Q] \setminus Q_{L,<}} \biguplus_{k \in k(q)} \Ncal_q(k) \\[0.5em]
\ubar{x}_i 
& \text{if }\displaystyle  i \in \biguplus_{q \in Q_{L,<}} \Ncal_q(k_q^{(1)}) \\[0.5em]
\hat{c}_i
& \text{if }\displaystyle  i \in \biguplus_{q \in Q_{L,<}} \biguplus_{ 2 \leq \ell \leq m_q } \Ncal_q( k_q^{(\ell)} ) \\[0.5em]
u_i 
& \text{if } \displaystyle i \in \biguplus_{q \in Q_H} \Ncal_{q,H}
\end{cases}
& \forall i \in \Ncal, \\[0.3em]
& (3)~~ \frac{1}{\eps}+1 \leq u_i \leq \hat{c}_{q,\max} 
& \forall i \in \biguplus_{q \in Q_H} \Ncal_{q,H}, \\[0.3em]
& (4)~~ \sum_{q \in Q_H} \sum_{i \in \Ncal_{q,H}} u_i
\;\leq\; \Ccal - \sum_{q \in [Q]} \sum_{k \in k(q)} \sum_{i \in \Ncal_q(k)} \hat{c}_i & \\[0.3em]
& (5)~~ h(A), u_i \ge 0. &
\end{align*}
Constraint~(1) ensures that $\{h(A)\}_{A \in \Acal_{\mathrm{RO}}}$ exactly spans the entire selling horizon. The constraints in~(2) and~(3) are used to align the chosen starting inventory vector with $\hat{c}$ as closely as possible. More specifically, constraint~(2) requires that the fluid sales of each non-heavy product exactly match the corresponding guessed quantity. For heavy products, constraint~(3) imposes the weaker requirement that their fluid sales lie in the interval $[\frac{1}{\eps}+1,\hat{c}_{q,\max}]$. Finally, constraint~(4) ensures that the capacity constraint is satisfied. The following lemma shows that~\ref{eqn:fluid_LP} is feasible and that its optimal objective value is at least a $(1-\eps)$-fraction of $\opt_{\hat{w}}$.
\begin{lemma}
    \label{lem:fluid_LP_lb}
    $Z^*_{\mathrm{LP}} \geq (1-\eps)\cdot \opt_{\hat{w}}.$
\end{lemma}
The proof of the lemma above appears in Appendix~\ref{app:proof_fluid_LP_lb} and it proceeds by constructing a feasible solution to \ref{eqn:fluid_LP} using $\hat{c}$ and the displayed assortment induced by the show-all policy under $\hat{c}$ that achieves an objective of at least $(1-\eps)\cdot \opt_{\hat{w}}$.

\paragraph{Remarks.} Before describing how we use~\ref{eqn:fluid_LP} to select our output initial inventory vector, it is important to note that this LP contains no explicit constraints guaranteeing that a valid show-all policy can be derived from an optimal solution. This is precisely where we exploit the CCD property of the mixed-NP-MNL model. We also note that~\ref{eqn:fluid_LP} has only $|\Acal_{\mathrm{RO}}| = O(n^Q) = O\bigl(n^{O(\frac{1}{\eps^2}\cdot \log \Delta)}\bigr)$ decision variables. Hence, when $\Delta = O(1)$, this LP can be solved optimally in polynomial time. In what follows, we construct our output initial inventory vector $c^{\mathrm{LP}}$ from an optimal solution to~\ref{eqn:fluid_LP}, denoted by $\{h^*(A)\}_{A \in \Acal_{\mathrm{RO}}}$ and $\{u_i^*\}_{i \in \biguplus_{q \in Q_H} \Ncal_{q,H}}$, and show that its fluid revenue is $(1-O(\eps))$-optimal for~\ref{eqn:opt_inv_rounded}.

\paragraph{Choosing starting inventories.} We construct $c^{\mathrm{LP}}$ as
\[
c^{\mathrm{LP}}_i
=
\begin{cases}
\hat{c}_i 
& \text{if } \displaystyle i \in \biguplus_{q \in [Q]} \biguplus_{k \in k(q)} \Ncal_q(k), \\[0.75em]
\lfloor u_i^* \rfloor 
& \text{if } \displaystyle i \in \biguplus_{q \in Q_H} \Ncal_{q,H},
\end{cases}
\]
which is clearly feasible by constraint~(4) of~\ref{eqn:fluid_LP}. The following lemma, whose proof can be found in Appendix~\ref{app:proof_final_guarantee}, analyzes the fluid revenue of $c^{\mathrm{LP}}$ and shows that it differs from $Z_{\mathrm{LP}}^*$ by at most an $O(\eps)$-factor.
\begin{lemma}
    \label{lem:final_guarantee}
    For any $\eps \in [0,\frac{1}{4}]$, we have that $\rev_{\hat{w}}(c^{\mathrm{LP}}) \geq (1-4\eps)\cdot Z_{\mathrm{LP}}^*$.
\end{lemma}
The proof of this lemma relies crucially on the CCD property of the mixed-NP-MNL model to show that the distribution over assortments induced by $\{h^*(A)\}_{A \in \Acal_{\mathrm{RO}}}$ can be represented as a simple nested sequence of assortments that corresponds almost exactly to a valid show-all policy. Moreover, the proof reveals why constraint~(3) of~\ref{eqn:fluid_LP} is critical.

\paragraph{The final guarantee.} In what follows, we analyze the fluid revenue of $c^{\mathrm{LP}}$, the starting inventory vector outputted by our approach, under the true $w$-weights. We then summarize the running time required to compute this starting inventory vector. To begin, observe that by combining Lemmas~\ref{lem:fluid_LP_lb} and~\ref{lem:final_guarantee}, we obtain
\[
\rev_{\hat{w}}(c^{\mathrm{LP}}) \geq (1-5\eps)\cdot \opt_{\hat{w}}.
\]
It then follows from Lemma~\ref{lem:proxy_near_op} that
\[
\rev_w(c^{\mathrm{LP}}) \geq (1-15\eps) \cdot \opt_w,
\]
showing that we indeed constructed an oracle for~\ref{eqn:opt_inv_oracle} with approximation factor $\alpha = 1 - O(\eps)$. The running time required to compute $c^{\mathrm{LP}}$ is dominated by the various guessing steps, which together yield an overall running time of
\[
O\left(\left(\frac{n\Ccal\Delta}{\eps^2}\right)^{O\left(\frac{1}{\eps^2} \cdot \log \Delta\right)}\right).
\]
The factor of $m$ stated in the final running time of Theorem~\ref{thm:one_half_fluid} comes from the running time of the even/odd-sized dynamic program.

\section{The Stochastic Show-All Inventory Stocking Problem}\label{sec:asym_one_half}

In this section, we consider a version of~\ref{eqn:Show_all} with scaling parameter $\theta \in \mathbb{Z}_+$, which enables our asymptotic analysis.  Specifically,  define $\Rcal_{\theta}(c)$ to be the random revenue earned over discrete periods $\{ 1,2,\ldots, \theta T \}$ when starting from an initial inventory vector $c \in \Fcal_{\theta} \equiv \{ c \in \mathbb{Z}_+^{n \times m} \, : \, || c ||_1 \leq \theta \cdot \Ccal \}$. We consider the problem
\begin{align}
	\label{eqn:INV-ST-theta}
	\tag{SA-$\theta$}
	\opt_{\theta} = \max_{c \in \Fcal_{\theta}} \ex{\Rcal_{\theta}(c)} 
\end{align}
 Our main result, formally stated below, is to show that we can construct a parameterized starting inventory vector that is asymptotically $\frac{1}{2}$-optimal as $\theta \to \infty$. 
\begin{theorem}
    \label{thm:one_half_stochastic}
    For every integer $\theta \geq 1$, we can find $\hat{c}_{\theta} \in \Fcal_{\theta}$ such that
    \[
    \lim_{\theta \to \infty} \frac{\ex{\rev_{\theta}(\hat{c}_{\theta})}}{\opt_{\theta}} = \frac{1}{2}.
    \]
\end{theorem}
The proof unfolds over the subsections that follow, beginning with a description of the fluid linear program that is used to select our starting inventory vector.

\subsection{The Fluid-LP based approach}
\label{subsec:asymptotic-fluid-LP}

Our construction of the asymptotic $1/2$-optimal solution relies on developing an asymptotically optimal solution to the parameterized version of~\ref{eqn:opt_inv_even}.  Namely, let 
\[
\Fcal_{\even, \theta}
=
\left\{
c \in \mathbb{Z}^{n \times m}_+ :
\sum_{(i,s) \in \Ucal_{\even}} c_{i,s} \leq \theta\cdot \Ccal,\;
\sum_{(i,s) \in \Ucal_{\odd}} c_{i,s} = 0
\right\}
\]
and consider the problem 
\begin{align}
	\label{eqn:INV-even-ST}
	\tag{SA-EVEN-$\theta$}
	\opt_{\even,\theta} = \max_{c \in \Fcal_{\even, \theta}} \ex{\Rcal_{\theta}(c)}.
\end{align}
Defining $\opt_{\odd,\theta}$ analogously, it is straightforward to show that
$
\max\{\opt_{\even,\theta},\opt_{\odd,\theta}\}
\geq \frac{\opt_{\theta}}{2}.
$
Thus, without loss of generality, we assume that $\opt_{\even,\theta}\geq \opt_{\odd,\theta}$, and consequently $\opt_{\even,\theta}\geq \frac{\opt_{\theta}}{2}$. To approximate \ref{eqn:INV-even-ST}, we rely on a fluid linear program that jointly determines inventory levels and assortment offerings. Recalling that $\Ucal_{\even}$ denotes the universe of products with even-indexed sizes, consider
\begin{equation*}
	\label{eqn:asymptotic-joint-LP}
	\tag{JLP-EVEN}
	\begin{aligned}
		Z^{\JLP}_{\even}(\theta) \quad = \quad   \underset{ h \geq 0, c \geq 0 }{\text{max}} \quad
		&    \sum_{(i,s) \in \Ucal_{\even} }  r_i \cdot c_{i,s} \\[0.3em]
		\text{s.t.} \quad
		& \sum_{A \subseteq {\Ucal}_{\even} } h(A) =\theta \cdot T, & \\[0.3em]
		&  \sum_{A \subseteq {\Ucal_{\even}}} \pi((i,s), A) \cdot h(A)  = c_{i,s} , & \forall (i,s) \in {\Ucal_{\even}}, \\[0.3em]
		& \sum_{(i,s) \in {\Ucal_{\even}} } c_{i,s} \leq \theta \cdot \Ccal. & %
	\end{aligned}
\end{equation*}
Although this LP has an exponential numbers of variables, in Appendix~\ref{subsec:app-asymptotic-constraint-generation} we show that it can be solved optimally via constraint generation with a $O(n)$-time separation oracle.

\paragraph{Constructing the asymptotically $1/2$-optimal solution.} Let $\bar{c}_\theta$ represent the optimal $c$-variables corresponding to $Z^{\JLP}_{\even}(\theta)$ and define
\[
\hat{c}_\theta \equiv \floor{ \bar{c}_\theta },
\]
where we take the integer part of the starting inventory level of each product $(i,s) \in \Ucal_{\even}$ under $\bar{c}_\theta$ while stocking nothing for products in $\Ucal_{\odd}$. In the remainder of the section, we will prove that
\begin{align}
	\label{eq:performance-of-rounded-CDLP-sol}
	\lim_{\theta \rightarrow \infty} \frac{  \Ebb \left[ \Rcal_\theta \left(  \hat{c}_\theta \right) \right]   }{ \text{OPT}_{\even,\theta}  } =1,
\end{align}
which, together with $\opt_{\even,\theta}\geq \frac{1}{2}\opt_{\theta}$, establishes that $\hat{c}_{\theta}$ is asymptotically $1/2$-optimal and proves Theorem~\ref{thm:one_half_stochastic}.

\subsection{Proof overview}
\label{appendix-subsec:asymptotic-proof-overview}

Our proof of the asymptotic performance guarantee proceeds in two steps and establishes a sequence of comparisons among three problems: (i) the joint LP, \ref{eqn:asymptotic-joint-LP}, (ii) the inventory problem under the fluid show-all demand process, and (iii) the inventory problem under the stochastic show-all demand process. %

\begin{itemize}
	\item {\bfseries Step 1: Bridging the joint LP and the fluid show-all process.} 
	The first step establishes that, under the fluid show-all process, the revenue generated by $\hat{c}_\theta$ asymptotically matches the objective value of \ref{eqn:asymptotic-joint-LP}. Conceptually, this step transfers the performance guarantee of $\hat{c}_\theta$ from a setting that allows dynamic assortment adjustments (as in~\ref{eqn:asymptotic-joint-LP}) to the more restrictive show-all setting in which no such flexibility is available. 

    Particularly, we show that, under the fluid show-all process, the revenue generated by $\hat{c}_\theta$ converges to $Z^{\JLP}_{\even}(\theta)$, the optimal value of \ref{eqn:asymptotic-joint-LP}. This convergence relies on two key properties: (i) the rounding from $\bar{c}_\theta$ to $\hat{c}_\theta$ incurs negligible loss in the fluid limit, and (ii) under $\hat{c}_\theta$, demand decomposes across sizes and follows the mixed-NP-MNL model. The CCD property of this model then implies that the fluid show-all revenue coincides with the joint LP objective, thereby completing the bridge between the two formulations.

	\item {\bfseries Step 2: Bridging the fluid and stochastic show-all processes.} 
	The second step shows that the expected revenue in the stochastic show-all process starting from the initial inventory vector $\hat{c}_\theta$ converges to its fluid counterpart. The main challenge is that, in the stochastic setting, random customer choices induce random stockout events, leading to random assortment sequences and consumption patterns that may differ substantially from the deterministic depletion trajectory of the fluid process.

	To address this issue, we construct an auxiliary stochastic process that mirrors the fluid trajectory by fixing the sequence of assortments according to the fluid stockout order, while keeping customer choices stochastic. This construction decouples assortment decisions from inventory feasibility: customers are allowed to select out-of-stock products (yielding zero revenue), and any remaining inventory at the end of the horizon is penalized. We show, via a coupling argument, that the expected revenue of this penalized process is dominated by that of the true stochastic process, while at the same time it converges to the fluid revenue. This establishes that the stochastic revenue of $\hat{c}_\theta$ asymptotically matches its fluid counterpart.
	
\end{itemize}

Combining the two steps, we establish the asymptotic limit \eqref{eq:performance-of-rounded-CDLP-sol} and conclude that $\hat{c}_\theta$ achieves $\frac{1}{2}$-asymptotic optimality. The detailed arguments for each step are presented in the following subsections.

\subsection{Step 1: Bridging the Joint LP with the Fluid Show-All Processes}

In this step, we establish a connection between \ref{eqn:asymptotic-joint-LP}, which constructs the inventory vector $\hat{c}_\theta$, and the show-all demand process under the fluid setting. This connection serves as the foundation for deriving the asymptotic performance guarantee.

\paragraph{The joint LP as an upper bound.} 
We begin by relating \ref{eqn:asymptotic-joint-LP} to~\ref{eqn:INV-even-ST}. We establish that \ref{eqn:asymptotic-joint-LP} provides an upper bound on $\text{OPT}_{\even,\theta}$.%

\begin{claim}
	\label{claim:asymptotic-joint-LP-upperbound}
	For all $\theta > 0$, we have $Z^{\JLP}_{\even}(\theta) \geq \opt_{\even,\theta}$.
\end{claim}

We omit the proof if this result  as it follows from standard arguments (see, e.g., Proposition 1 of \cite{sun2024unified}) that the stochastic show-all inventory problem is upper bounded by its fluid counterpart when both the initial inventory and subsequent assortment decisions are jointly optimized.

\paragraph{Bridging to the fluid process via CCD.} 
We next connect the joint LP, \ref{eqn:asymptotic-joint-LP}, to the fluid show-all demand process. To this end, let $\Rcal_\theta^{\FP}(c)$ denote the total accrued fluid revenue over the continuous time interval $[0, \theta T]$. Here, we will concurrently consider fluid and stochastic revenues, so it is important to introduce notation to differentiate the two. The following claim shows that the fluid revenue obtained from the constructed solution $\hat{c}_\theta$ converges to the optimal objective value of \ref{eqn:asymptotic-joint-LP}.

\begin{claim}
	\label{claim:asymptotic-fluid-and-joint-LP}
	$
	\lim_{\theta \rightarrow \infty}  \Rcal^{\FP}_\theta(\hat{c}_\theta) \big/ Z^{\JLP}_{\even}(\theta) = 1$.
\end{claim}
For brevity, we omit the proof of Claim~\ref{claim:asymptotic-fluid-and-joint-LP} and instead outline the main. First, by the fluid demand regularity (Claim~\ref{claim:fluid_regular}), it is straightforward to observe that the sales loss incurred for each product when rounding $\bar{c}_\theta$ to $\hat{c}_\theta$ is negligible in the fluid limit, since $\bar{c}_\theta$ scales linearly with $\theta$. Consequently,
\[
\Rcal^{\FP}_\theta(\hat{c}_\theta) \rightarrow \Rcal^{\FP}_\theta(\bar{c}_\theta) \quad \text{as} \quad \theta \rightarrow \infty.
\]
Second, under the inventory vector $\bar{c}_\theta$, which corresponds to the optimal $c$-variables of \ref{eqn:asymptotic-joint-LP}, demands for stocked products with distinct sizes are independent. Therefore,  we know that for each size $s \in \Scal_{\even}$, the associated demand process is characterized by the mixed-NP-MNL model, which satisfies the CCD property. This property implies that, when starting with the inventory vector $\bar{c}_\theta$, the fluid revenue coincides with the optimal objective value of \ref{eqn:asymptotic-joint-LP}, i.e.,
\[
\Rcal^{\FP}_\theta(\bar{c}_\theta) = Z^{\JLP}_{\even}(\theta).
\]
Combining these observations yields the desired convergence.

\subsection{Step 2: Bridging the Fluid and Stochastic show-all processes.}

We now turn to the relationship between the stochastic and fluid demand processes. The following claim establishes that, in expectation, the stochastic revenue converges to its fluid counterpart asymptotically.

\begin{claim}
	\label{claim:asymptotic-limit-of-stochastic-and-fluid}
	$\lim_{\theta \rightarrow \infty} \ex{\Rcal_{\theta}(\hat{c}_\theta)} \big/  \Rcal^{\FP}_\theta(\hat{c}_\theta) =  1$.
\end{claim}
This convergence result is the final ingredient needed to establish the asymptotic performance guarantee. Indeed, combining Claims~\ref{claim:asymptotic-joint-LP-upperbound}, \ref{claim:asymptotic-fluid-and-joint-LP}, and~\ref{claim:asymptotic-limit-of-stochastic-and-fluid}, we obtain
\[
\frac{\ex{\Rcal_{\theta}(\hat{c}_\theta)}}{\opt_{\even,\theta}}
\;\geq\;
\frac{\ex{\Rcal_{\theta}(\hat{c}_\theta)}}{Z^{\JLP}_{\even}(\theta)}
\;\rightarrow\; 1
\quad \text{as} \quad \theta \rightarrow \infty,
\]
which establishes the $\frac{1}{2}$-asymptotic optimality of $\hat{c}_\theta$ since $\opt_{\even,\theta}\geq \frac{1}{2}\opt_{\theta}$. We note that it suffices to prove Claim~\ref{claim:asymptotic-limit-of-stochastic-and-fluid} separately for each size $s \in \Scal_{\even}$ under $\hat{c}_\theta$. This follows from the independence of demand across different sizes in $\Scal_{\even}$ under the constructed inventory vector $\hat{c}_{\theta}$. Consequently, convergence at the level of individual sizes implies convergence for the aggregate system. 

\paragraph{Proving Claim~\ref{claim:asymptotic-limit-of-stochastic-and-fluid} via a probabilistic coupling}.  We fix a size $s \in \Scal_{\even}$ throughout. To simplify notation, we denote the inventories of the size-$s$ products $\left(\hat{c}_\theta\right)_{(i,s) \in \Ucal(s)}$ by $\hat{c}_\theta$, where $\Ucal(s)$ is the set of products of size $s$. Under this convention, $\hat{c}_\theta$ can be viewed as an $n$-dimensional inventory vector indexed by products in $\Ucal(s)$.  For ease of notation, since the size $s$ is fixed, we drop this subscript and assume that $\hat{c}_{\theta} = (c_1, \ldots c_n)$ and that $c_i >0$ for each $i \in [n]$. As discussed at the end of Step~1, the independence of demand across sizes in $\Scal_{\even}$ implies that the choice probabilities follows the mixed-NP-MNL model.

Our goal is to establish the convergence of $\Ebb\left[\Rcal_\theta(\hat{c}_\theta)\right]$ to $\Rcal^{\FP}_\theta(\hat{c}_\theta)$ as $\theta$ tends to infinity. To this end, we develop a coupling argument that relates the random revenue earned under the true stochastic process to the random revenue earned by an assortment policy that mirrors the assortment trajectory of the fluid show-all process, but is penalized for deviations from the required show-all policy that must be implemented.  Next, we formally define these stochastic processes. Since they all share the same initial inventory vector $\hat{c}_\theta$, we suppress the dependence on $\hat{c}_\theta$ in the notation. 

\paragraph{The true stochastic process.}
We begin by recalling the true stochastic process, denoted by $\Pcal$. Let $A_t$ denote the set of products with strictly positive inventory at time $t$. In each period $t \in \{1,2,\ldots,\theta T\}$, a customer arrives and selects either a product $i_t \in A_t$ or the no-purchase option $i_t = 0$, according to the mixed-NP-MNL model. If $i_t \in A_t$, then the inventory level of product $i_t$ decreases by one. The total revenue collected under this process is given by $\REV_\Pcal = \sum_{t=1}^{\theta T} r_{i_t}$. By definition, $\Ebb\left[\Rcal_\theta(\hat{c}_\theta)\right] = \Ebb\left[\REV_\Pcal\right]$.

\paragraph{The penalized fluid-mirrored process.}
We next introduce an auxiliary stochastic process, referred to as the penalized fluid-mirrored process, which is denoted by $\Pcal_{\downarrow}$. This process is designed to mirror the sequence of assortment offerings induced by the fluid solution while allowing customers to choose products without regard to inventory feasibility. $\Pcal_{\downarrow}$ serves as an intermediate process that remains analytically tractable, closely tracks the fluid dynamics, and is dominated by the true stochastic process due to explicit penalties introduced in its revenue function.

To define $\Pcal_{\downarrow}$, we first introduce several key definitions and observations associated with the fluid process under the initial inventory $\hat{c}_\theta$:
\begin{itemize}
	\item \emph{Stockout times.} Let $\tau_i$ denote the stockout time of product $i$ in the fluid process. Without loss of generality, we index products in nondecreasing order of their stockout times, so that
	\[
	0 \equiv \tau_0 \leq \tau_1 \leq \tau_2 \leq \cdots \leq \tau_n \leq \theta T.
	\]
    Note that $\tau_n \leq \theta T$ always holds. Since $\bar{c}_\theta$ is the optimal-$c$ variable to \ref{eqn:asymptotic-joint-LP}, the CCD property of the mixed-NP-MNL model implies that $\bar{c}_\theta$ is fully depleted under the fluid process. It then follows by Claim~\ref{claim:fluid_regular} that $\hat{c}_\theta = \floor{\bar{c}_\theta}$ is also fully depleted by time $\theta T$ in the fluid process.

	\item \emph{Predetermined assortment offerings.} Based on the stockout times, we define a sequence of assortments $\Acal_\ell \equiv \{\ell, \ell+1, \ldots, n\}$ 
	for each $\ell \in [n]$. In process $\Pcal_{\downarrow}$, assortment $\Acal_\ell$ is offered during periods
	\[
	\left\{ \sum_{k=1}^{\ell-1} \floor{\tau_k - \tau_{k-1}} + 1, \ldots, \sum_{k=1}^{\ell} \floor{\tau_k - \tau_{k-1}} \right\}.
	\]
	Thus, the sequence of assortments is fixed in advance so as to replicate the stockout order in the fluid process.
	
	\item \emph{Decoupling from inventory feasibility.} In contrast to the true stochastic process, the offered assortment at any time does not depend on the remaining inventory. Consequently, a customer may select a product that is already out of stock, or products with positive inventory may be excluded from the offered assortment. The former leads to lost sales opportunities, while the latter results in unsold inventory at the end of the horizon.

	\item \emph{Real versus virtual purchases.} A purchase is called \emph{real} if the selected product has positive inventory at the time of purchase, and \emph{virtual} otherwise. Only real purchases reduce inventory levels, whereas virtual purchases do not affect inventory and do not generate revenue.
	
	\item \emph{Penalty.} The discrepancies introduced above are incorporated through two types of penalties. First, a virtual purchase represents a lost sales opportunity and contributes zero revenue. Second, any unit that remains unsold at the end of the horizon incurs a penalty of $r_{\max} = \max_{i \in [n]} r_i$. These penalties ensure $\Pcal_{\downarrow}$ is pessimistic relative to the true stochastic process.
	
\end{itemize}

The process $\Pcal_{\downarrow}$ does not correspond to a feasible implementation of the show-all policy; rather, it is introduced as an analytical device that decouples assortment decisions from inventory dynamics. The penalty structure ensures that the resulting process remains conservative relative to the true stochastic process. Under $\Pcal_{\downarrow}$, each arriving customer chooses from the offered assortment according to the mixed-NP-MNL model, independently of the current inventory levels. In particular, customer choices depend only on the predetermined assortment sequence and not on inventory availability.

\paragraph{Revenue representation of the penalized process.}
We now formalize the revenue function of $\Pcal_{\downarrow}$. For each pair $(i,\ell)$, let $Y_{i,\ell}$ denote the number of times product $i$ is selected during the periods in which assortment $\Acal_\ell$ is offered. Then
\[
Y_{i,\ell} \sim \text{Binomial}\left(\floor{\tau_\ell - \tau_{\ell-1}}, \pi(i,\Acal_\ell)\right).
\]
Let $Y_i = \sum_{\ell \in [n]} Y_{i,\ell}$ denote the total number of selections of product $i$ over the entire horizon. The contribution of product $i$ to the total revenue is given by
\[
\REV_{\Pcal_{\downarrow},i} \, \equiv \, r_i \cdot \min\left\{ \hat{c}_{\theta,i}, Y_i\right\}
- r_{\max} \cdot \left( \hat{c}_{\theta,i} - Y_i\right)^+,
\]
where $(a)^+ = \max\{a,0\}$. The first term captures revenue generated by real purchases, while the second term reflects the penalty from unsold inventory. Summing over all products, the revenue of the penalized process is $
\REV_{\Pcal_{\downarrow}} = \sum_{i \in [n]} \REV_{\Pcal_{\downarrow},i}$. 
An important feature of $\Pcal_{\downarrow}$ is that stockout events do not influence future customer choices, as the assortment sequence is predetermined; their impact is reflected solely through the realized revenue via virtual purchases and end-of-horizon penalties.

\paragraph{The fluid benchmark process.}
Finally, let $\Pcal_F$ denote the fluid process under the show-all policy. It is precisely the fluid process with initial inventory vector $\hat{c}_\theta$, time horizon $\theta T$, and is deterministic. Using the stockout times defined above, the total fluid revenue can be expressed as
\[
\REV_{\Pcal_F} \equiv \Rcal^{\FP}_\theta(\hat{c}_\theta)
= \sum_{\ell \in [n]} (\tau_\ell - \tau_{\ell-1}) \cdot \sum_{i \in \Acal_\ell} r_i \cdot \pi(i,\Acal_\ell).
\]
This representation highlights that the fluid process induces the same sequence of assortments as $\Pcal_{\downarrow}$, but without stochastic fluctuations or feasibility violations.

\paragraph{Bridging the stochastic and fluid processes.} In the subsequent analysis, we bridge the true stochastic process $\Pcal$ and the fluid benchmark $\Pcal_F$ through intermediate process $\Pcal_\downarrow$. Specifically, we first couple $\Pcal$ and $\Pcal_\downarrow$ through the same customer types and utility realizations for all products, and then compare $\Pcal_\downarrow$ with $\Pcal_F$.

\begin{lemma}
	\label{lemma:asymptotic-performance-comparison}
	We have (i) $\Ebb \left[  \REV_\Pcal  \right] \geq \Ebb \left[  \REV_{\Pcal_{\downarrow}}  \right] $ for all $\theta > 0$; and (ii) $\lim_{\theta \rightarrow \infty}  \Ebb \left[  \REV_{\Pcal_{\downarrow}}  \right]  = \REV_{\Pcal_F}$.
\end{lemma}

We provide the detailed proof in Appendix~\ref{appendix-subsubsec:asymptotic-proof-coupling}. The first statement follows from the aforementioned coupling between $\Pcal$ and $\Pcal_{\downarrow}$, under which we show that $\REV_\Pcal \geq \REV_{\Pcal_{\downarrow}}$ almost surely. The second statement follows by showing that $\REV_{\Pcal_{\downarrow}}$ concentrates around $\REV_{\Pcal_F}$ as $\theta \to \infty$. Together, these two statements establish Claim~\ref{claim:asymptotic-limit-of-stochastic-and-fluid}, thereby completing the second step of our analysis.

Combining this result with Step~1 completes the proof of Theorem~\ref{thm:one_half_stochastic}. In particular, Step~1 shows that the fluid revenue generated by $\hat{c}_\theta$ asymptotically matches the objective value of the joint LP, which is at least one half of the optimal stochastic benchmark $\opt_\theta$. Step~2 then transfers this performance guarantee from the fluid process to the stochastic show-all process through the coupling developed above. Consequently, the expected revenue generated by $\hat{c}_\theta$ is asymptotically at least $\frac{1}{2} \opt_\theta$, establishing the desired $\frac{1}{2}$-asymptotic optimality. The auxiliary process $\Pcal_{\downarrow}$ is central to this transfer: by fixing the assortment trajectory according to the fluid process while retaining stochastic inventory consumption, it provides a common basis for comparing the otherwise distinct stochastic and fluid dynamics.

\paragraph{Extension: Ordering/stocking costs.}
Our results extend naturally to settings with per-unit ordering costs. Suppose that ordering one unit of product $(i,s)\in\Ucal_{\even}$ incurs a cost $o_{i,s}$. In the fluid setting, it is without loss of optimality to order exactly $c_{i,s}$ units, since any inventory remaining at the end of the horizon can simply be removed from the initial inventory vector without affecting sales. Consequently, incorporating ordering costs into~\ref{eqn:asymptotic-joint-LP} requires only replacing the objective coefficient $r_i$ of each $c_{i,s}$ variable by its unit margin $r_i-o_{i,s}$. The remainder of our analysis then carries through unchanged, and the asymptotic $\frac{1}{2}$-approximation guarantee continues to hold.

\section{Numerical Experiments} \label{sec:numerical_experiments}

In this section, we evaluate the performance of practical adaptations of the proposed inventory policies using real-world sales data from a large footwear retailer. We first calibrate the CFTC model following \cite{akchen2023size} and then study the resulting inventory planning problem in the fluid setting under different levels of size substitution and problem scales.

\subsection{Data and model calibration}

We estimate the CFTC model using the same real-world sales dataset analyzed in \cite{akchen2023size}. Below, we provide a high-level overview of the data and calibration procedure; additional details can be found in Section~4 of \cite{akchen2023size}.

\paragraph{Data.} The dataset contains transaction records from a large footwear retailer operating hundreds of brick-and-mortar stores across North America\footnotemark{}. The data span 33 weeks during the 2019--2020 season, from late July 2019 to mid-March 2020, and focus on women's casual ankle boots.\footnotetext{Owing to a non-disclosure agreement, the retailer's identity remains confidential.}

The category contains sales records over a 33-week horizon for 51 shoe bases (i.e., shoe styles), each offered in nine sizes ranging from size 6 to size 10, including half sizes. To maintain computational tractability, we restrict attention to the $n$ most popular bases, where $n \in \{ 4,8,16 \}$, and aggregate the remaining bases and their associated sizes into the outside option. This preprocessing yields instances with $n \times m = 36$, $72$, and $144$ products, where $m=9$ is the number of sizes. These instance sizes are comparable to, or larger than, those commonly used in the assortment and inventory optimization literature \citep{berbeglia2022comparative,zhang2025leveraging}.

\paragraph{Model calibration.} We transform the transaction data into assortment--choice pairs. For each store and week, the dataset records the set of products with positive inventory at the beginning of the week. Following \cite{akchen2023size}, we assume that this assortment remains unchanged throughout the week so that all customers visiting a given store during the same week observe the same assortment. Each observed purchase is therefore interpreted as a customer choice from the corresponding weekly assortment. This procedure converts the transaction data into assortment--choice pairs, from which we estimate the CFTC model using the expectation--maximization (EM) algorithm proposed in \cite{akchen2023size}.

The calibrated model directly satisfies the assumptions introduced at the beginning of Section~\ref{sec:fluid_setting}. Moreover, the EM procedure assumes that the size-based discount factors $\beta_{0,g}$ and $\beta_{1,g}$ are homogeneous across customer types. Without loss of generality, we normalize $\beta_{0,g}=1$ for all $g \in \Gcal$ and define $\beta_1 \equiv \beta_{1,g}$ throughout this section.

\paragraph{Model summary.} The calibrated model consists of 18 customer types, each characterized by a pair $(s_g,\tau_g)$, where $s_g \in [m]$ denotes the customer's most preferred size and $\tau_g \in \{\uparrow,\downarrow\}$ captures the direction of size substitution. The estimated base preference weights $(w_i)$ range from 0.044 to 0.175, while the customer-type population $(\mu_g)$ shares range from 3.6\% to 8.8\%. Consistent with intuition, customer types associated with middle sizes constitute a larger proportion of the population, whereas customers with more extreme size preferences are less prevalent. In the footwear context, prices are generally associated with bases and do not vary across sizes. We estimate the price $r_i$ of each base $i \in \Ncal$ by averaging its observed transaction prices. The resulting base prices range from $49.99$ to $131.06$. All but two bases offer the complete size range; the remaining two offer only odd-indexed sizes. We therefore set the unit revenues of their even-indexed sizes to zero.

In all subsequent experiments, we fix the estimated model parameters and vary only the size-based discount factor $\beta_1$ to evaluate the performance of inventory policies under different degrees of size substitution. Specifically, we consider $\beta_1 \in \{0.05, 0.2, 0.4, 0.6, 0.8, 0.95 \}$. Larger values of $\beta_1$ correspond to stronger size substitution, indicating that customers are more willing to switch to their next preferred size when their most preferred size is unavailable. Conversely, smaller values imply weaker substitution behavior. For reference, \cite{akchen2023size} reports an estimated value of approximately $\beta_1 = 0.25$.

\subsection{Experimental setting and implemented policies}

\paragraph{Experimental setup.} We restrict our analysis to the fluid inventory setting for the asymptotic region described in Section~\ref{sec:asym_one_half}. To evaluate the performance of different inventory policies, for each base $n \in \{4,8,16 \}$, we scale both the horizon length $T$ and the inventory capacity $\Ccal$ while keeping their ratio fixed. Specifically, we consider $T \in \{400,800,1200,1600 \}$. We choose $\Ccal$ to ensure that the inventory constraint plays a meaningful role in the optimization problem. If $\Ccal$ is too large relative to $T$, inventory is abundant and the problem essentially reduces to identifying and stocking the products in the optimal assortment. Conversely, if $\Ccal$ is too small, the firm can stock only a few products, making it sufficient to concentrate inventory on the most profitable ones. For $T=400$, we set $\Ccal=50$, $75$, and $100$ for $n=4$, $8$, and $16$, respectively, and scale $\Ccal$ proportionally with $T$ for the remaining instances.

\paragraph{A Greedy benchmark.} We benchmark our CDLP-based approached against a greedy heuristic, denoted by \texttt{Greedy}. Starting from the zero inventory vector $c=0$, the algorithm iteratively adds one unit of inventory to the product that yields the largest marginal increase in revenue. Specifically, at each iteration, the policy adds one unit of inventory to the product $(i,s)$ that maximizes $\Rcal( c+\mathbb{I}_{(i,s)})-\Rcal(c)$ over all $(i,s) \in \Ucal$, provided that the maximum marginal increase is positive. The procedure terminates when no further one-unit inventory addition increases the objective value or when the capacity constraint becomes tight.

\paragraph{Inventory policies.} We evaluate three inventory policies, two of which are inspired by the CDLP framework developed in Section~\ref{sec:fluid_setting}. Recall that our nearly $1/2$-optimal approximation algorithm combines four key components: a dynamic program for allocating capacity across the even- and odd-size partitions, a preference-weight rounding procedure, a guessing procedure for low-stock products, and a CDLP-based optimization routine for highly stocked products.

Because the goal of our numerical study is to evaluate practically implementable inventory policies rather than to reproduce the full approximation algorithm, we consider two simple CDLP-inspired policies. \emph{Our theoretical results provide a foundation for these heuristics, as the central component of our algorithms is the use of the fluid linear program to determine initial inventory levels.} The two policies retain this core optimization structure while omitting technical procedures introduced primarily to establish worst-case approximation guarantees. In particular, they preserve the capacity-allocation and linear-programming components but exclude the preference-weight rounding and low-stock guessing procedures, which are needed mainly to guarantee performance in the small-$T$ regime.

\paragraph{CDLP-H.} Our first CDLP-inspired policy, denoted by \texttt{CDLP-H} (where ``H'' stands for ``half of the sizes''), exploits the decomposition induced by the even- and odd-size partitions. For a fixed horizon $T$ and an auxiliary capacity level $\Ccal' \leq \Ccal$, we solve the joint linear program \ref{eqn:asymptotic-joint-LP} over the even-size product set $\Ucal_{\mathrm{even}}$, using horizon $T$ and capacity $\Ccal'$. Let $c'$ denote the resulting optimal inventory vector. Since $c'$ is fractional, we construct an integral solution by rounding each component down to obtain $\lfloor c' \rfloor$. Because rounding may leave residual capacity, i.e., $|| \lfloor c' \rfloor||_1 < \Ccal$, we further augment the solution using the greedy procedure described above until no further improvement in the objective value is possible or the capacity constraint becomes tight. We repeat the same procedure for the odd-size product set $\Ucal_{\mathrm{odd}}$. To account for different allocations of inventory across the two partitions, we consider $\Ccal' \in \{0.1 \cdot \Ccal, 0.2 \cdot \Ccal, \ldots, \Ccal\}$. This generates twenty candidate inventory vectors in total, and we select the one with the highest revenue.

The \texttt{CDLP-H} policy closely mirrors the approximation framework developed in Section~\ref{subsec:asymptotic-fluid-LP}. It also resembles Steps~1 and~4 of the approximation procedure in Section~\ref{subsec:technical_overview_opt_inv} while omitting the guessing procedure in Steps~2 and~3, which are mainly required to establish worst-case performance guarantees in finite-capacity settings.

\paragraph{CDLP-F.} Our second CDLP-inspired policy, denoted by \texttt{CDLP-F}, utilizes the full product set $\Ucal$. Specifically, we solve the full-size counterpart of \ref{eqn:asymptotic-joint-LP} by replacing $\Ucal_{\mathrm{even}}$ with $\Ucal$ and setting the horizon and capacity parameters to $T$ and $\Ccal$, respectively. After obtaining the optimal fractional inventory vector, we round each component down and subsequently apply the greedy augmentation procedure.

We solve the resulting joint linear program using the constraint generation framework described in Appendix~\ref{subsec:app-asymptotic-constraint-generation}. Unlike the joint linear program defined over $\Ucal_{\mathrm{even}}$, the formulation based on the full product set $\Ucal$ gives rise to a separation subproblem that coincides with the single-period assortment optimization problem in \ref{eqn:AO_CFTC}, which is NP-hard. For large-scale instances, the PTAS developed in Appendix~\ref{app-sec:AO} can be integrated into the constraint generation procedure to obtain a $(1-\epsilon)$-approximate solution to the joint linear program in polynomial time. For simplicity, we solve the assortment optimization subproblem exactly using integer programming, which is straightforward to implement.

\paragraph{Measure of Optimality Gap.} To benchmark the performance of the proposed inventory policies, we first use the optimal objective value of the joint linear program over the full product set $\Ucal$ as an upper bound. This follows from the standard observation that the fluid relaxation of the joint inventory and assortment offering problem establishes an upper bound for the optimal value of the stochastic show-all inventory problem; see Proposition~1 of \cite{sun2024unified}. We denote this upper bound by $Z_{\mathrm{Upper}}$. Since \texttt{CDLP-F} requires solving the full joint linear program, $Z_{\mathrm{Upper}}$ is obtained as a byproduct of computing the \texttt{CDLP-F} inventory vector. For any feasible inventory vector $c$, we define its optimality gap as $\texttt{G}(c) \equiv 1 - \frac{\Rcal(c)}{Z_{\mathrm{Upper}}}$.

\subsection{Results}

Table~\ref{table:experiment_optimality_gap} reports the optimality gaps of the three inventory policies across different values of the size-substitution parameter $\beta_1$, horizon length $T$, and number of bases $n$, with nine sizes throughout. We first focus on the sixteen-base instances ($n=16$), with similar patterns generally observed for four- and eight-base instances.

\begin{table}[]
	\centering
    \SingleSpacedXI
	\small
	\begin{tabular}{ccccccccccc} \toprule[1.0pt]
		&        &      \multicolumn{3}{c}{\bf Four Bases ($nm=36$)}               & \multicolumn{3}{c}{\bf Eight Bases ($nm=72$)}               & \multicolumn{3}{c}{\bf Sixteen Bases ($nm=144$)}              \\ \cmidrule[1.0pt](lr){1-2} \cmidrule[1.0pt](lr){3-5} \cmidrule[1.0pt](lr){6-8} \cmidrule[1.0pt](lr){9-11}
		$\beta_1$ & $T$       & \texttt{Greedy} & \texttt{CDLP-H} & \texttt{CDLP-F} & \texttt{Greedy} & \texttt{CDLP-H} & \texttt{CDLP-F} & \texttt{Greedy} & \texttt{CDLP-H} & \texttt{CDLP-F} \\ \cmidrule[1.0pt](lr){1-2} \cmidrule[1.0pt](lr){3-5} \cmidrule[1.0pt](lr){6-8} \cmidrule[1.0pt](lr){9-11}
		0.05      & 400     & 2.14                         & 2.14                             & 2.35                             & 2.57                         & 1.68                             & 1.66                             & 4.66                         & 2.10                              & 2.06                             \\
		& 800     & 2.15                         & 1.64                             & 1.75                             & 3.49                         & 0.89                             & 0.65                             & 4.39                         & 1.57                             & 1.00                                \\ 
		& 1200    & 2.12                         & 1.29                             & 1.15                             & 3.66                         & 1.42                             & 0.66                             & 4.80                          & 1.50                              & 0.66                             \\
		& 1600    & 1.45                         & 0.83                             & 0.85                             & 3.78                         & 0.93                             & 0.35                             & 5.33                         & 1.91                             & 0.47                             \\ \cmidrule[0.6pt](lr){1-2} \cmidrule[0.6pt](lr){3-5} \cmidrule[0.6pt](lr){6-8} \cmidrule[0.6pt](lr){9-11}
		0.2       & 400     & 2.76                         & 2.76                             & 2.69                             & 2.05                         & 1.13                             & 1.13                             & 5.22                         & 2.66                             & 2.04                             \\ 
		& 800     & 2.38                         & 1.87                             & 1.66                             & 4.14                         & 2.09                             & 0.63                             & 6.13                         & 2.89                             & 0.75                             \\
		& 1200    & 3.87                         & 2.63                             & 0.77                             & 3.82                         & 2.50                              & 0.35                             & 6.48                         & 3.62                             & 0.57                             \\
		& 1600    & 2.87                         & 2.05                             & 0.69                             & 4.05                         & 2.65                             & 0.17                             & 6.72                         & 3.60                              & 0.36                             \\ \cmidrule[0.6pt](lr){1-2} \cmidrule[0.6pt](lr){3-5} \cmidrule[0.6pt](lr){6-8} \cmidrule[0.6pt](lr){9-11}
		0.4       & 400     & 2.93                         & 2.48                             & 1.97                             & 3.32                         & 1.46                             & 0.55                             & 5.39                         & 3.49                             & 1.88                             \\
		& 800     & 2.46                         & 1.51                             & 0.40                              & 3.04                         & 1.35                             & 0.35                             & 5.63                         & 3.55                             & 0.80                              \\
		& 1200    & 2.16                         & 1.44                             & 0.21                             & 2.73                         & 1.38                             & 0.17                             & 6.48                         & 3.37                             & 0.65                             \\
		& 1600    & 1.52                         & 1.52                             & 0.31                             & 3.76                         & 1.59                             & 0.21                             & 6.57                         & 3.40                              & 0.33                             \\ \cmidrule[0.6pt](lr){1-2} \cmidrule[0.6pt](lr){3-5} \cmidrule[0.6pt](lr){6-8} \cmidrule[0.6pt](lr){9-11}
		0.6       & 400     & 1.99                         & 1.47                             & 1.58                             & 3.33                         & 0.88                             & 0.77                             & 4.65                         & 2.10                              & 1.08                             \\
		& 800     & 2.00                            & 0.47                             & 0.40                              & 3.55                         & 0.44                             & 0.26                             & 5.86                         & 1.84                             & 0.78                             \\
		& 1200    & 0.69                         & 0.62                             & 0.48                             & 3.48                         & 0.54                             & 0.23                             & 5.79                         & 1.54                             & 0.40                              \\
		& 1600    & 1.30                          & 0.34                             & 0.23                             & 3.23                         & 0.48                             & 0.22                             & 5.65                         & 1.52                             & 0.28                             \\ \cmidrule[0.6pt](lr){1-2} \cmidrule[0.6pt](lr){3-5} \cmidrule[0.6pt](lr){6-8} \cmidrule[0.6pt](lr){9-11}
		0.8       & 400     & 1.62                         & 0.75                             & 1.40                              & 2.18                         & 0.56                             & 0.29                             & 3.74                         & 1.48                             & 0.82                             \\
		& 800     & 0.51                         & 0.45                             & 0.29                             & 3.01                            & 0.42                             & 0.39                             & 4.84                         & 1.07                             & 0.40                              \\
		& 1200    & 1.15                         & 0.32                             & 0.33                             & 3.00                            & 0.40                              & 0.33                             & 5.02                         & 0.69                             & 0.34                             \\
		& 1600    & 1.46                         & 0.18                             & 0.34                             & 3.26                         & 0.23                             & 0.09                             & 5.15                         & 0.76                             & 0.25                             \\ \cmidrule[0.6pt](lr){1-2} \cmidrule[0.6pt](lr){3-5} \cmidrule[0.6pt](lr){6-8} \cmidrule[0.6pt](lr){9-11}
		0.95      & 400     & 0.51                         & 0.51                             & 1.15                             & 2.03                         & 0.39                             & 0.40                              & 3.82                         & 1.12                             & 0.89                             \\
		& 800     & 0.46                         & 0.36                             & 0.51                             & 2.78                         & 0.32                             & 0.24                             & 4.52                         & 0.43                             & 0.40                              \\
		& 1200    & 1.12                         & 0.30                              & 0.29                             & 2.50                          & 0.29                             & 0.11                             & 4.64                         & 0.44                             & 0.39                             \\
		& 1600    & 0.59                         & 0.15                             & 0.35                             & 2.99                         & 0.28                             & 0.14                             & 4.93                         & 0.25                             & 0.15                             \\ \cmidrule[1.0pt](lr){1-2} \cmidrule[1.0pt](lr){3-5} \cmidrule[1.0pt](lr){6-8} \cmidrule[1.0pt](lr){9-11}
		\multicolumn{2}{c}{Average} & 1.76                      & 1.17                             & 0.92                       & 3.16                   & 1.01                           & 0.43                          & 5.27                   & 1.95                       & 0.74     \\ \bottomrule[1.0pt]                 
	\end{tabular}

    \caption{Optimality gaps (in \%) of the inventory policies under different values of $\beta_1$ and $T$ for the calibrated choice models with four, eight, and sixteen bases, each with nine sizes.}
    \label{table:experiment_optimality_gap}
\end{table}

\paragraph{Comparison of inventory policies.} Overall, the three policies exhibit a clear performance ordering: \texttt{Greedy} yields the largest optimality gaps, followed by \texttt{CDLP-H}, while \texttt{CDLP-F} performs best across nearly all instances. In the sixteen-base case, their average optimality gaps are $5.27\%$, $1.95\%$, and $0.74\%$, respectively. The greedy heuristic performs substantially worse and shows little improvement as $T$ increases. In some cases, its performance even deteriorates; for example, when $\beta_1=0.2$, its gap increases from $5.22\%$ at $T=400$ to $6.72\%$ at $T=1600$. In contrast, \texttt{CDLP-F} improves consistently with $T$: at $T=1600$, its gap is below $0.5\%$ for every value of $\beta_1$. This highlights the effectiveness of the CDLP-based approach: although the underlying joint linear program optimizes both inventory and dynamic assortment decisions, its inventory solution remains highly effective when applied to the show-all setting considered in this paper.

The performance of \texttt{CDLP-H} lies between those of \texttt{Greedy} and \texttt{CDLP-F}. This policy is closely related to the approximation framework underlying Theorems~\ref{thm:one_half_fluid} and~\ref{thm:one_half_stochastic}, which provide nearly $1/2$ and asymptotic $1/2$ guarantees in the fluid and stochastic settings, respectively. Despite these conservative worst-case guarantees, \texttt{CDLP-H} performs substantially better in practice, with a gap below $4\%$ in every instance and an average gap of only $1.95\%$ for the sixteen-base case.

Compared with \texttt{CDLP-F}, however, \texttt{CDLP-H} benefits less consistently as $T$ increases. This limitation is particularly visible under moderate levels of size substitution. For example, in the sixteen-base case with $\beta_1=0.4$, its optimality gap remains around $3.5\%$ as $T$ increases, whereas the gap of \texttt{CDLP-F} decreases from $1.88\%$ to $0.33\%$. One explanation is that \texttt{CDLP-H} relies more heavily on the greedy augmentation procedure to fill unused capacity, which may prevent it from fully capturing the benefits of the fluid solution as the problem scales. Nevertheless, its optimality gap remains below $4\%$ across all instances, substantially outperforming its conservative worst-case guarantee.

\paragraph{Impact of size substitution.} The simpler policies appear to be particularly challenged by moderate levels of size substitution. In the sixteen-base case, the optimality gaps of both \texttt{Greedy} and \texttt{CDLP-H} are generally larger for intermediate values of $\beta_1$, particularly $\beta_1=0.2$ and $0.4$, than when substitution is strong. When $\beta_1=0.4$, the gap of \texttt{CDLP-H} remains around $3.5\%$, whereas under strong substitution ($\beta_1=0.95$), it decreases from $1.12\%$ at $T=400$ to only $0.25\%$ at $T=1600$.

This pattern suggests that moderate substitution creates a greater challenge for policies that do not fully exploit substitution across all sizes simultaneously. When substitution is weak, cross-size interactions play a limited role in inventory decisions; when substitution is strong, adjacent sizes can compensate for stockouts. At intermediate levels, however, substitution is consequential but cannot fully mitigate inventory imbalances, making it more important to account for substitution interactions across sizes. This may explain why \texttt{CDLP-F}, which jointly optimizes inventory decisions across \emph{all} sizes, exhibits a more pronounced advantage over \texttt{CDLP-H} in this regime.

\paragraph{Impact of input size.} As the number of bases increases, the performance advantage of the CDLP-based policies becomes more pronounced. The average optimality gap of \texttt{Greedy} increases substantially from $1.76\%$ in the four-base (36-product) case to $3.16\%$ in the eight-base (72-product) case and $5.27\%$ in the sixteen-base (144-product) case. In contrast, \texttt{CDLP-F} remains effective, with average gaps of $0.92\%$, $0.43\%$, and $0.74\%$, respectively. \texttt{CDLP-H} also performs well as the number of bases increases, although its average gap rises from $1.17\%$ in the four-base case to $1.95\%$ in the sixteen-base case. Overall, these results suggest that CDLP-based inventory planning becomes increasingly valuable relative to the greedy heuristic as the problem size grows.

\paragraph{Runtime.} Despite the large product space of nearly 150 products, the computational runtimes of both \texttt{CDLP-H} and \texttt{CDLP-F} remain very small. Excluding the greedy augmentation step, the CDLPs in both policies, which are solved via constraint generation, require less than two seconds across all instances in Table~\ref{table:experiment_optimality_gap} on a MacBook Pro with an Apple M2 Pro and 16 GB of memory.

\section{Conclusion}\label{sec:conclusions}

Motivated by size substitution in apparel retailing, this paper develops the CFTC model to capture two-dimensional demand substitution across both product bases and sizes, and studies the resulting assortment and inventory optimization problems. For assortment optimization, we establish computational hardness and develop a PTAS under bounded size substitution. For inventory optimization, we develop approximation algorithms in both fluid and stochastic settings by exploiting the structure induced by stocking alternating sizes, the CCD property of the mixed-NP-MNL model, and a new coupling technique that connects the stochastic inventory process to its fluid counterpart. Our numerical experiments using instances calibrated from footwear data further demonstrate strong performance across a range of substitution patterns and problem settings.

Beyond the specific application domain, our results provide new structural and algorithmic tools for inventory optimization under dynamic substitution. In particular, the CCD result for the mixed-NP-MNL model and the coupling technique for bridging fluid and stochastic inventory processes may be useful for studying inventory optimization under other choice models.

\makeatletter
\newcommand*\mysize{%
\@setfontsize\mysize{10.0}{11.0}%
}
\makeatother

\renewcommand{\bibfont}{\mysize}
{\setlength{\bibsep}{3pt}
\bibliographystyle{plainnat}
\bibliography{bib_size-sub.bib}
}

\ECSwitch

\ECHead{Electronic Companion}

\begin{appendices}

\section{Proof of Theorem~\ref{thm:AO_hardness}}
\label{appendix-subsec:proof_of_AO_hardness}

We utilize a reduction from the \emph{Partition} problem, one of Karp’s original 21 NP-complete problems~\citeyearpar{karp1972reducibility}. In this setting, we are given as input a collection of nonnegative integers $a_1,\ldots,a_n$. Letting $L=\sum_{i\in\Ncal} a_i$, the goal is to decide whether there exists a subset $Q\subseteq\Ncal$ such that $\sum_{i\in Q} a_i = \frac{L}{2}$. Given an instance of this form, we construct a corresponding instance of~\ref{eqn:AO_CFTC}, phrased as a feasibility problem, as follows:
\begin{itemize}
\item {\em Product universe and revenues.}
There are $n+1$ bases indexed by $\Ncal=\{1,\ldots,n+1\}$ and two sizes indexed by $\Scal=\{1,2\}$. The available sizes for each base $i\in\Ncal$ are given by
\[
\Scal_i =
\begin{cases}
    \{1\}, & \text{if } i \in \Ncal,\\
    \{2\}, & \text{if } i = n+1.
\end{cases}
\]
Consequently, the product universe is $\Ucal = \{(1,1),(2,1),\ldots,(n,1),(n+1,2)\}$. The revenues for bases $i\in\Ncal$ are all set to one, i.e., $r_{1,1} = r_{2,1} = \cdots = r_{n,1} = 1$, while the revenue of base $n+1$ is set to $r_{n+1,2} = \frac{371}{145}$.
\item {\em Customer types and weights.}
There are two customer types indexed by $\Gcal=\{1,2\}$, each arriving with equal probability, so that $\mu_1=\mu_2=\frac{1}{2}$. The first customer type is characterized by $(s_1,\pi_1)=(1,\uparrow)$, while the second customer type is characterized by $(s_2,\pi_2)=(2,\downarrow)$. For each $g\in\Gcal$, the weights associated with bases $i\in\Ncal$ satisfy $w_{i,g}=\frac{a_i}{L}$, while the weight for base $n+1$ is $w_{n+1,g}=1$. Moreover, the two customer types share the same disutility factors, with $\beta_{0,g}=1$ and $\beta_{1,g}=\frac{1}{2}$, noting that with two sizes the maximum deviation from the ideal size is one.
\item {\em Revenue threshold.}
Finally, we fix the expected revenue threshold to be $K=\frac{31}{29}$.
\end{itemize}
The remainder of this proof is devoted to showing that there exists an assortment $A$ that  garners an expected revenue of at least $K$ if and only if a valid partition exists. Moving forward, we encode assortment decisions using a binary vector $x \in \{0,1\}^{n+1}$, where $x_i = 1$ indicates that base $i$ is offered. Since each base can appear in only a single size, there is a one-to-one correspondence between assortments $A \subseteq \Ucal$ and their encodings as binary vectors $x \in \{0,1\}^{n+1}$. 

The expected revenue earned from offering assortment $x$ can be expressed as
\begin{eqnarray*}\label{eqn:hardness_rev}
    \Rcal(x) = \frac{1}{2} \cdot \left( \frac{w(x) + \frac{1}{2}\cdot\frac{371}{145}x_{n+1}}{1 + w(x) + \frac{1}{2}\cdot x_{n+1}} \right) + \frac{1}{2} \cdot \left( \frac{\frac{1}{2} \cdot w(x) + \frac{371}{145}x_{n+1}}{1 + \frac{1}{2}\cdot w(x) + x_{n+1}} \right),
\end{eqnarray*}
where $w(x) =\sum_{i \in \Ncal}w_{i,1}\cdot x_i$.  It is straightforward to verify that the expected revenue of any assortment $x$ with $x_{n+1}=0$ can be improved by setting $x_{n+1}=1$, since product $(n+1,2)$ is the largest revenue product.  Given this observation, our \ref{eqn:AO_CFTC} feasibility problem reduces to the question of whether there exists an assortment $x$ with $x_{n+1}=1$ that satisfies
\[
 \frac{1}{2} \cdot \left( \frac{w(x) + \frac{1}{2}\cdot\frac{371}{145}}{\frac{3}{2} + w(x) } \right) + \frac{1}{2} \cdot \left( \frac{\frac{1}{2} \cdot w(x) + \frac{371}{145}}{2 + \frac{1}{2}\cdot w(x)} \right) \geq \frac{31}{29}
\]
The inequality above can be equivalently written as 
\[
\frac{-4\cdot (w(x) - \frac{1}{2})^2}{29\cdot (2 w(x)^2 + 11w(x)  +12)} \geq 0
\]
which occurs if only if $w(x) = \frac{1}{2}$. Recalling that $w_{i,1} = \frac{a_i}{L}$ for every $i \in \Ncal$, it easy to see that the condition $w(x)=\frac{1}{2}$ is identical to having a set $Q \subseteq \Ncal$ with $\sum_{i \in Q} a_i = \frac{L}{2}$. 

\section{The Details of our Approximation Scheme for \ref{eqn:AO_CFTC}}\label{app-sec:AO}

  For ease of exposition, we focus on the unconstrained variant of this problem. However, our proposed approach extends naturally to cardinality-constrained variants.  Throughout the remainder of this section, we assume without loss of generality that $\frac{1}{\eps}$ is an integer larger than 2, and that the number of sizes $m$ is an integer multiple of $\frac{\ell_{\max}}{\eps}$. Moreover, we use $\opt$ to denote the optimal objective value of~\ref{eqn:AO_CFTC}.

\subsection{Preliminaries}\label{subsec:AO_prelims}

Before formalizing the details of our two algorithmic steps, we introduce the notions of a \emph{sparse size block} and a \emph{well-separated assortment}, which together lay the groundwork for the technical exposition that follows. We also formalize the precise nature of the assortment oracle. 

\paragraph{Sparse size block.} A subset of sizes $B \subseteq \Scal$ is referred to as a \emph{size block} if it forms a consecutive interval of sizes; that is, there exist $\ubar{s}, \bar{s} \in \Scal$ such that $B = [\ubar{s}, \bar{s}]$. A size block $B$ is deemed \emph{sparse} if $|B| \leq \frac{2\ell_{\max}}{\eps}$. As a shorthand, we will refer to size blocks as blocks and sparse size blocks and sparse blocks. For any block $B$, sparse or otherwise, we let
$
\Ucal(B) = \{(i,s) : i \in \Ncal,\; s \in B \cap \Scal_i\}
$
denote the set of products that can be offered at sizes within block $B$. For an assortment $A \subseteq \Ucal$, we further define
$
A(B) = A \cap \Ucal(B)
$
to be the projection of $A$ onto the product universe $\Ucal(B)$.
\begin{itemize}
	\item \emph{$\eps$-uniform sparse blocks}: It will be useful to define the following partition of the size set $\Scal$ into sparse blocks of equal size. To this end, and recalling that we assume $m$ is an integer multiple of $\frac{\ell_{\max}}{\epsilon}$, 
	we define the $\epsilon$-uniform sparse blocks as
	\[
	B_{\epsilon}(k)
	=
	\left[(k-1)\cdot \frac{\ell_{\max}}{\epsilon} + 1,\;
	k\cdot \frac{\ell_{\max}}{\epsilon}\right]
	\]
	for $k \in [K_{\eps}]$, where $K_{\eps} = \frac{m\eps}{\ell_{\max}}$. 
	It is straightforward to verify that
	\[
	\Scal = \biguplus_{k \in [K_{\eps}]} B_{\epsilon}(k),
	\]
	where $\biguplus$ denotes the disjoint union operator, and each such uniform block spans exactly $\frac{\ell_{\max}}{\epsilon}$ sizes. 
\end{itemize}

\paragraph{Well-separated assortments.} An assortment $A \subseteq \Ucal$ is said to be \emph{well-separated} if there exist $Q \in \mathbb{Z}$ pairwise disjoint size blocks $B_1,\ldots,B_Q$ such that
$
A = \biguplus_{q \in [Q]} A(B_q).
$
Moreover, indexing these blocks in increasing order of size so that
\[
\max B_1 < \max B_2 < \cdots < \max B_Q,
\]
we require that the blocks be sufficiently separated in the sense that, for any $q \in [Q-1]$,
\[
\min B_{q+1} > \max B_q + \ell_{\max} .
\]
That is, between any two successive blocks, at least $\ell_{\max}$ sizes are skipped. A key feature of any well-separated assortment $A$ is that its expected revenue is decomposable across blocks. In particular,
$
\rev(A) = \sum_{q \in [Q]} \rev\!\left(A(B_q)\right),
$
since each size block is separated from the others by at least $\ell_{\max}$ skipped sizes.

\paragraph{The black-box oracle.}  For an arbitrary sparse block $B$, consider the assortment problem
\begin{equation}
	\label{eqn:AO_oracle}
	\tag{AO-Oracle}
	\opt(B) = \max_{A \subseteq \Ucal(B)} \rev(A),
\end{equation}
where we seek the expected revenue maximizing assortment that uses only products with sizes from block $B$.  It is important to emphasize, in~\ref{eqn:AO_oracle}, we consider all customer types, some which might have an ideal size outside of $B$.  For the entirety of the first step, for any sparse block $B$,  we assume oracle access to an assortment $\cal{O}(B) \subseteq \Ucal(B)$, whose expected revenue satisfies $\rev(\cal{O}(B)) \geq \alpha \cdot \opt(B)$ for some $\alpha \in [0,1]$.

\subsection{Step 1: The size-skipping dynamic program}\label{subsec:step_1_AO}

In this section, we show how to use dynamic programming to compute a $\alpha\cdot (1-\eps)$-optimal assortment for~\ref{eqn:AO_CFTC}, given access to an assortment oracle that returns an $\alpha$-approximate solution to~\ref{eqn:AO_oracle}.

\paragraph{State description.}  Each state $(k, \ubar{s}, \bar{s})$ of our dynamic program consists of the following three ingredients:
\begin{itemize}
	\item An index $k \in [K_{\eps}]_0$, where $k=0$ will correspond to the initial state, and $k \in [K_{\eps}]$ an $\eps$-uniform sparse block.
	\item Two sizes $(\ubar{s}, \bar{s}) \in \Fcal(k)$, where
	\[
	\Fcal(k) =
	\begin{cases}
		\{(0,0)\} & \text{if}~k=0\\
		\{(s_1,s_2) \in B_{\eps}(k): |[s_1,s_2]| = \ell_{\max}\}  & \text{otherwise}.
	\end{cases}
	\]
	When $k \in [K_{\eps}]$, the set $\Fcal(k)$ denotes for all blocks strictly contained within $B_{\eps}(k)$ that span exactly $\ell_{\max}$ sizes.  As the recursion of our dynamic program will reveal, the block $[\ubar{s}, \bar{s}]$ will be skipped.
\end{itemize}

\paragraph{Value function.} The value function is most naturally interpreted in the case $\alpha = 1$, meaning that the oracle returns an optimal assortment for~\ref{eqn:AO_oracle}. In this setting, $\Vcal(k,\ubar{s},\bar{s})$ represents the maximum expected revenue that can be obtained from a well-separated assortment $A \subseteq \biguplus_{\kappa \in [k, K_{\eps}]} \Ucal(B_{\eps}(k))$ that, from each $\eps$-uniform sparse block $B_{\eps}(k), \ldots, B_{\eps}(K_{\eps})$, skips a single block containing at least $\ell_{\max}$ sizes. Moreover, the skipped block from $B_{\eps}(k)$ must be a superset of $[\ubar{s}, \bar{s}]$. Formally, and generalizing beyond the $\alpha = 1$ case, we have
\begin{equation}
	\label{eqn:AO_value_functions_outer}
	\Vcal(k,\ubar{s},\bar{s})=\displaystyle \max_{(\ubar{s}',\bar{s}') \in \Fcal(k+1)}
	\left\{
	\rev\!\left(\cal{O}\bigl([\bar{s}+1,\ubar{s}'-1]\bigr)\right)
	+
	\Vcal\!\left(k+1,\ubar{s}',\bar{s}'\right)
	\right\},
\end{equation}
where we use the convention that $\Ocal([s_1,s_2]) = \emptyset$ if $s_1 > s_2$. The base cases are given by
\[
\Vcal(K_{\eps},\ubar{s},\bar{s})
=
\rev\!\left(\cal{O}\bigl([\bar{s}+1,m]\bigr)\right).
\]
In essence, the recursion in~\eqref{eqn:AO_value_functions_outer} proceeds sequentially over the $\eps$-uniform sparse blocks, selecting for each a block $[\ubar{s}, \bar{s}]$ containing exactly $\ell_{\max}$ sizes to skip, and then applying the assortment oracle to the intervening non-skipped blocks. It is straightforward to verify that these non-skipped blocks are indeed sparse, and hence the application of the oracle to each such block is valid. Specifically, at any state $(k,\ubar{s},\bar{s})$ with $k \in [K_{\eps}-1]$, we select $(\ubar{s}',\bar{s}') \in \Fcal(k+1)$ and apply the assortment oracle to the block $[\bar{s}+1,\ubar{s}'-1]$. By construction, we have $\bar{s} \in B_{\eps}(k)$ and $\ubar{s}' \in B_{\eps}(k+1)$, and thus
\[
|[\bar{s}+1,\ubar{s}'-1]|
\le
\max B_{\eps}(k+1) - \min B_{\eps}(k)
\le
\frac{2\ell_{\max}}{\eps}.
\]
Moreover, if $k=0$, we have $\bar{s}=0$ and $\ubar{s}' \in B_{\eps}(1)$, and so
\[
|[1,\ubar{s}'-1]|
\le
|B_{\eps}(1)| 
=
\frac{\ell_{\max}}{\eps}.
\]
Finally, for states $(K_{\eps},\ubar{s},\bar{s})$, we apply the assortment oracle to the block $[\bar{s}+1,m] \subseteq B_{\eps}(K_{\eps})$, which therefore contains fewer than $\frac{\ell_{\max}}{\eps}$ sizes.

\paragraph{Analysis.} Let
\[
(0,\ubar{s}^{(0)},\bar{s}^{(0)})
\quad \xrightarrow{(\ubar{s}^{(1)}, \bar{s}^{(1)})} \quad
(1,\ubar{s}^{(1)}, \bar{s}^{(1)})
\quad \xrightarrow{(\ubar{s}^{(2)}, \bar{s}^{(2)})} \quad
\cdots
\quad \xrightarrow{(\ubar{s}^{(K_{\eps})}, \bar{s}^{(K_{\eps})})} \quad
(K_{\eps}, \ubar{s}^{(K_{\eps})}, \bar{s}^{(K_{\eps})})
\]
denote the sequence of states and actions obtained by following the dynamic program in~\eqref{eqn:AO_value_functions_outer} starting from the initial state $(0,\ubar{s}^{(0)},\bar{s}^{(0)}) = (0,0,0)$. Moreover, for $k \in [K_{\eps}+1]$, let $\hat{B}^{(k)} = [\bar{s}^{\YC{(k-1)}}+1, \ubar{s}^{(k)}-1]$ denote the sparse block to which the assortment oracle is applied, where we set $\ubar{s}^{(K_{\eps}+1)} = m+1$. Let
\[
\hat{\Ocal}
=
\biguplus_{k \in [K_{\eps}+1]}
\Ocal\bigl(\hat{B}^{(k)}\bigr)
\]
denote the amalgamation of the assortments returned by the oracle. The following lemma, whose proof is provided in Appendix~\ref{app:proof_O_rev_performance}, establishes the desired revenue performance guarantee for $\hat{\Ocal}$.

\begin{lemma}
	\label{lem:O_rev_performance}
	$
	\rev(\hat{\Ocal})
	\;\ge\;
	\alpha \cdot (1-\eps) \cdot \opt.
	$
\end{lemma}
The proof of Lemma~\ref{lem:O_rev_performance} proceeds in two parts.
First, we show that $\hat{\Ocal}$ is a well-separated assortment, which implies that
\[
\Vcal(0,0,0)
=
\sum_{k \in [K_{\eps}+1]}
\rev\!\left(\Ocal\bigl(\hat{B}^{(k)}\bigr)\right)
=
\rev(\hat{\Ocal}).
\]
Second, we argue that
\[
\Vcal(0,0,0)
\;\ge\;
\alpha \cdot (1-\eps) \cdot \opt
\]
by constructing a feasible sequence of states and actions for the dynamic program in~\eqref{eqn:AO_value_functions_outer} that accumulates a total value of at least $\alpha \cdot (1-\eps) \cdot \opt$.

\paragraph{Overall running time.} In what follows, we show that the assortment $\hat{\Ocal}$ can be recovered in an overall running time of $O\!\left(\frac{m\ell_{\max}}{\eps}\cdot Oracle\right)$, where $Oracle$ denotes the time required for a single call to the assortment oracle. To this end, we first argue that the total number of potential states $(k,\ubar{s},\bar{s})$ in our dynamic program is only $O(m)$. This follows from the observations that:
\begin{itemize}
	\item The index $k$ ranges over $[K_{\eps}+1]$, where $K_{\eps} = \frac{m\eps}{\ell_{\max}}$.
	\item The size tuple $(\ubar{s},\bar{s})$ is uniquely determined by $\ubar{s}$, since $\bar{s}=\ubar{s}+\ell_{\max}-1$. For each $k \in [K_{\eps}+1]$, there are $\frac{\ell_{\max}}{\eps}$ possibilities for $\ubar{s} \in B_{\eps}(k)$.
\end{itemize}
Consequently, the total number of states is at most
\[
(K_{\eps}+1)\cdot \frac{\ell_{\max}}{\eps}
=
O\!\left(\frac{m\eps}{\ell_{\max}}\cdot \frac{\ell_{\max}}{\eps}\right)
=
O(m).
\]

Finally, the optimal choice of sizes $(\ubar{s}',\bar{s}') \in \Fcal(k+1)$ from any state $(k,\ubar{s},\bar{s})$ can be recovered by complete enumeration over all $O\!\left(\frac{\ell_{\max}}{\eps}\right)$ possibilities for $\ubar{s}' \in B_{\eps}(k+1)$, again noting that $\bar{s}'=\ubar{s}'+\ell_{\max}-1$. Therefore, accounting for oracle calls, the overall running time of our approach is $O\!\left(\frac{m\ell_{\max}}{\eps}\cdot Oracle\right)$.

\subsection{Step 2: The assortment oracle} \label{subsec:oracle}

In this section, we give shape to the assortment oracle exploited throughout Step~1 of our approach. In particular, the following theorem formalizes our algorithmic result for~\ref{eqn:AO_oracle}. The proof is provided in Appendix~\ref{app:assort_Oracle_thm}.

\begin{theorem}
	\label{thm:assort_Oracle}
	For any accuracy level $\eps > 0$, \ref{eqn:AO_oracle} admits a $(1-\eps)$-approximation in a running time of $O\!\left(\left(\frac{n\ell_{\max}}{\eps}\right)^{O\left(\frac{\ell_{\max}^2}{\eps}\right)}\right)$.
\end{theorem}

Several remarks help clarify the implications of this result. First, Theorem~\ref{thm:AO_hardness} establishes that~\ref{eqn:AO_oracle} is NP-hard, and thus a polynomial-time approximation scheme is essentially the strongest achievable algorithmic guarantee under standard complexity assumptions. Second, Theorem~\ref{thm:assort_Oracle} justifies invoking the assortment oracle in Step~1 with $\alpha = 1-\eps$. Finally, substituting $Oracle = O\!\left(\left(\frac{n\ell_{\max}}{\eps}\right)^{O\left(\frac{\ell_{\max}^2}{\eps}\right)}\right)$ into the running time derived at the end of Appendix~\ref{subsec:step_1_AO} yields the overall complexity stated in Theorem~\ref{thm:main_theorem_assort}.

\subsection{Proof of Lemma~\ref{lem:O_rev_performance}}\label{app:proof_O_rev_performance}

Throughout this proof, we use $A^*$ to denote an optimal solution to~\ref{eqn:AO_CFTC}, with $\opt = \rev(A^*)$ denoting its expected revenue. Before proceeding with the proof of the lemma, we introduce additional notation that, for any assortment $A \subseteq \Ucal$ and size block $B \subseteq \Scal$, provides an explicit measure of the expected revenue contribution of the products in $A(B)$. Equipped with this notation, we show how to remove products from $A^*$ to construct a well-separated assortment $A^*_{WS}$, which we refer to as the \emph{well-separated proxy}. We then state two intermediate claims that together enable a streamlined proof of the lemma. Together, these claims establish that both $\hat{\Ocal}$ and $A^*_{WS}$ are well-separated assortments, and that the products removed from $A^*$ in forming $A^*_{WS}$ contribute at most an $\eps$-fraction of $\opt$.

\paragraph{Block revenue contribution.}  For any assortment $A \subseteq \Ucal$ and size block $B \subseteq \Scal$, let 
\[
R_B(A) = \sum_{(i,s) \in  A(B)} r_{i,s}\cdot \pi((i,s),A)
\]
denote the expected revenue contribution of product $A(B)$ when the assortment $A$ is offered, where $$\pi((i,s),A) = \sum_{\substack{g \in \Gcal:\\ (i,s) \in C_g(A)}} \mu_g\cdot \pi_g((i,s),A)$$ is the total choice probability of product $(i,s) \in A$.

\paragraph{The well-separated proxy.}  We start by introducing a further partition of each $\eps$-uniform sparse block into $\frac{1}{\eps}$ sub-blocks, each containing exactly $\ell_{\max}$ sizes. Formally, for each $k \in [K_{\eps}]$ and $q \in [\frac{1}{\eps}]$, let
\[
B_{\eps,q}(k)
=
\left[
(k-1)\cdot \frac{\ell_{\max}}{\eps}
+
(q-1)\cdot \ell_{\max}
+
1,
\;
(k-1)\cdot \frac{\ell_{\max}}{\eps}
+
q\cdot \ell_{\max}
\right]
\]
denote the $q$-th such sub-block of $B_{\eps}(k)$. With this notation in hand, define
\[
q^*_k
=
\argmin_{q \in [\frac{1}{\eps}]}
R_{B_{\eps,q}(k)}(A^*)
\]
to be the sub-block of $B_{\eps}(k)$ with the smallest expected revenue contribution under $A^*$. Using the shorthand $B_{\eps,q^*(k)} = B_{\eps,q^*_k}(k)$ for notational convenience, the well-separated proxy is defined as
\[
A^*_{WS}
=
A^*
\setminus
\left(
\biguplus_{k \in [K_{\eps}]}
\Ucal\bigl(B_{\eps,q^*(k)}\bigr)
\right),
\]
meaning that we delete from $A^*$ every product $(i,s) \in A^*$ such that
$s \in \biguplus_{k \in [K_{\eps}]} B_{\eps,q^*(k)}$. While this definition of $A^*_{WS}$ is conceptually transparent, it is technically convenient to express it via projections of $A^*$ onto the non-skipped blocks. To this end, let
\[
\ubar{s}^*_k
=
\min B_{\eps,q^*(k)}
\qquad\text{and}\qquad
\bar{s}^*_k
=
\max B_{\eps,q^*(k)},
\]
so that $B_{\eps,q^*(k)} = [\ubar{s}^*_k,\bar{s}^*_k]$ for each $k \in [K_{\eps}]$. Also define $\bar{s}^*_0 = 0$ and $\ubar{s}^*_{K_{\eps}+1} = m+1$. From here, define the intervening (non-skipped) blocks
\[
B^*_k
=
[\bar{s}^*_{k-1}+1,\;\ubar{s}^*_k-1]
\qquad\text{for } k \in [K_{\eps}+1],
\]
which correspond to the blocks retained when transitioning from $A^*$ to $A^*_{WS}$. Accordingly,
\begin{equation}
\label{eqn:well_sep_proxy}
A^*_{WS}
=
\biguplus_{k \in [K_{\eps}+1]}
A^*(B^*_k)
\end{equation}
provides the alternative representation of the well-separated proxy.

\paragraph{Intermediate claims.} In what follows, we state two intermediate claims, whose proofs can be found in Appendices~\ref{subsec:proof_props_A_WS} and~\ref{subsec:proof_props_O_hat}.  The first claim establishes that $A^*_{WS}$ is indeed well-separated and that its expected revenue is at least a $(1-\eps)$-fraction of $\opt$.  The second establishes that $\hat{\Ocal}$ is well-separated.
\begin{claim}
    \label{claim:A_WS_properties}
    The assortment $A^*_{WS}$ is (i) well-separated and (ii) garners an expected revenue of at least $(1-\eps)\cdot \opt$.
\end{claim}
\begin{claim}
    \label{claim:O_WS}
    The assortment $\hat{\Ocal}$ is well-separated.
\end{claim}

\paragraph{Proof of lemma.} Since $\hat{\Ocal}$ is a well-separated assortment by Claim~\ref{claim:O_WS}, we have
\[
\rev(\hat{\Ocal})
=
\sum_{k \in [K_{\eps}+1]} \rev\!\left(\Ocal(\hat{B}^{(k)})\right)
=
\Vcal(0,0,0),
\]
and hence it remains only to show that $\Vcal(0,0,0) \geq \alpha \cdot (1-\eps)\cdot \opt$. To establish this inequality, we construct a feasible sequence of states and actions for the dynamic program in~\eqref{eqn:AO_value_functions_outer} derived from the well-separated proxy $A^*_{WS}$, and then lower bound the total value accumulated along this sequence.

Specifically, define
\[
(0,0,\bar{s}^{*}_0)
\quad \xrightarrow{(\ubar{s}^{*}_1, \bar{s}^{*}_1)} \quad
(1,\ubar{s}^{*}_1, \bar{s}^{*}_1)
\quad \xrightarrow{(\ubar{s}^{*}_2, \bar{s}^{*}_2)} \quad
\cdots
\quad \xrightarrow{(\ubar{s}^{*}_{K_{\eps}}, \bar{s}^{*}_{K_{\eps}})} \quad
(K_{\eps}, \ubar{s}^{*}_{K_{\eps}}, \bar{s}^{*}_{K_{\eps}})
\]
to be this sequence. It is feasible because $\ubar{s}^*_k,\bar{s}^*_k \in B_{\eps}(k)$ and $|[\ubar{s}^*_k,\bar{s}^*_k]|=\ell_{\max}$ by construction, and thus $(\ubar{s}^*_k,\bar{s}^*_k) \in \Fcal(k)$. Moreover, by following this sequence, we apply the assortment oracle to the well-separated blocks $B^*_1,\ldots,B^*_{K_{\eps}+1}$, yielding the well-separated assortment
\[
\Ocal^*
=
\biguplus_{k \in [K_{\eps}+1]} \Ocal(B^*_k)
\]
with total value
\[
\Vcal^*
=
\sum_{k \in [K_{\eps}+1]} \rev\!\left(\Ocal(B^*_k)\right)
=
\rev(\Ocal^*).
\]
Since this sequence is feasible within the recursion in~\eqref{eqn:AO_value_functions_outer}, we have
\begin{eqnarray}
    \Vcal(0,0,0) \geq \Vcal^*
    &=&
    \sum_{k \in [K_{\eps}+1]} \rev\!\left(\Ocal(B^*_k)\right)
    \nonumber \\
    &\geq&
    \alpha \cdot \sum_{k \in [K_{\eps}+1]} \rev\!\left(A^*_{WS}(B^*_k)\right)
    \label{eqn:2a}\\
    &=&
    \alpha \cdot \rev(A^*_{WS})
    \label{eqn:2b}\\
    &\geq&
    \alpha \cdot (1-\eps)\cdot \opt,
    \label{eqn:2c}
\end{eqnarray}
where~\eqref{eqn:2a} follows because $A^*_{WS}(B^*_k)$ is a feasible solution to~\ref{eqn:AO_oracle} with $B=B^*_k$, while~\eqref{eqn:2b} and~\eqref{eqn:2c} follow from properties (i) and (ii) of Claim~\ref{claim:A_WS_properties}, respectively.

\subsection{Proof of Claim~\ref{claim:A_WS_properties} } \label{subsec:proof_props_A_WS}
We prove the two properties sequentially below.

\paragraph{Property (i).} To show that $A^*_{WS}$ is well-separated, it suffices to show that for $k \in [K_{\eps}]$, we have
\[
\min B^*_{k+1} > \max B^*_{k} + \ell_{\max}.
\]
To show the above inequality, observe that
\[
\min B^*_{k+1}  = \bar{s}^*_k+1 = \ubar{s}^*_k+ \ell_{\max} > \max B^*_k+ \ell_{\max},
\]
where the final inequality follows since $\max B^*_k = \ubar{s}^*_k -1$ by construction.

\paragraph{Property (ii).} To begin, for ease of notation, define
\[
\beta_k = B_{\eps}(k) \setminus B_{\eps,q^*(k)}
\qquad \text{for each } k \in [K_{\eps}],
\]
to be the non-skipped portion of each $\eps$-uniform sparse block in constructing $A^*_{SW}$.  We then have
\begin{eqnarray}
    \rev(A^*)
    &=&
    \sum_{k \in [K_{\eps}]}
    \left(
        R_{\beta_k}(A^*)
        +
        R_{B_{\eps,q^*(k)}}(A^*)
    \right)
    \nonumber \\
    &\le&
    \frac{1}{1-\eps}\cdot
    \sum_{k \in [K_{\eps}]}
    R_{\beta_k}(A^*)
    \label{eqn:1a} \\
    &=&
    \frac{1}{1-\eps} \cdot
    \sum_{k \in [K_{\eps}]}
    \sum_{(i,s) \in A^*(\beta_k)}
    r_{i,s}\cdot
    \pi((i,s),A^*)
    \nonumber \\
    &=&
    \frac{1}{1-\eps} \cdot
    \sum_{k \in [K_{\eps}]}
    \sum_{(i,s) \in A^*_{WS}(\beta_k)}
    r_{i,s}\cdot
    \pi((i,s),A^*)
    \label{eqn:1b} \\
    &\le&
    \frac{1}{1-\eps} \cdot
    \sum_{k \in [K_{\eps}]}
    \sum_{(i,s) \in A^*_{WS}(\beta_k)}
    r_{i,s}\cdot
    \pi((i,s),A^*_{WS})
    \label{eqn:1c} \\
    &=&
    \frac{1}{1-\eps} \cdot
    \sum_{k \in [K_{\eps}]}
    R_{\beta_k}(A^*_{WS})
    \nonumber \\
    &=&
    \frac{1}{1-\eps} \cdot
    \rev(A^*_{WS}),
    \nonumber
\end{eqnarray}
as desired. Above, inequality~\eqref{eqn:1a} follows from the observation that
\[
R_{B_{\eps,q^*(k)}}(A^*) \leq \eps \cdot \left( R_{\beta_k}(A^*)
+
R_{B_{\eps,q^*(k)}}(A^*) \right)
\]
by the construction of $B_{\eps,q^*(k)}$. Equality~\eqref{eqn:1b} holds because $\biguplus_{k \in [K_{\eps}+1]} B^*_k = \biguplus_{k \in [K_{\eps}]} \beta_k$, and so for any $k \in [K_{\eps}]$, we have
$A^*(\beta_k) = A^*_{WS}(\beta_k)$. Finally, inequality~\eqref{eqn:1c}
follows from the regularity of the CFTC model, as established in
Appendix~\ref{appendix-sec:regularity_property}. In particular,
for assortments $A,A' \subseteq \Ucal$ with $A \subseteq A'$, we have
$\pi((i,s),A') \ge \pi((i,s),A)$ for all $(i,s) \in A'$.

\subsection{Proof of Claim~\ref{claim:O_WS}.}  \label{subsec:proof_props_O_hat}

Recalling that $\hat{\Ocal} = \sum_{k \in [K_\eps+1]}\Ocal(\hat{B}^{(k)})$, to show that $\hat{\Ocal}$ is well-separated requires showing that
\[
\min \hat{B}^{(k+1)} > \max \hat{B}^{(k)} + \ell_{\max}
\]
for each $k \in [K_{\eps}]$.  For this purpose, observe that
\[
\min \hat{B}^{(k+1)}  = \bar{s}^{(k)}+1 = \ubar{s}^{(k)} + \ell_{\max} > \max \hat{B}^{(k)} + \ell_{\max},
\]
where the final inequality follows since $\max \hat{B}^{(k)} = \ubar{s}^{(k)} -1$ by construction.

\section{Proof of Theorem~\ref{thm:assort_Oracle}} \label{app:assort_Oracle_thm}

In this section, we show how to approximate~\ref{eqn:AO_oracle} within factor $1-\eps$ of optimal in a running time of $O\!\left(\left(\frac{n\ell_{\max}}{\eps}\right)^{O\left(\frac{\ell_{\max}^2}{\eps}\right)}\right)$. We will abuse notation slightly and use $A^*$ as the optimal solution to~\ref{eqn:AO_oracle}.  We begin by introducing additional notation that will ease the development of our approach, before presenting a high-level outline of its two steps, which are fully developed in the sections that follow.

\paragraph{Relevant customer types.}  Since~\ref{eqn:AO_oracle} considers the restricted product set $\Ucal(B)$ for some sparse block $B$, there may be customer types whose consideration sets are empty regardless of the assortment offered. Accordingly, define
\[
\Gcal(B)
=
\{\, g \in \Gcal : C_g(\Ucal(B)) \neq \emptyset \,\}
\]
to be the collection of customer types that would have a non-empty consideration set if the full restricted universe $\Ucal(B)$ were offered. When solving~\ref{eqn:AO_oracle}, it suffices to account only for types $g \in \Gcal(B)$, since any type $g \in \Gcal \setminus \Gcal(B)$ will never make a purchase and therefore contributes zero expected revenue. It is worth noting that $|\Gcal(B)| = O\!\left(\frac{\ell_{\max}^2}{\eps}\right)$. Indeed, if $g \in \Gcal(B)$, then necessarily
$
s_g \in [\min B - \ell_{\max},\, \max B + \ell_{\max}],
$
since otherwise no product in $\Ucal(B)$ could enter the consideration set of type $g$. Because $B$ is sparse, this interval spans $O\!\left(\frac{\ell_{\max}}{\eps}\right)$ distinct size indices. Moreover, for each such size index, there are at most $\ell_{\max}$ possibilities for $\ell_g$ and two possibilities for $\tau_g$, yielding the stated bound.

\subsection{Technical outline} \label{app:step_2_outline}

At a high-level, our approximation scheme is comprised of two steps, which are summarized below.  We refer to these two steps as Step 2a and Step 2b since their introduction falls within the over-arching Step 2.

\paragraph{Step 2a: Guessing and rewards (Appendix~\ref{app:step_2a}).} For each customer type $g \in \Gcal(B)$, we first show how to efficiently compute an approximate estimate $\hat{w}_g$ of the total weight $w^*_g = w_g(C_g(A^*))$, that is, the total consideration-set weight induced by the optimal assortment $A^*$.  We then use this estimated total weight to define a proxy performance measure, which we refer to as the \emph{reward}. This reward mirrors the true expected revenue, except that for each type $g \in \Gcal(B)$, the denominator weight of the induced consideration set is fixed at $(1+\eps)\cdot \hat{w}_g$. This fixed-weight surrogate enables us to decouple assortment decisions across bases while maintaining a $(1-\eps)$-approximation guarantee.

\paragraph{Step 2b: The approximate dynamic program (Appendix~\ref{app:step_2b}).} With these guessed quantities in hand, we formulate a dynamic program that sequentially selects the sizes to offer for each base with the goal of maximizing the total accumulated reward. The state space of this dynamic program stores rounded versions of the accumulated total weight for each type’s induced consideration set. Moreover, for each type $g \in \Gcal(B)$, we enforce that the rounded accumulated weight does not exceed $\hat{w}_g$. This constraint ensures that the accumulated total reward closely approximates the true expected revenue.

\subsection{Step 2a: Guessing and rewards} \label{app:step_2a}

In what follows, we show how to produce approximate total weights $\{\hat{w}_g\}_{g \in \Gcal(B)}$ such that \[
w_g^* \le \hat{w}_g \le (1+\eps)\cdot w_g^*
\]
for each $g \in \Gcal(B)$. recalling that $w^*_g = w_g(C_g(A^*))$ is used as a shorthand.  From here, we describe how these guessed quantities are used to define the reward of an assortment $A \subseteq \Ucal(B)$.

\paragraph{Guessing total weights.}  For each customer type $g \in \Gcal(B)$, we first guess $w_{\max}(g) = \max\{ w_{i,s,g} : (i,s) \in C_g(A^*) \}$, which corresponds to the largest size-adjusted weight among all products considered by a type-$g$ customer under $A^*$. For each type $g$, there are $O(n\ell_{\max})$ possible values for $w_{\max}(g)$, and hence the collection $\{w_{\max}(g)\}_{g \in \Gcal(B)}$ can be recovered by enumerating over
$
O\bigl((n\ell_{\max})^{|\Gcal(B)|}\bigr)
=
O\left((n\ell_{\max})^{O\left(\frac{\ell_{\max}^2}{\eps}\right)}\right)
$
guesses. With these guesses in hand, we know that
$
w_g^* \in [\, w_{\max}(g),\, n w_{\max}(g) \,].
$
Accordingly, within this interval we consider guesses $\hat{w}_g$ of the form
\[
\hat{w}_g = w_{\max}(g)\cdot (1+\eps)^q,
\qquad
q = 0, \ldots, \left\lceil \log_{1+\eps} n \right\rceil,
\]
resulting in a total of
$
O\!\left(\left(\frac{\log n}{\eps}\right)^{|\Gcal(B)|}\right)
=
O\!\left(\left(\frac{\log n}{\eps}\right)^{O\left(\frac{\ell_{\max}^2}{\eps}\right)}\right)
$
guesses across all types. Combining both guessing steps, the total number of guesses for $\{\hat{w}_g\}_{g \in \Gcal(B)}$ is
$
O\!\left(\left(\frac{n\ell_{\max}}{\eps}\right)^{O\left(\frac{\ell_{\max}^2}{\eps}\right)}\right),
$
and at least one such guess must satisfy
\[
w_g^* \le \hat{w}_g \le (1+\eps)\cdot w_g^*,
\]
for each $g \in \Gcal(B)$.

\paragraph{Reward.} For each type $g \in \Gcal(B)$, and any assortment $A \subseteq \Ucal(B)$, let
\[
\hat{\rev}_g(A) = \frac{1}{1+ (1+\eps)\cdot \hat{w}_g}\cdot \rho_g(C_g(A)),
\]
where $\rho_g(C_g(A)) = \sum_{(i,s) \in C_g(A)}r_{i,s}\cdot w_{i,s,g}$, denote the ``reward'' of a type-$g$ customer.  Within the approximate dynamic program proposed next, we seek a reward-maximizing assortment for which the accumulated weight of each relevant customer type’s consideration set is close to $\hat{w}_g$. As such, the reward $\hat{\rev}_g(A)$ serves as a valid proxy for the true expected revenue.

\subsection{Step 2b: The approximate dynamic program.} \label{app:step_2b}

In what follows, we present a dynamic program in which assortment decisions are made base-by-base, and as we proceed through the bases, for each type $g \in \Gcal(B)$, we store carefully rounded versions of the total accumulated weight.

\paragraph{State description.}  Each state $(i,\omega)$ in our dynamic program is characterized by the following two components. For notational convenience, let $G = |\Gcal(B)|$ denote the number of relevant customer types.

\begin{itemize}
    \item The current base $i \in \Ncal$ for which we must decide which sizes to offer. For the purpose of these base-by-base assortment decisions, define
    \[
    \Ucal_i(B)
    =
    \{(i,s) : s \in \Scal_i \cap B\}
    \]
    to be the set of products in base $i$ whose sizes lie within the block $B$.

    \item A vector $\omega = (\omega_1,\ldots,\omega_G)$ representing the accumulated weight for each customer type $g \in \Gcal(B)$ resulting from assortment decisions across bases $1,\ldots,i-1$. For each type $g \in \Gcal(B)$, define the discretization grid
    \[
    \Dcal_g
    =
    \left\{
        k \cdot \frac{\eps \hat{w}_g}{n}
        :
        k = 0,\ldots,\frac{2n}{\eps}
    \right\},
    \]
    together with the rounding operator $\Dcal_g(\cdot)$, which maps its argument down to the nearest point in $\Dcal_g$. Define
    \[
    \Dcal = \Dcal_1 \times \cdots \times \Dcal_G,
    \qquad
    \Dcal(\omega)
    =
    \bigl(\Dcal_1(\omega_1),\ldots,\Dcal_G(\omega_G)\bigr).
    \]
    Throughout, we restrict attention to accumulated weight vectors $\omega \in \Dcal$.
\end{itemize}

\paragraph{Value function.}  For each state $(i,\omega)$, the value function $\Vcal(i,\omega)$ represents the maximum reward that can be accrued from bases $i,\ldots,n$, given that the total accumulated (rounded) weight vector from decisions across bases $1,\ldots,i-1$ is $\omega$, and subject to the requirement that after processing all bases, the total accumulated weight does not exceed $\hat{w}_g$ in any component. Formally, the value function is defined recursively as
\begin{equation}\label{eqn:dp_oracle_PTAS}
    \Vcal(i,\omega)
    =
    \max_{A_i \subseteq \Ucal_i(B)}
    \left\{
        \sum_{g \in \Gcal(B)} \mu_g \cdot \hat{\rev}_g(A_i)
        +
        \Vcal\!\left(i+1,\; \Dcal\!\left(\omega + w|_{A_i}\right)\right)
    \right\},
\end{equation}
where
\[
w|_{A_i}
=
\left(
    w_1(C_1(A_i)),
    \ldots,
    w_G(C_G(A_i))
\right)
\]
denotes the vector of base-$i$ weight contributions to the consideration-set weight of each type $g \in \Gcal(B)$. Note that for each type $g \in \Gcal(B)$, the set $C_g(A_i)$ contains at most one product, since $A_i$ consists exclusively of base-$i$ products. The base cases of the recursion are given by
\[
\Vcal(n+1,\omega)
=
\begin{cases}
    0, & \text{if } \omega_g \le \hat{w}_g \quad \forall\, g \in \Gcal(B), \\
    -\infty, & \text{otherwise}.
\end{cases}
\]
These base cases ensure that the rounded accumulated weight of each customer type’s consideration set does not exceed the guessed quantity $\hat{w}_g$ used in computing the reward.

\subsection{Analysis: Concluding Theorem~\ref{thm:assort_Oracle}}

Let $\tilde{A}_1, \ldots, \tilde{A}_n$ denote the assortment decisions that result by following the dynamic program in~\eqref{eqn:dp_oracle_PTAS} from the initial state of $(1,\vec{0})$, and let $\tilde{A} = \biguplus_{i \in \Ncal}\tilde{A}_i$.  The following lemma shows that the expected revenue of $\tilde{A}$ is within a $(1-\eps)$-factor of optimal.
\begin{lemma}
    \label{lem:opt_B_PTAS}
    $\rev(\tilde{A}) \geq (1-\eps)\cdot \opt(B)$.
\end{lemma}
\proof{Proof.}
Our proof proceeds in two parts: first, we show that $\rev(\tilde{A}) \geq \Vcal(1,\vec{0})$, and then we show that $\Vcal(1,\vec{0}) \geq (1-\eps)\cdot \opt(B)$.

\paragraph{Part 1: Lower bounding $\rev(\tilde{A})$ by $\Vcal(1,\vec{0})$.}
Let $\tilde{\omega}^{(n+1)}$ denote the accumulated \emph{rounded} total-weight vector induced by the base-by-base decisions $\tilde{A}_1,\ldots,\tilde{A}_n$. Observe that
\begin{eqnarray*}
    \rev(\tilde{A})
    &=&
    \sum_{i \in \Ncal}\sum_{g \in \Gcal(B)}
    \mu_g \cdot \frac{1}{1+w_g(C_g(\tilde{A}))}\cdot \rho_g(C_g(\tilde{A}_i)) \\
    &\geq&
    \sum_{i \in \Ncal}\sum_{g \in \Gcal(B)}
    \mu_g \cdot \frac{1}{1+\tilde{\omega}^{(n+1)}_g + \eps\hat{w}_g}\cdot \rho_g(C_g(\tilde{A}_i)) \\
    &\geq&
    \sum_{i \in \Ncal}\sum_{g \in \Gcal(B)}
    \mu_g \cdot \frac{1}{1+(1+\eps)\cdot \hat{w}_g}\cdot \rho_g(C_g(\tilde{A}_i)) \\
    &=&
    \sum_{i \in \Ncal}\sum_{g \in \Gcal(B)} \mu_g \cdot \hat{\rev}_g(\tilde{A}_i)
    \;+\;
    \Vcal(n+1,\tilde{\omega}^{(n+1)}) \\
    &=&
    \Vcal(1,\vec{0}).
\end{eqnarray*}
The first inequality holds because $\tilde{\omega}^{(n+1)}_g$ is precisely the rounded version of $w_g(C_g(\tilde{A}))$: there are $n$ stages, and at each stage we round down by at most $\frac{\eps}{n}\hat{w}_g$, yielding a total rounding error of at most $\eps\hat{w}_g$. For the last inequality and second to last equality, we exploit the fact that we must have $\tilde{\omega}^{(n+1)}_g \leq \hat{w}_g$ since offering the empty set at each state is feasible within the dynamic program in~\eqref{eqn:dp_oracle_PTAS} and thus $\Vcal(1,\vec{0}) \geq 0$.

\paragraph{Part 2: Lower bounding $\Vcal(1,\vec{0})$ by $(1-\eps)\opt(B)$.}
To establish $\Vcal(1,\vec{0}) \ge (1-\eps)\opt(B)$, we use the optimal assortment $A^*$ to construct a feasible sequence of states and actions:
\[
(1,\omega^{*(1)})
\quad \xrightarrow{A^*_1} \quad
(2,\omega^{*(2)})
\quad \xrightarrow{A^*_2} \quad
\cdots
\quad \xrightarrow{A^*_n} \quad
(n+1,\omega^{*(n+1)}),
\]
where $A^*_i = A^* \cap \Ucal_i(B)$ and $\omega^{*(1)}=\vec{0}$, while for $i \in [2,n+1]$ we define
\[
\omega^{*(i)} = \Dcal\!\left(\omega^{*(i-1)} + w|_{A^*_i}\right).
\]
Let $\Vcal^*$ denote the total reward accumulated along this feasible sequence. Unrolling the recursion gives
\begin{eqnarray*}
    \Vcal^*
    &=&
    \sum_{i \in \Ncal}\sum_{g \in \Gcal(B)} \mu_g \cdot \hat{\rev}_g(A^*_i)
    \;+\;
    \Vcal(n+1,\omega^{*(n+1)}) \\
    &=&
    \sum_{i \in \Ncal}\sum_{g \in \Gcal(B)}
    \mu_g \cdot \frac{1}{1+(1+\eps)\cdot \hat{w}_g}\cdot \rho_g(C_g(A^*_i))
    \;+\;
    \Vcal(n+1,\omega^{*(n+1)}) \\
    &\geq&
    (1-\eps)^2 \cdot
    \sum_{i \in \Ncal}\sum_{g \in \Gcal(B)}
    \mu_g \cdot \frac{1}{1+w^*_g}\cdot \rho_g(C_g(A^*_i)) \\
    &=&
    (1-\eps)^2 \cdot \rev(A^*).
\end{eqnarray*}
Here, $\Vcal(n+1,\omega^{*(n+1)})=0$ because $\Dcal(\cdot)$ rounds down, so $\omega^{*(n+1)}_g \le w^*_g \le \hat{w}_g$ for all $g \in \Gcal(B)$. The inequality uses $\hat{w}_g \le (1+\eps)\cdot w^*_g$ (and hence $(1+\eps)\hat{w}_g \le (1+\eps)^2 w^*_g$), together with the bound $\frac{1}{1+(1+\eps)^2 w^*_g} \ge (1-\eps)^2\cdot \frac{1}{1+w^*_g}$ for $\eps \in (0,1)$.
Since this sequence is feasible, we have $\Vcal(1,\vec{0}) \ge \Vcal^*$, which yields the desired bound.

\endproof

\paragraph{Running time of the oracle.} In what follows, we show that the assortment $\hat{A}$ can be recovered in an overall running time of
\[
Oracle = O\!\left(\left(\frac{n\ell_{\max}}{\eps}\right)^{O\left(\frac{\ell_{\max}^2}{\eps}\right)}\right).
\]
To this end, we first argue that the total number of states $(i,\omega)$ in our dynamic program is
\[
O\!\left(\left(\frac{n}{\eps}\right)^{O\left(\frac{\ell_{\max}^2}{\eps}\right)}\right).
\]
This follows from the observations that:
\begin{itemize}
\item The base index $i$ can take on at most $n$ values.
\item The weight vector $\omega$ takes at most
\[
|\Dcal|
=
O\!\left(\left(\frac{n}{\eps}\right)^{G}\right)
=
O\!\left(\left(\frac{n}{\eps}\right)^{O\left(\frac{\ell_{\max}^2}{\eps}\right)}\right)
\]
possible values.
\end{itemize}
Finally, the optimal base-$i$ assortment to offer in each state can be found by enumerating over all feasible base-$i$ assortments, of which there are at most $O(2^{|B|})=2^{O\left(\frac{\ell_{\max}}{\eps}\right)}$. Putting these components together, and accounting for the guessing step in Step~2a, yields an overall running time
\[
Oracle
=
O\!\left(\left(\frac{n\ell_{\max}}{\eps}\right)^{O\left(\frac{\ell_{\max}^2}{\eps}\right)}\right),
\]
as claimed.

\section{Omitted Proofs from Section~\ref{sec:even_odd}}
\label{appendix-sec:proofs_inv}

\subsection{Proof of Lemma~\ref{lem:CCD}}\label{app:proof_CCD}

For the entirety of this proof, we consider an instance of the mixed-NP-MNL model with customer types $\Gcal$, arrival probabilities $\{\lambda_g\}_{g \in \Gcal}$, product weights $\{w_i\}_{i \in \Ncal}$, and no-purchase weights $\{w_{0,g}\}_{g \in \Gcal}$.  As such, for assortment $A$, the choice probability of product $i \in A$ is
\[
\pi(i,A) = \sum_{g \in\Gcal}\lambda_g \cdot \frac{w_i}{w_{0,g} + w(A)}.
\]
To prove the lemma we use the following necessary and sufficient condition of any CCD choice model given in~\cite{goyal2023pricing}.

\paragraph{CCD definition.}  For an arbitrary indexing of the $n$ products, let $A^{(i)} = \{1, \ldots, i\}$, and consider revenues $\theta \in \mathbb{R}^n_+$ induced by the following system of equations

	\begin{align}
		\label{eqn:CCD_linear_system}
		\begin{bmatrix}
			\pi(1 , A^{(1)}) & 0 & \cdots & 0\\
			\pi(1 , A^{(2)}) & \pi(2,A^{(2)}) & \cdots & 0 \\
			\vdots &  \vdots & \cdots & \vdots \\
			\pi(1 , A^{(n)}) & \pi(2, A^{(n)}) & \cdots & \pi(n, A^{(n)}) 
		\end{bmatrix}
		\begin{bmatrix}
			\theta_1\\
			\theta_2 \\
			\vdots \\
			\theta_n
		\end{bmatrix}
		= 
		\begin{bmatrix}
			1\\
			1\\
			\vdots \\
			1
		\end{bmatrix}.
	\end{align}
	As a consequence of Lemma 6 of \cite{goyal2023pricing}, the mixed-NP-MNL model is a CCD model if and only if 
    \begin{equation}
    \label{eqn:CCD}
        \max_{A \subseteq [n]}  \sum_{i \in S} \theta_i \cdot \pi(i, S)= 1.
    \end{equation}
   
To prove that~\eqref{eqn:CCD} holds, we will exploit the following fact regarding the assortment optimization problem under the mixed-NP-MNL model, which was established in Theorem~4.1 of~\cite{rusmevichientong2014assortment}. For completeness, we provide an alternative proof of the theorem in Appendix~\ref{app:proof-of-rev-order}.

\begin{theorem}\label{thm:rev_ordered}
Consider an arbitrary instance of the assortment optimization problem, where the products are indexed in non-decreasing order of revenue so that $r_1\geq \ldots, r_n $. There exists an optimal revenue-ordered assortment, i.e. there exists an optimal assortment of the form $A^{(i)}$ for some $i \in \Ncal$.       
\end{theorem}
Given Theorem~\ref{thm:rev_ordered}, to show~\eqref{eqn:CCD}, it suffices to prove that the revenues $\theta$ induced by~\eqref{eqn:CCD_linear_system} satisfy $\theta_1 \geq \theta_2 \geq \ldots \geq \theta_n$.  Indeed, if this were the case,~\eqref{eqn:CCD_linear_system} would then stipulate  that all revenue-ordered assortments garner an expected revenue of 1, and since at least one revenue-ordered assortment must be optimal via Theorem~\ref{thm:rev_ordered},~\eqref{eqn:CCD} would hold. Along this line, define
\[
d(A) = \sum_{g \in \Gcal} \frac{\lambda_g}{w_{0,g} + w(A)},
\]
so that $\pi(i,A) = w_i \cdot d(A)$. For ease of notation, define
$
d_j = d(A^{(j)}).
$
Then, for each $i \le j$, we have
$
\pi(i,A^{(j)}) = w_i d_j.
$
Let $\thetab$ satisfy the linear system in Definition~\eqref{eqn:CCD_linear_system}, and define
$
T_j = \sum_{i=1}^j \theta_i w_i.
$
From the $j$-th equation in the defining linear system, we obtain
\[
\sum_{i=1}^j \theta_i\cdot  \pi(i,A^{(j)}) = 1.
\]
Substituting $\pi(i,A^{(j)}) = w_i d_j$ yields
$
d_j \sum_{i=1}^j \theta_i w_i = 1,
$
and hence
$
T_j = \frac{1}{d_j}.
$
Next, define the function
\[
H(x) = \left( \sum_{g \in \Gcal} \frac{\lambda_g}{w_{0,g} + x} \right)^{-1},
\]
and note that
$
T_j = H(w(A^{(j)})).
$
We now show that $H(x)$ is concave on $\mathbb{R}_+$, which, because $w(A^{(j)})$ is increasing in $j$, implies that 
\[
\theta_j - \theta_{j+1}
=
\frac{T_j-T_{j-1}}{w_j}
-
\frac{T_{j+1}-T_j}{w_{j+1}}
=
\frac{H(w(A^{(j)}))-H(w(A^{(j-1)}))}{w_j}
-
\frac{H(w(A^{(j+1)}))-H(w(A^{(j)}))}{w_{j+1}}
\ge 0,
\]
where the last inequality uses the concavity of $H$. Hence, the sequence $\{\theta_j\}_{j=1}^n$ is non-increasing.

To show the concavity of $H(x)$, let
\[
f(x)=\sum_{g \in \Gcal} \frac{\lambda_g}{w_{0,g}+x},
\]
so that $H(x)=\frac{1}{f(x)}$. Then,
\[
f'(x)=-\sum_{g \in \Gcal} \frac{\lambda_g}{(w_{0,g}+x)^2},
\qquad
f''(x)=2\sum_{g \in \Gcal} \frac{\lambda_g}{(w_{0,g}+x)^3},
\]
and hence
\[
H''(x)
=
\frac{2(f'(x))^2-f(x)\cdot f''(x)}{f(x)^3}.
\]
Thus, it suffices to show that
\[
 f(x)\cdot \sum_{g \in \Gcal} \frac{\lambda_g}{(w_{0,g}+x)^3} - (f'(x))^2 \geq 0.
\]
Along this line, observe that
\begin{align*}
&f(x)\cdot \sum_{g \in \Gcal} \frac{\lambda_g}{(w_{0,g}+x)^3}
-
\left(\sum_{g \in \Gcal} \frac{\lambda_g}{(w_{0,g}+x)^2}\right)^2 \\
&\qquad=
\sum_{g,h \in \Gcal}
\lambda_g\lambda_h
\left(
\frac{1}{(w_{0,g}+x)(w_{0,h}+x)^3}
-
\frac{1}{(w_{0,g}+x)^2(w_{0,h}+x)^2}
\right).
\end{align*}
The diagonal terms with $g=h$ vanish, so grouping the remaining terms pairwise over $g<h$ gives
\begin{align*}
&f(x)\cdot \sum_{g \in \Gcal} \frac{\lambda_g}{(w_{0,g}+x)^3}
-
\left(\sum_{g \in \Gcal} \frac{\lambda_g}{(w_{0,g}+x)^2}\right)^2 \\
&\qquad=
\sum_{\substack{g,h \in \Gcal\\ g<h}}
\lambda_g\lambda_h \cdot 
\Bigg(
\frac{1}{(w_{0,g}+x)(w_{0,h}+x)^3}
+\frac{1}{(w_{0,h}+x)(w_{0,g}+x)^3}
-\frac{2}{(w_{0,g}+x)^2(w_{0,h}+x)^2}
\Bigg).
\end{align*}
Letting $a=w_{0,g}+x$ and $b=w_{0,h}+x$, the bracketed term becomes
\[
\frac{1}{ab^3}+\frac{1}{ba^3}-\frac{2}{a^2b^2}
=
\frac{a^2+b^2-2ab}{a^3b^3}
=
\frac{(a-b)^2}{a^3b^3}
\geq 0.
\]
Therefore,
\[
f(x) \cdot \sum_{g \in \Gcal} \frac{\lambda_g}{(w_{0,g}+x)^3}
-
\left(\sum_{g \in \Gcal} \frac{\lambda_g}{(w_{0,g}+x)^2}\right)^2
\geq 0,
\]
which implies $H''(x)\leq 0$, and thus $H$ is concave on $\mathbb{R}_+$.

\section{Omitted Proofs from Section~\ref{sec:fluid_setting}}

\subsection{Proof of Claim~\ref{claim:dp_even_works}}\label{app:proof_dp_even_works}

The claim follows immediately from the assumption that $\ell_{\max}=1$, which in turn implies that for any $c \in \Fcal_{\even}$, we have
\[
\rev(c) = \sum_{s \in \Scal_{\even}} \rev(c(s)),
\]
where $c(s)$ denotes the projection of $c$ onto the products of size $s$.

\subsection{Proof of Lemma~\ref{lem:proxy_near_op}}\label{app:proof_proxy_near_op}

In what follows, we begin by showing that for any $c \in \Fcal$ satisfying
$
\rev_w(c) \geq \frac{1}{2}\cdot \opt_w,
$
we have
\begin{equation}
    \label{eqn:rounding_lb}
    \rev_{\hat{w}}(c) \geq (1-5\eps)\cdot \rev_w(c).
\end{equation}
Next, for any $\eps \in (0,\frac{1}{2}]$, we show that if $\tilde{c} \in \Fcal$ satisfies
$
\rev_{\hat{w}}(\tilde{c}) \geq (1-\eps)\cdot \opt_{\hat{w}},
$
then
\begin{equation}
    \label{eqn:rounding_ub}
    \rev_w(\tilde{c}) \geq (1-5\eps)\cdot \rev_{\hat{w}}(\tilde{c}).
\end{equation}

Taken together, these two bounds imply the lemma. Indeed,
\begin{align*}
\rev_w(\tilde{c})
&\geq (1-5\eps)\cdot \rev_{\hat{w}}(\tilde{c}) \\
&\geq (1-5\eps)(1-5\eps)\cdot \opt_{\hat{w}} \\
&\geq (1-10\eps)\cdot \rev_{\hat{w}}(c^*) \\
&\geq (1-15\eps)\cdot \rev_w(c^*) \\
&= (1-15\eps)\cdot \opt_w,
\end{align*}
as desired. Before proving~\eqref{eqn:rounding_lb} and~\eqref{eqn:rounding_ub}, we introduce some additional notation and state several intermediate claims that streamline the arguments to follow.

\paragraph{Preliminaries.} We partition the time horizon into intervals 
\[
\Ical_{\mathrm{prefix}} = [0, (1-\eps^2)\cdot T)
\quad \text{and} \quad
\Ical_{\mathrm{suffix}} = [(1-\eps^2)\cdot T, T],
\]
and use $\rev_w(c,\Ical_{\mathrm{prefix}})$ and $\rev_w(c,\Ical_{\mathrm{suffix}})$ to denote the fluid revenue earned over these intervals for any $c \in \Fcal$ under the original weights.  Moreover, with a slight abuse of notation, we use $x_i(c,\Ical_{\mathrm{prefix}})$ and $x_i(c,\Ical_{\mathrm{suffix}})$ to denote the fluid sales of product $i \in \Ncal$ over these respective intervals. We define
\[
\Ncal_{\eps, \geq 1}(c) = \{i \in \Ncal : x_i(c) \geq 1,\; x_i(c,\Ical_{\mathrm{suffix}}) \leq \eps\}
\]
to be the set of products whose total fluid sales under $c$ is at least $1$ over the full horizon, but at most $\eps$ over $\Ical_{\mathrm{suffix}}$. Similarly, let
\[
\Ncal_{\eps, < 1}(c) = \{i \in \Ncal : x_i(c) < 1,\; x_i(c,\Ical_{\mathrm{suffix}}) \leq \eps\},
\]
and define
$
\Ncal_{\eps}(c) = \Ncal_{\eps, \geq 1}(c) \cup \Ncal_{\eps, < 1}(c)
$
to be the set of all products whose fluid sales over $\Ical_{\mathrm{suffix}}$ is at most $\eps$. With this notation, we can decompose the total fluid revenue as
\begin{equation}
\label{eqn:fluid_rev_decomp}
\rev_w(c)
=
\underbrace{\sum_{i \in \Ncal_{\eps, \geq 1}(c)} r_i \cdot x_i(c)}_{= \rev_{\eps, \geq 1}(c)}
+
\underbrace{\sum_{i \in \Ncal_{\eps, < 1}(c)} r_i \cdot x_i(c)}_{= \rev_{\eps, < 1}(c)}
+
\underbrace{\sum_{i \in \Ncal \setminus \Ncal_{\eps}(c)} r_i \cdot x_i(c)}_{= \rev_{\mathrm{other}}(c)}.
\end{equation}
 This decomposition separates products that accrue at most $\eps$ fluid sales over $\Ical_{\mathrm{suffix}}$ (the first two terms) from the remaining products.

\paragraph{Intermediate claims.}  We prove the following three intermediate claims. The first claim, whose proof is deferred to Appendix~\ref{app:proof_cut_off_rev}, shows that the fluid revenue earned under the rounded weights dominates the fluid revenue earned under the original weights over the prefix $\Ical_{\mathrm{prefix}}$ of the horizon. The second claim establishes that the rounded instance captures (up to a $(1-\eps)$ factor) two of the three components in the decomposition~\eqref{eqn:fluid_rev_decomp}. The final claim shows that the remaining component is negligible, provided that $c \in \Fcal$ is $\frac{1}{2}$-optimal. The latter two claims are proved at the end of this section.

\begin{claim}
    \label{claim:cut_off_rev}
    For any $c \in \Fcal$, we have
    \[
    \rev_{\hat{w}}(c) \geq \rev_w(c,\Ical_{\mathrm{prefix}}).
    \]
\end{claim}

\begin{claim}
    \label{claim:rev_two_terms}
    For any $c \in \Fcal$, we have 
    \[
    \rev_{\hat{w}}(c)
    \geq
    (1-\eps)\cdot \left(
    \rev_{\eps, \geq 1}(c)
    + \sum_{i \in \Ncal_{\eps,<1}(c)} r_i\, x_i(c,\Ical_{\mathrm{prefix}})+
    \sum_{i \in \Ncal \setminus \Ncal_{\eps}(c)} r_i \cdot x_i(c,\Ical_{\mathrm{prefix}})
    \right).
    \]
\end{claim}

\begin{claim}
    \label{claim:rev_last_terms}
    For any $c \in \Fcal$ such that $\rev_w(c) \geq \frac{1}{2}\cdot \opt_w$, we have
    \begin{enumerate}[label=(\roman*)]
        \item $2\eps \cdot \rev_w(c) \geq \displaystyle \sum_{i \in \Ncal_{\eps, \YC{<} 1}(c)} r_i \cdot x_i(c, \Ical_{\mathrm{suffix}} )$, and
        \item $2\eps \cdot \rev_w(c) \geq \displaystyle  \sum_{i \in \Ncal \setminus \Ncal_{\eps}(c)} r_i \cdot x_i(c,\Ical_{\mathrm{suffix}})$.
    \end{enumerate}
\end{claim}

\paragraph{Proof of~\eqref{eqn:rounding_lb}.}  Building off of~\eqref{eqn:fluid_rev_decomp}, we have that
\begin{align*}
\rev_w(c)
&= \rev_{\eps,\geq 1}(c)
   + \underbrace{
        \sum_{i \in \Ncal_{\eps,<1}(c)} r_i\, x_i(c,\Ical_{\mathrm{prefix}})
        + \sum_{i \in \Ncal_{\eps,<1}(c)} r_i\, x_i(c,\Ical_{\mathrm{suffix}})
     }_{=\,\rev_{\eps,<1}(c)} \\
&\hspace{1.5em}
   + \underbrace{
        \sum_{i \in \Ncal \setminus \Ncal_{\eps}(c)} r_i\, x_i(c,\Ical_{\mathrm{prefix}})
        + \sum_{i \in \Ncal \setminus \Ncal_{\eps}(c)} r_i\, x_i(c,\Ical_{\mathrm{suffix}})
     }_{=\,\rev_{\mathrm{other}}(c)} \\
&\le \rev_{\eps,\geq 1}(c)
   + \sum_{i \in \Ncal_{\eps,<1}(c)} r_i\, x_i(c,\Ical_{\mathrm{prefix}})
   + \sum_{i \in \Ncal \setminus \Ncal_{\eps}(c)} r_i\, x_i(c,\Ical_{\mathrm{prefix}})
   + 4\eps\, R_w(c) \\
&\le \frac{R_{\hat{w}}(c)}{1-\eps} + 4\eps\, R_w(c).
\end{align*}
as desired. The first inequality uses Claim~\ref{claim:rev_last_terms} and the second uses Claim~\ref{claim:rev_two_terms}.

\paragraph{Proof of~\eqref{eqn:rounding_ub}.} To establish this inequality, we take the rounded weights $\hat{w}$ and scale each of them upward by a factor of $(1+\eps^2)$ to obtain weights $\hat{w}^{\uparrow}$, where
\[
\hat{w}^{\uparrow}_i = (1+\eps^2)\cdot \hat{w}_i \geq w_i.
\]
We likewise scale the no-purchase weights by the same factor, so that the choice probabilities induced by $\hat{w}$ and $\hat{w}^{\uparrow}$ are identical. Consequently, for any $c \in \Fcal$, we have
$
\rev_{\hat{w}}(c) = \rev_{\hat{w}^{\uparrow}}(c).
$
From this point, we treat $\hat{w}^{\uparrow}$ as the ``original'' weight vector and view the true weights $w$ as the rounded-down counterpart. Applying~\eqref{eqn:rounding_lb} at the inventory vector $\tilde{c}$ therefore yields
\[
\rev_w(\tilde{c}) \geq (1-5\eps)\cdot \rev_{\hat{w}^{\uparrow}}(\tilde{c})
=
(1-5\eps)\cdot \rev_{\hat{w}}(\tilde{c}),
\]
as desired.

\paragraph{Proof of Claim~\ref{claim:rev_two_terms}.} We have
\begin{eqnarray*}
    \rev_{\hat{w}}(c)
    &\geq&
    \rev_w(c,\Ical_{\mathrm{prefix}})\\
    &=&
    \sum_{i \in \Ncal_{\eps, \geq 1}(c)} r_i \cdot x_i(c,\Ical_{\mathrm{prefix}})
    +
    \sum_{i \in \Ncal_{\eps,<1}(c)} r_i\, x_i(c,\Ical_{\mathrm{prefix}})+
    \sum_{i \in \Ncal \setminus \Ncal_{\eps}(c)} r_i \cdot x_i(c,\Ical_{\mathrm{prefix}}) \\
    &\geq&
    (1-\eps)\cdot \left(
    \rev_{\eps, \geq 1}(c)
    +
     \sum_{i \in \Ncal_{\eps,<1}(c)} r_i\, x_i(c,\Ical_{\mathrm{prefix}})+
    \sum_{i \in \Ncal \setminus \Ncal_{\eps}(c)} r_i \cdot x_i(c,\Ical_{\mathrm{prefix}})
    \right),
\end{eqnarray*}
where the first inequality follows from Claim~\ref{claim:cut_off_rev}, and the last inequality uses the fact that
\[
\sum_{i \in \Ncal_{\eps, \geq 1}(c)} r_i \cdot x_i(c,\Ical_{\mathrm{prefix}})
\geq
(1-\eps)\cdot \rev_{\eps, \geq 1}(c),
\]
since any product $i \in \Ncal_{\eps, \geq 1}(c)$ accrues at least a $(1-\eps)$ fraction of its total sales over $\Ical_{\mathrm{prefix}}$ by construction.

\paragraph{Proof of Claim~\ref{claim:rev_last_terms} - property (i).} Assume by way of contradiction that
\[
\sum_{i \in \Ncal_{\eps,\YC{<} 1}(c)} r_i \cdot x_i(c,\Ical_{\mathrm{suffix}})
> 2\eps \cdot \rev_w(c).
\]
We show that this implies the existence of an alternative inventory vector 
$u \in \Fcal$ such that $\rev_w(u) > \opt_w$, yielding a contradiction.

To begin, note that for each $i \in \Ncal_{\eps,<1}(c)$, we have $x_i(c) < 1$ by construction, and hence such products never stock out over the horizon. Consequently, over the interval $\Ical_{\mathrm{suffix}}$, the fluid sales rate of each $i \in \Ncal_{\eps,<1}(c)$ is upper bounded by $\pi(i,\Ncal_{\eps,<1}(c))$. Motivated by this observation, define $u \in \mathbb{Z}^n_+$ as
\[
u_i =
\begin{cases}
1, & \text{if } i \in \Ncal_{\eps,<1}(c),\\
0, & \text{otherwise}.
\end{cases}
\]
Clearly, $u \in \Fcal$. Under $u$, as long as product $i$ remains in stock, its choice probability is bounded below by $\pi(i,\Ncal_{\eps,<1}(c))$. Therefore, for each $i \in \Ncal_{\eps,<1}(c)$,
\[
x_i(u,\eps T)
\;\ge\;
\min\!\left\{1,\;  T\eps \cdot  \,\pi(i,\Ncal_{\eps,<1}(c))\right\}
\;\ge\;
\frac{1}{\eps}\, x_i(c,\Ical_{\mathrm{suffix}}),
\]
where the second inequality uses that $[0,\eps T]$ is $\frac{1}{\eps}$ times longer than $\Ical_{\mathrm{suffix}}$, together with the bound $x_i(c,\Ical_{\mathrm{suffix}}) \le \eps$. Summing over $i \in \Ncal_{\eps,<1}(c)$ yields
\[
\rev_w(u)
\;\ge\;
\frac{1}{\eps}
\sum_{i \in \Ncal_{\eps,<1}(c)} r_i \cdot x_i(c,\Ical_{\mathrm{suffix}})
\;>\;
2\,\rev_w(c)
\;\ge\;
\opt_w,
\]
where the final inequality uses that $c$ is $\tfrac{1}{2}$-optimal. This establishes the desired contradiction.

\paragraph{Proof of Claim~\ref{claim:rev_last_terms} (property (ii)).}
We again proceed by contradiction, following a similar high-level structure as in part (i). However, more nuanced technical arguments are required in the present case. To begin, note that the total fluid sales rate aggregated across all products is non-increasing over time, since products are only removed from the displayed assortment. Therefore, because at most $\Ccal$ total units can be consumed over the entire horizon, it must be the case that
\begin{equation}
    \label{eqn:inv_consumption_bound}
    \sum_{i \in \Ncal \setminus \Ncal_{\eps}(c)} x_i(c,\Ical_{\mathrm{suffix}})
    \leq
    \eps^2 \Ccal,
\end{equation}
since $\Ical_{\mathrm{suffix}}$ occupies only an $\eps^2$-fraction of the full horizon.

With this in mind, consider the alternative starting inventory vector $u \in \mathbb{Z}^n_+$ defined by
\[
u_i =
\begin{cases}
\left\lfloor \dfrac{1}{\eps^2}\cdot x_i(c,\Ical_{\mathrm{suffix}}) \right\rfloor,
& \text{if } i \in \Ncal \setminus \Ncal_{\eps}(c),\\[1em]
0,
& \text{otherwise}.
\end{cases}
\]
Using~\eqref{eqn:inv_consumption_bound}, it is straightforward to verify that $u \in \Fcal$. Moreover, in Appendix~\ref{app:proof_u_stock_out}, we show
\begin{equation}
    \label{eqn:stock_out_u}
    x_i(u) = u_i
    \qquad \forall i \in \Ncal \setminus \Ncal_{\eps}(c),
\end{equation}
meaning that every stocked product under $u$ eventually stocks out. Consequently,
\[
    \rev_w(u) = 
    \sum_{i \in \Ncal \setminus \Ncal_{\eps}(c)}
    r_i \cdot
    \left\lfloor \frac{1}{\eps^2}\cdot x_i(c,\Ical_{\mathrm{suffix}}) \right\rfloor \geq
    \frac{1}{\eps}\cdot
    \sum_{i \in \Ncal \setminus \Ncal_{\eps}(c)}
    r_i \cdot x_i(c,\Ical_{\mathrm{suffix}})  > 
    2\,\rev_w(c) \geq
    \opt_w,
\]
which yields the desired contradiction. Above, the first inequality follows from
\[
\left\lfloor \frac{1}{\eps^2}\,x_i(c,\Ical_{\mathrm{suffix}}) \right\rfloor
\geq
\frac{1}{\eps^2}\,x_i(c,\Ical_{\mathrm{suffix}}) - 1
\geq
\frac{1}{\eps}\,x_i(c,\Ical_{\mathrm{suffix}}),
\]
where the final inequality uses that
$
x_i(c,\Ical_{\mathrm{suffix}}) \geq \eps
$
for each $i \in \Ncal \setminus \Ncal_{\eps}(c)$ by construction.

\subsection{Proof of Claim~\ref{claim:cut_off_rev}} \label{app:proof_cut_off_rev}

For ease of exposition, we assume without loss of generality that $c$ stocks all $n$ products. Under the original $w$-process, we re-index the products in non-decreasing order of their stock-out times. Let $k^* \in [n]$ denote the largest index of a product that stocks out during $\Ical_{\mathrm{prefix}}$, and let $\tau_1,\ldots,\tau_{k^*}$ denote the corresponding stock-out moments.

With this ordering, the show-all policy under the $w$-process over $\Ical_{\mathrm{prefix}}$ displays the following sequence of assortments:
\begin{itemize}
    \item $A^{(1)} = \{1,\ldots,n\}$ over $[0,\tau_1]$,
    \item $A^{(2)} = \{2,\ldots,n\}$ over $[\tau_1,\tau_2]$,
    \item \hspace{1em}$\vdots$
    \item $A^{(k^*)} = \{k^*,\ldots,n\}$ over $[\tau_{k^*-1},\tau_{k^*}]$,
    \item $A^{(k^*+1)} = \{k^*+1,\ldots,n\}$ over $[\tau_{k^*},(1-\eps^2)T]$.
\end{itemize}

We prove via induction over $\ell \in [k^*+1]$ that for every $i \in [n]$,
\[
\hat{x}_i\bigl(c,\tau_{\ell}\cdot (1+\eps^2)\bigr) \geq x_i(c,\tau_\ell).
\]
Applying this inequality at $\ell = k^*+1$ yields, for all $i \in [n]$,
\[
x_i(c,\Ical_{\mathrm{prefix}})
=
x_i\bigl(c,(1-\eps^2)T\big)
\leq
\hat{x}_i\bigl(c,T(1-\eps^2)(1+\eps^2)\bigr)
\leq
\hat{x}_i(c),
\]
where the final inequality follows from the monotonicity of fluid sales and the fact that $T(1-\eps^2)(1+\eps^2) \leq T$.

\begin{itemize}

    \item \emph{Base case: $\ell=1$.} Over the interval $[0,\tau_1(1+\eps^2)]$, the sales rate of each product $i \in \Ncal$ is at least $\hat{\pi}(i,A^{(1)})$ when it is in stock. Therefore,
    \begin{align*}
        \hat{x}_i\bigl(c,\tau_1(1+\eps^2)\bigr)
        &\geq
        \min\Big\{
            c_i,\;
            \tau_1(1+\eps^2)\cdot \hat{\pi}(i,A^{(1)})
        \Big\} \\
        &=
        \min\Big\{
            c_i,\;
            \tau_1(1+\eps^2)
            \sum_{g \in \Gcal}
            \lambda_g \cdot
            \frac{\hat{w}_i}{w_{0,g}+\hat{w}(A^{(1)})}
        \Big\} \\
        &\geq
        \tau_1
        \sum_{g \in \Gcal}
        \lambda_g \cdot
        \frac{w_i}{w_{0,g}+w(A^{(1)})}
        \;=\;
        x_i(c,\tau_1),
    \end{align*}
    where the last inequality uses that $\hat{w}_i \leq w_i \leq (1+\eps^2)\hat{w}_i$. 

    \item \emph{Induction step.} Fix $\ell \in \{2,\ldots,k^*+1\}$ and assume that
\[
\hat{x}_i\bigl(c,\tau_{\ell-1}\cdot (1+\eps^2)\bigr) \geq x_i(c,\tau_{\ell-1})
\qquad \forall i \in [n].
\]
We show that
\[
\hat{x}_i\bigl(c,\tau_{\ell}\cdot (1+\eps^2)\bigr) \geq x_i(c,\tau_{\ell})
\qquad \forall i \in [n].
\]

We will assume that $\hat{x}_i\bigl(c,\tau_{\ell}\cdot (1+\eps^2)\bigr) < c_i$, since otherwise the statement holds trivially. By the induction hypothesis and the ordering of stock-out times in the $w$-process, products $1,\ldots,\ell-1$ have already stocked out under the $\hat{w}$-process by time $\tau_{\ell-1}\cdot (1+\eps^2)$. Hence, over
$
[\tau_{\ell-1}\cdot (1+\eps^2),\tau_{\ell}\cdot (1+\eps^2)],
$
the assortment under the $\hat{w}$-process is a subset of $A^{(\ell)} = \{\ell,\ldots,n\}$. Consequently, for each $i \in \{\ell,\ldots,n\}$, while product $i$ remains in stock under the $\hat{w}$-process, its sales rate is at least $\hat{\pi}(i,A^{(\ell)})$. Therefore,
\begin{align}
&\hat{x}_i\bigl(c,[\tau_{\ell-1}\cdot (1+\eps^2),\tau_{\ell}\cdot (1+\eps^2)]\bigr)
\nonumber\\
&\qquad=
(\tau_{\ell}-\tau_{\ell-1})(1+\eps^2)\cdot \hat{\pi}(i,A^{(\ell)})
\nonumber\\
&\qquad\geq
(\tau_{\ell}-\tau_{\ell-1})\cdot \pi(i,A^{(\ell)})
\label{eqn:app_1}
\end{align}
where the second inequality uses
$
(1+\eps^2)\cdot \hat{\pi}(i,A^{(\ell)}) \ge \pi(i,A^{(\ell)}).
$

Now fix $i \in \{\ell,\ldots,n\}$. Since product $i$ has not stocked out under the $w$-process by time $\tau_\ell$, we have
\[
x_i(c,\tau_{\ell})
=
x_i(c,\tau_{\ell-1})
+
(\tau_{\ell}-\tau_{\ell-1})\cdot \pi(i,A^{(\ell)}).
\]
Using the induction hypothesis and~\eqref{eqn:app_1}, it follows that
\begin{align*}
x_i(c,\tau_{\ell})
&=
x_i(c,\tau_{\ell-1})
+
(\tau_{\ell}-\tau_{\ell-1})\cdot \pi(i,A^{(\ell)}) \\
&\leq
\hat{x}_i\bigl(c,\tau_{\ell-1}\cdot (1+\eps^2)\bigr)
+
(\tau_{\ell}-\tau_{\ell-1})\cdot \pi(i,A^{(\ell)}) \\
&\leq
\hat{x}_i\bigl(c,\tau_{\ell-1}\cdot (1+\eps^2)\bigr)
+
\hat{x}_i\bigl(c,[\tau_{\ell-1}\cdot (1+\eps^2),\tau_{\ell}\cdot (1+\eps^2)]\bigr) \\
&=
\hat{x}_i\bigl(c,\tau_{\ell}\cdot (1+\eps^2)\bigr).
\end{align*}

For $i<\ell$, both processes have already stocked out product $i$, and hence
\[
\hat{x}_i\bigl(c,\tau_{\ell}\cdot (1+\eps^2)\bigr)
=
c_i
=
x_i(c,\tau_{\ell}).
\]
This completes the induction step.
\end{itemize}

\subsection{Proof of~\eqref{eqn:stock_out_u}}\label{app:proof_u_stock_out}

To prove~\eqref{eqn:stock_out_u}, we first establish the result without the floor operator; the full statement then follows immediately from fluid regularity (Claim~\ref{claim:fluid_regular}). Accordingly, for the remainder of the proof, we assume that
\[
u_i =
\begin{cases}
\dfrac{1}{\eps^2}\cdot x_i(c,\Ical_{\mathrm{suffix}}),
& \text{if } i \in \Ncal \setminus \Ncal_{\eps}(c),\\[1em]
0,
& \text{otherwise}.
\end{cases}
\]
Moreover, since we focus only on the products in $\Ncal \setminus \Ncal_{\eps}(c)$, we relabel these products so that
$
\Ncal \setminus \Ncal_{\eps}(c) = \Ncal_{\mathrm{other}}(c)= \{1,\ldots,n_{\mathrm{other}}\},
$
and index them in non-decreasing order of their stockout times under $c$. By construction, these products can stock out only during $\Ical_{\mathrm{suffix}}$, since each of them has positive fluid sales over this interval. For some $k^* \in [n_{\mathrm{other}}]$, let
$
\tau_1,\ldots,\tau_{k^*} \in \Ical_{\mathrm{suffix}}
$
denote the stockout moments of products $1,\ldots,k^*$ under $c$, so that products $k^*+1,\ldots,n_{\mathrm{other}}$ do not stock out under $c$. Letting $\tau_{k^*+1}=T$, define the intervals
$
\Ical_{\ell} = [(1-\eps^2)\cdot T,\tau_{\ell}]
~\text{for each } \ell \in [k^*+1].
$
Finally, for each $\ell \in [k^*+1]$, let
$
\hat{\tau}_{\ell} = \tau_{\ell} - (1-\eps^2)\cdot T
$
denote the amount of time elapsed within $\Ical_{\mathrm{suffix}}$ up to moment $\tau_{\ell}$.

We prove via induction over $\ell \in [k^*+1]$ that for each $i \in \Ncal_{\mathrm{other}}(c)$, we have
\[
x_i\!\left(u,\frac{1}{\eps^2}\cdot \hat{\tau}_{\ell}\right) \geq \frac{1}{\eps^2}\cdot x_i(c,\Ical_{\ell}),
\]
which, when applied at $\ell = k^*+1$, gives the desired result, since
\[
\frac{1}{\eps^2}\cdot \hat{\tau}_{k^*+1} = T
\qquad\text{and}\qquad
\Ical_{k^*+1} = \Ical_{\mathrm{suffix}}.
\]

\begin{itemize}

    \item \emph{Base case: $\ell=1$.} Over the interval $[0,\frac{1}{\eps^2}\cdot \hat{\tau}_1]$, the assortment displayed under $u$ is a subset of $\Ncal_{\mathrm{other}}(c)$. Hence, for each product $i$ in $\Ncal_{\mathrm{other}}(c)$, if it remains in stock during $[0,\frac{1}{\eps^2}\cdot \hat{\tau}_1]$ under $u$, its fluid sales rate under $u$ is at least its fluid sales rate over $\Ical_1$ under $c$, since all products in $\Ncal_{\mathrm{other}}(c)$ are in stock over $\Ical_1$ under $c$. We thus obtain
\[
    x_i\!\left(u,\frac{1}{\eps^2}\cdot \hat{\tau}_1\right)
    \geq
    \frac{1}{\eps^2}\cdot x_i(c,\Ical_1),
\]
by observing that the interval $[0,\frac{1}{\eps^2}\cdot \hat{\tau}_1]$ is exactly $\frac{1}{\eps^2}$ times longer than $\Ical_1$. Meanwhile, if $i$ is out of stock during $[0,\frac{1}{\eps^2}\cdot \hat{\tau}_1]$, then we simply have
\[
    x_i\!\left(u,\frac{1}{\eps^2}\cdot \hat{\tau}_1\right) = u_i = \frac{1}{\eps^2}\cdot x_i(c,  \Ical_{\mathrm{suffix}} ) \geq  \frac{1}{\eps^2}\cdot x_i(c,\Ical_1).
\]

    \item \emph{Induction step.} Fix $\ell \in \{2,\ldots,k^*+1\}$ and assume, for each $ i \in \Ncal_{\mathrm{other}}(c)$, that
    \[
    x_i\!\left(u,\frac{1}{\eps^2}\cdot \hat{\tau}_{\ell-1}\right)
    \geq
    \frac{1}{\eps^2}\cdot x_i(c,\Ical_{\ell-1}).
    \]
    By the induction hypothesis and the ordering of stockout times, products $1,\ldots,\ell-1$ have already stocked out under $u$ by time $\frac{1}{\eps^2}\cdot \hat{\tau}_{\ell-1}$. Thus, over the interval
    $
    \left[\frac{1}{\eps^2}\cdot \hat{\tau}_{\ell-1},\; \frac{1}{\eps^2}\cdot \hat{\tau}_{\ell}\right],
    $
    the assortment displayed under $u$ is a subset of the assortment displayed under $c$ over the interval
    $
    [\tau_{\ell-1},\tau_{\ell}].
    $
    Since the mixed-NP-MNL model is regular, it follows that for every product $i \in \{\ell,\ldots,n_{\mathrm{other}}\}$, the fluid sales rate under $u$ during the former interval is at least the fluid sales rate of the same product under $c$ during the latter interval. Therefore, for each $i \in \{\ell,\ldots,n_{\mathrm{other}}\}$, we have 
    \[
    x_i\!\left(u,\left[\frac{1}{\eps^2}\cdot \hat{\tau}_{\ell-1},\; \frac{1}{\eps^2}\cdot \hat{\tau}_{\ell}\right]\right)
    \geq
    \frac{1}{\eps^2}\cdot x_i(c,[\tau_{\ell-1},\tau_{\ell}]).
    \]
    Combining this with the induction hypothesis yields that for each $i \in \{\ell,\ldots,n_{\mathrm{other}}\}$
    \begin{align*}
    x_i\!\left(u,\frac{1}{\eps^2}\cdot \hat{\tau}_{\ell}\right)
    &=
    x_i\!\left(u,\frac{1}{\eps^2}\cdot \hat{\tau}_{\ell-1}\right)
    + 
    x_i\!\left(u,\left[\frac{1}{\eps^2}\cdot \hat{\tau}_{\ell-1},\; \frac{1}{\eps^2}\cdot \hat{\tau}_{\ell}\right]\right)
    \\
    &\geq
    \frac{1}{\eps^2}\cdot x_i(c,\Ical_{\ell-1})
    +
    \frac{1}{\eps^2}\cdot x_i(c,[\tau_{\ell-1},\tau_{\ell}]) \\
    &= \frac{1}{\eps^2}\cdot x_i(c,\Ical_{\ell}).
    \end{align*}
    For $i < \ell$, the inequality is immediate because product $i$ has already stocked out under $u$, and hence
    \[
    x_i\!\left(u,\frac{1}{\eps^2}\cdot \hat{\tau}_{\ell}\right)
    = u_i
    = \frac{1}{\eps^2}\cdot x_i(c,\Ical_{\mathrm{suffix}})
    \geq \frac{1}{\eps^2}\cdot x_i(c,\Ical_{\ell}).
    \]
    This completes the induction.
\end{itemize}

\subsection{Proof of Claim~\ref{claim:stock_non_dec_rev}}\label{app:proof_stock_non_dec_rev}

Assume by way of contradiction that $r_i > r_j$, but $\hat{c}_j > \hat{c}_i$. In this case, we construct an alternative inventory vector $u$ such that
\[
u_k =
\begin{cases}
    \hat{c}_k & \text{if } k \notin \{i,j\}, \\
    \hat{c}_j & \text{if } k = i, \\
    \hat{c}_i & \text{if } k = j.
\end{cases}
\]
Since products $i$ and $j$ belong to the same weight class, we have $\hat{w}_i = \hat{w}_j$. Consequently, under $u$, their total fluid sales simply swap relative to $\hat{c}$, i.e., $\hat{x}_i(u) = \hat{x}_j(\hat{c})$ and $\hat{x}_j(u) = \hat{x}_i(\hat{c})$, while the fluid sales of all other products remain unchanged. It therefore follows that
\[
\rev_{\hat{w}}(u) - \rev_{\hat{w}}(\hat{c})
=
(r_i - r_j)\cdot \bigl( \hat{x}_j(\hat{c}) - \hat{x}_i(\hat{c}) \bigr) > 0,
\]
which yields a contradiction. The inequality holds since $r_i > r_j$ and $\hat{x}_j(\hat{c}) > \hat{x}_i(\hat{c})$, where the latter follows from $\hat{x}_j(\hat{c}) > \hat{c}_j - 1 \geq \hat{c}_i \geq \hat{x}_i(\hat{c})$.

\subsection{Proof of Claim~\ref{claim:stock_out_group}} \label{app:proof_stock_out_group}

Assume by way of contradiction that there exists $i > i_q^{(1)}$ such that $\hat{x}_i(\hat{c}) < \hat{c}_i$. In this case, consider any $j \in \Ncal_q(k_q^{(1)}) = \{1,\ldots,i_q^{(1)}\}$. Since $\hat{c}_j > \hat{c}_i$, product $j$ also cannot stock out. Moreover, because $i$ and $j$ belong to the same weight class, they have the same rounded weight and therefore experience the same fluid sales while both remain available. It follows that
\[
\hat{x}_i(\hat{c}) = \hat{x}_j(\hat{c}) < \hat{c}_i \leq \hat{c}_j - 1,
\]
which contradicts the definition of $\hat{c}$, since the inventory of product $j$ could be reduced by one unit without affecting total fluid revenue.

\subsection{Proof of Lemma~\ref{lem:fluid_LP_lb}}\label{app:proof_fluid_LP_lb}

We prove the lemma by constructing a feasible solution to~\ref{eqn:fluid_LP} whose objective value is at least $(1-\eps)\cdot \opt_{\hat{w}}$.

\paragraph{Preliminaries.}
Consider the starting inventory vector $\hat{v} \in \mathbb{R}^n_+$ defined by $\hat{v}_i = \hat{x}_i(\hat{c})$ for each $i \in \Ncal$. By construction, $\hat{v} \leq \hat{c}$ componentwise. Therefore, by the fluid regularity property established in Claim~\ref{claim:fluid_regular}, it must be the case that all products stock-out starting from $v$, and thus $\hat{x}_i(\hat{v}) = \hat{x}_i(\hat{c})$ for each $i \in \Ncal$. Next, define the inventory vector $\hat{u} \in \mathbb{R}^n_+$ by
\[
\hat{u}_i =
\begin{cases}
\hat{c}_i
& \text{if } \displaystyle i \in \biguplus_{q \in [Q] \setminus Q_{L,<}} \biguplus_{k \in k(q)} \Ncal_q(k), \\[0.75em]
\ubar{x}_i
& \text{if } \displaystyle i \in \biguplus_{q \in Q_{L,<}} \Ncal_q(k_q^{(1)}), \\[0.75em]
 \hat{c}_i 
& \text{if }\displaystyle  i \in \biguplus_{q \in Q_{L,<}} \biguplus_{ 2 \leq \ell \leq m_q } \Ncal_q( k_q^{(\ell)} )  \\[0.75em]
\hat{x}_i(\hat{c})
& \text{if } \displaystyle i \in \biguplus_{q \in Q_H} \Ncal_{q,H}.
\end{cases}
\]
Thus, $\hat{u}$ differs from $\hat{v}$ only in the middle case above. Moreover, by the definition of $\ubar{x}_i$, we have
\[
\hat{u}_i \leq \hat{v}_i \leq (1+\eps)\cdot \hat{u}_i
\qquad
\forall i \in \biguplus_{q \in Q_{L,<}} \Ncal_q(k_q^{(1)}).
\]
Finally, for some $L \in [n]$, let $\hat{A}^{(1)}, \ldots, \hat{A}^{(L)}$ denote the sequence of assortments induced by the show-all policy starting from $\hat{u}$, and let $\hat{\tau}^{(1)}, \ldots, \hat{\tau}^{(L)}$ denote their respective durations.

\paragraph{Solution construction.} Note that under $\hat{u}$, we preserve the property stated in Claim~\ref{claim:stock_non_dec_rev}; namely, within each weight class, stocking levels are non-decreasing in revenue. Accordingly, under $\hat{u}$, stockouts within each weight class must occur in reverse revenue order (recall that when guessing the sales $\bar{x}_i$ for $i \in \Ncal_q \left( k_q^{(1)} \right)$, we have specified that it is larger than $k_q^{(2)}$). Specifically, while in-stock, all products in the same weight class experience the same fluid sales rate since they have the same rounded weight, and stocking levels under $\hat{u}$ are non-decreasing in revenue.  In turn,  this implies that $\hat{A}^{(\ell)} \in \Acal_{\mathrm{RO}}$ for every $\ell \in [L]$. We may therefore consider the following solution to~\ref{eqn:fluid_LP}:

\begin{itemize}
    \item For each $i \in \biguplus_{q \in Q_H} \Ncal_{q,H}$, set $u_i = \hat{u}_i$.
    \item Also, set
    \[
    h(A) =
    \begin{cases}
        {\hat{\tau}^{(\ell)}} & \text{if } A = \hat{A}^{(\ell)},\\
        T - \displaystyle\sum_{\ell=1}^L {\hat{\tau}^{(\ell)}} & \text{if } A = \emptyset,\\
        0 & \text{otherwise.}
    \end{cases}
    \]
\end{itemize}
We next verify the feasibility of this proposed solution and then evaluate the objective value it attains.

\paragraph{Feasibility.}  It is immediate to verify that constraints~(1),~(3),~(4), and~(5) are satisfied. For constraint~(2), note that $\hat{u} \leq \hat{v}$ componentwise, so Claim~\ref{claim:fluid_regular} implies that all products will stock out under $\hat{u}$, and thus $\hat{x}_i(\hat{u}) = \hat{u}_i$ for every $i \in \Ncal$. Therefore, for each $i \in \Ncal$,
\[
\sum_{\substack{A \in \Acal_{\mathrm{RO}}:\\ i \in A}} h(A)\cdot\pi(i,A)
= \sum_{\ell \in [L]} \hat{\tau}^{(\ell)} \cdot \pi(i,\hat{A}^{(\ell)})
= \hat{x}_i(\hat{u})
= \hat{u}_i.
\]
Thus, constraint~(2) is also satisfied.

\paragraph{Objective value.} This solution obtains an objective value of
\begin{eqnarray*}
    \rev_{\hat{w}}(\hat{u}) &=& \sum_{i \in \Ncal}r_i \hat{u}_i\\
    &\geq& (1-\eps)\cdot \sum_{i \in \Ncal}r_i \hat{v}_i\\
    &=& (1-\eps)\cdot \sum_{i \in \Ncal}r_i \hat{x}_i(\hat{c})\\
    &=& (1-\eps)\cdot \opt_{\hat{w}},
\end{eqnarray*}
where the lone inequality follows by $\hat{u}_i \geq (1-\eps)\cdot \hat{v}_i$ for each $i \in \Ncal$ by construction.

\subsection{Proof of Lemma~\ref{lem:final_guarantee}}\label{app:proof_final_guarantee}

To begin, throughout this proof we use $\Ncal_H = \biguplus_{q \in Q_H} \Ncal_{q,H}$ as shorthand for the set of heavy products, and $\Ncal_L = \biguplus_{q \in [Q]} \biguplus_{k \in k(q)} \Ncal_q(k)$ for the set of \emph{light} products, namely those that stock at most $\frac{1}{\eps}$ units under $\hat{c}$. To prove the lemma, we consider the starting inventory vector $u^{\mathrm{LP}} \in \mathbb{R}_+^n$, defined by
\[
u_i^{\mathrm{LP}}
=
\begin{cases}
\hat{c}_i 
& \text{if } \displaystyle i \in \Ncal_L, \\[0.75em]
u_i^*
& \text{if } \displaystyle i \in \Ncal_H,
\end{cases}
\]
which is simply $c^{\mathrm{LP}}$ without the floor operator applied to the starting inventories of the heavy products. We will prove that
\begin{equation}
    \label{eqn:RHS_bound}
   \hat{x}_i(u^{\mathrm{LP}}) \geq   (1-\eps)^2 \cdot \mathrm{RHS}_i \qquad \forall i \in \Ncal,
\end{equation}
where
\[
\mathrm{RHS}_i
=
\begin{cases}
\hat{c}_i 
& \text{if } \displaystyle i \in \biguplus_{q \in [Q] \setminus Q_{L,<}} \biguplus_{k \in k(q)} \Ncal_q(k), \\[0.5em]
\ubar{x}_i 
& \text{if } \displaystyle i \in \biguplus_{q \in Q_{L,<}} \Ncal_q\bigl(k_q^{(1)}\bigr), \\[0.5em]
\hat{c}_i  &  \text{if }\displaystyle  i \in \biguplus_{q \in Q_{L,<}} \biguplus_{ 2 \leq \ell \leq m_q } \Ncal_q( k_q^{(\ell)} )  \\[0.75em]
u_i^* 
& \text{if } \displaystyle i \in \biguplus_{q \in Q_H} \Ncal_{q,H},
\end{cases}
\]
denotes the right-hand side of constraint~(2) in the optimal solution to~\ref{eqn:fluid_LP}. Indeed, if~\eqref{eqn:RHS_bound} holds, then
\begin{eqnarray*}
   \rev_{\hat{w}}(c^{\mathrm{LP}}) &=& \sum_{i \in \Ncal} r_i \cdot \hat{x}_i(c^{\mathrm{LP}}) \\
   &\geq& (1-2\eps) \cdot \sum_{i \in \Ncal} r_i \cdot \hat{x}_i(u^{\mathrm{LP}}) \\
   &\geq& (1-2\eps)(1-\eps)^2 \cdot \sum_{i \in \Ncal} r_i \cdot \mathrm{RHS}_i \\
   &=& (1-2\eps)(1-\eps)^2 \cdot Z_{\mathrm{LP}}^*,
\end{eqnarray*}
where the first inequality follows from Claim~\ref{claim:fluid_regular}. Specifically, since $c^{\mathrm{LP}} \leq u^{\mathrm{LP}}$ componentwise, for each light product $i \in \Ncal_L$ we have $c_i^{\mathrm{LP}} = u_i^{\mathrm{LP}}$, and therefore fluid regularity yields $\hat{x}_i(c^{\mathrm{LP}}) \geq \hat{x}_i(u^{\mathrm{LP}})$, since in this case $\delta = 0$. For each heavy product $i \in \Ncal_H$, we have
\[
u_i^{\mathrm{LP}} - c_i^{\mathrm{LP}} = u_i^* - \lfloor u_i^* \rfloor \leq 1.
\]
Thus, applying Claim~\ref{claim:fluid_regular} with $\delta = 1$ yields
\[
\hat{x}_i(c^{\mathrm{LP}}) \geq \hat{x}_i(u^{\mathrm{LP}}) - 1 \geq (1-2\eps) \cdot \hat{x}_i(u^{\mathrm{LP}}),
\]
since by~\eqref{eqn:RHS_bound} we have
\[
\hat{x}_i(u^{\mathrm{LP}}) \geq (1-\eps)^2 \cdot u_i^* \geq (1-\eps)^2 \cdot \frac{1}{\eps},
\]
and hence $2\eps \cdot \hat{x}_i(u^{\mathrm{LP}}) \geq 1$ for $\eps \leq \frac{1}{2}$.

\paragraph{Proof of~\eqref{eqn:RHS_bound}.}  To prove this inequality, we use the following two intermediate claims, whose proofs can be found at the end of this section. The first claim  gives a general relation between the fluid sales of any two products. The second claim involves $\tau^*$, defined to be he earliest  moment where any product $i \in \biguplus_{q \in Q_{L,<}} \Ncal_q\bigl(k_q^{(1)}\bigr)$ hits its right-hand-side fluid sales of $\ubar{x}_i$ within the show-all policy starting at $u^{\mathrm{LP}}$. Specifically, we exploit the CCD property of the mixed-NP-MNL model to show that $\tau^* \leq T$, which simply confirms that this stock-out moment is well-defined.
\begin{claim}
    \label{claim:fluid_IIA}
    For any initial inventory vector $c \in \mathbb{R}_+^n$, let $\tau_i$ denote the stockout time of product $i \in \Ncal$, with the convention that $\tau_i = T$ if product $i$ never stocks out. Then, for any product $j \in \Ncal \setminus \{i\}$,
    \[
        \hat{x}_j(c,\tau_i) = \min\left\{c_j,\, \hat{x}_i(c,\tau_i) \cdot \frac{\hat{w}_j}{\hat{w}_i}\right\}.
    \]
\end{claim}

\begin{claim}
    \label{claim:tau_star}
We have that $\tau^* \in [0,T]$.
\end{claim}
Let $f \in \biguplus_{q \in Q_{L,<}} \Ncal_q\bigl(k_q^{(1)}\bigr)$ denote the product that stocks out at time $\tau^*$. Moreover, since $f$ never stocks out under $\hat{c}$, Claim~\ref{claim:fluid_IIA} implies that
\begin{equation}
\label{eqn:fluid_f_sales}
\hat{x}_i(\hat{c}) = \min\left\{\hat{c}_i,\, \hat{x}_f(\hat{c}) \cdot \frac{\hat{w}_i}{\hat{w}_f}\right\}
\qquad \forall i \in \Ncal \setminus \{f\}.
\end{equation}
Consequently, for any light product $i \in \Ncal_L$, we have
\begin{eqnarray*}
    \hat{x}_i(u^{\mathrm{LP}}) &\geq& \hat{x}_i(u^{\mathrm{LP}},\tau^*) \\
    &\geq& \min\left\{u_i^{\mathrm{LP}},\, \hat{x}_f(\hat{c},\tau^*) \cdot \frac{\hat{w}_i}{\hat{w}_f}\right\} \\
    &=& \min\left\{u_i^{\mathrm{LP}},\, \ubar{x}_f \cdot \frac{\hat{w}_i}{\hat{w}_f}\right\} \\
    &\geq& \min\left\{u_i^{\mathrm{LP}},\, (1-\eps) \cdot \hat{x}_f(\hat{c}) \cdot \frac{\hat{w}_i}{\hat{w}_f}\right\} \\
    &\geq& (1-\eps) \cdot \min\left\{u_i^{\mathrm{LP}},\, \hat{x}_i(\hat{c})\right\} \\
    &\geq& (1-\eps) \cdot \mathrm{RHS}_i,
\end{eqnarray*}
where the second inequality uses Claim~\ref{claim:fluid_IIA}, and the second-to-last inequality uses~\eqref{eqn:fluid_f_sales}. The final inequality follows because for each $i \in \Ncal_L$,
\[
\min\left\{u_i^{\mathrm{LP}},\, \hat{x}_i(\hat{c})\right\}
=
\min\left\{\hat{c}_i,\, \hat{x}_i(\hat{c})\right\}
=
\hat{x}_i(\hat{c})
\geq
\mathrm{RHS}_i.
\]

Now consider a heavy product $i \in \Ncal_H$, and suppose $i \in \Ncal_{\hat{q},H}$ for some $\hat{q} \in Q_H$. Then
\begin{eqnarray}
    \hat{x}_i(u^{\mathrm{LP}}) &\geq& \hat{x}_i(u^{\mathrm{LP}},\tau^*) \nonumber \\
    &\geq& \min\left\{u_i^{\mathrm{LP}},\, \hat{x}_f(\hat{c},\tau^*) \cdot \frac{\hat{w}_i}{\hat{w}_f}\right\} \nonumber \\
    &=& \min\left\{u_i^{\mathrm{LP}},\, \hat{x}_f(\hat{c},\tau^*) \cdot \frac{\hat{w}_1}{\hat{w}_f}\right\} \nonumber \\
    &=& \min\left\{u_i^{\mathrm{LP}},\, \ubar{x}_f \cdot \frac{\hat{w}_1}{\hat{w}_f}\right\} \nonumber \\
    &\geq& (1-\eps) \cdot \min\left\{u_i^{\mathrm{LP}},\, \hat{x}_f(\hat{c}) \cdot \frac{\hat{w}_1}{\hat{w}_f}\right\} \nonumber \\
    &\geq& (1-\eps) \cdot \min\left\{u_i^{\mathrm{LP}},\, \hat{c}_{\hat{q},\max} - 1\right\}. \label{eqn:heavy_last_steps}
\end{eqnarray}
Here, the second inequality again uses Claim~\ref{claim:fluid_IIA}, and the first equality uses the fact that within each weight class $q \in [Q]$, we have $\hat{w}_1 = \cdots = \hat{w}_{n_q}$. Recall that product~1 is the highest-revenue product in class $\hat{q}$ and is stocked at level $\hat{c}_{\hat{q},\max}$ under $\hat{c}$. To justify the final inequality, first note that if product~1 in class $\hat{q}$ does not stock out under $\hat{c}$, then
\[
\hat{c}_{\hat{q},\max} - 1 \leq \hat{x}_1(\hat{c}) = \hat{x}_f(\hat{c}) \cdot \frac{\hat{w}_1}{\hat{w}_f}.
\]
On the other hand, if product~1 stocks out at some time $\tau_1$, then
\[
\hat{c}_{\hat{q},\max} = \hat{x}_1(\hat{c},\tau_1) = \hat{x}_f(\hat{c},\tau_1) \cdot \frac{\hat{w}_1}{\hat{w}_f} \leq \hat{x}_f(\hat{c}) \cdot \frac{\hat{w}_1}{\hat{w}_f},
\]
which yields the same conclusion.

From here, we obtain
\begin{eqnarray*}
    \eqref{eqn:heavy_last_steps} \geq (1-\eps) \cdot (u_i^{\mathrm{LP}} - 1) 
    \geq  (1-\eps)^2 \cdot \mathrm{RHS}_i,
\end{eqnarray*}
where the first inequality follows from constraint~(3), which implies that $u_i^{\mathrm{LP}} \leq \hat{c}_{\hat{q},\max}$, and the second follows from constraint~(3) again, since $u_i^{\mathrm{LP}} \geq \frac{1}{\eps}$.

\paragraph{Proof of Claim~\ref{claim:fluid_IIA}.} First suppose that $j$ stocks out before $i$, so that $\tau_j \leq \tau_i$. Then
\[
c_j = \hat{x}_j(c,\tau_j) = \hat{x}_i(c,\tau_j) \cdot \frac{\hat{w}_j}{\hat{w}_i} \leq \hat{x}_i(c,\tau_i) \cdot \frac{\hat{w}_j}{\hat{w}_i},
\]
where the second equality follows because over the interval $[0,\tau_j]$, both $i$ and $j$ are in stock, and under the mixed-NP-MNL model the choice probabilities satisfy
\[
\hat{\pi}(j,A) = \hat{\pi}(i,A) \cdot \frac{\hat{w}_j}{\hat{w}_i}
\]
for every assortment $A \subseteq \Ncal$ containing both $i$ and $j$. Therefore,
\[
\hat{x}_j(c,\tau_i) = c_j = \min\left\{c_j,\, \hat{x}_i(c,\tau_i) \cdot \frac{\hat{w}_j}{\hat{w}_i}\right\},
\]
as claimed.

Now suppose instead that $j$ has not stocked out by time $\tau_i$. Then
\[
c_j \geq \hat{x}_j(c,\tau_i) = \hat{x}_i(c,\tau_i) \cdot \frac{\hat{w}_j}{\hat{w}_i},
\]
where the equality follows by the same proportionality argument, since both $i$ and $j$ remain in stock throughout $[0,\tau_i]$. Hence,
\[
\hat{x}_j(c,\tau_i) = \hat{x}_i(c,\tau_i) \cdot \frac{\hat{w}_j}{\hat{w}_i} = \min\left\{c_j,\, \hat{x}_i(c,\tau_i) \cdot \frac{\hat{w}_j}{\hat{w}_i}\right\},
\]
which completes the proof.

\paragraph{Proof of Claim~\ref{claim:tau_star}.} Consider the scaled convex combination of assortments induced by $\{h^*(A)\}_{A \in \Acal_{\mathrm{RO}}}$, corresponding to the optimal $h$-variables in~\ref{eqn:fluid_LP}. By constraint~(1), we have
\[
\sum_{A \in \Acal_{\mathrm{RO}}} h^*(A) = T,
\]
and by constraint~(2),
\[
\phi_i^* = \sum_{\substack{A \in \Acal_{\mathrm{RO}} :\\ i \in A}} h^*(A)\, \hat{\pi}(i,A) = \mathrm{RHS}_i.
\]
Consequently, by the CCD property of the mixed-NP-MNL model, starting from the initial inventory vector $\phi^* = (\phi_1^*, \ldots, \phi_n^*)$, all products stock out by time $T$. Furthermore, observe that $u^{\mathrm{LP}}$ differs from $\phi^*$ only for products in $\biguplus_{q \in Q_{L,<}} \Ncal_q\bigl(k_q^{(1)}\bigr)$, for which
\[
\phi_i^* = \ubar{x}_i < \hat{c}_i = u_i^{\mathrm{LP}}.
\]
Therefore, the fluid sales of every product under $u^{\mathrm{LP}}$ coincide with those under $\phi^*$ until the first time that some product in $\biguplus_{q \in Q_{L,<}} \Ncal_q\bigl(k_q^{(1)}\bigr)$ stocks out. This stockout time is precisely $\tau^*$. Since all products stock out by time $T$ under $\phi^*$, it follows that $\tau^* < T$.

\section{Omitted Discussion and Proofs from Section~\ref{sec:asym_one_half}}

\subsection{Constraint Generation for~\ref{eqn:asymptotic-joint-LP}}
\label{subsec:app-asymptotic-constraint-generation}

In this subsection, we show that \ref{eqn:asymptotic-joint-LP} can be solved by a constraint generation procedure equipped with a polynomial-time separation oracle. We note that the odd-size counterpart, JLP-ODD, can be solved analogously. We first observe that \ref{eqn:asymptotic-joint-LP} is equivalent to
\begin{equation*}
	\begin{aligned}
		\underset{ h }{\text{max}} \quad
		&  \sum_{ {A \subseteq  {\Ucal_{\even}} } }   h(A) \cdot \left( \sum_{(i,s) \in A }  r_i \cdot \pi((i,s), A) \right) \\[0.3em]
		\text{s.t.} \quad
		& \sum_{A \subseteq {\Ucal_{\even}} } h(A) = \theta \cdot T &, \\[0.3em]
		& \sum_{A \subseteq {\Ucal_{\even}} } h(A) \cdot \left(  \sum_{(i,s) \in A } \pi((i,s), A) \right)   \leq \theta \cdot \Ccal, & \\[0.3em]
		& h(A) \geq 0 & \forall A \subseteq {\Ucal_{\even}}.
	\end{aligned}
\end{equation*}
Its dual takes the form
\begin{equation*}
	\begin{aligned}
		\underset{ \eta , \mu  }{\text{min}} \quad
		&  \theta \cdot (T \eta + \Ccal \mu ) \\[0.3em]
		\text{s.t.} \quad
		& \eta \geq \max_{A \subseteq \Ucal_{\even}} \left\lbrace \sum_{(i,s) \in A}  (r_i - \mu) \cdot \pi \left(  (i,s) , A  \right) \right\rbrace &, \\[0.3em]
		& \mu \geq 0. &
	\end{aligned}
\end{equation*}
Therefore, to obtain a polynomial-time separation oracle for the dual, it suffices to solve the assortment optimization problem appearing in the first constraint in polynomial time.

To this end, fix $(\eta,\mu)$ and define $r'_i \equiv r_i - \mu$. Then, for any assortment $A \subseteq \Ucal_{\even}$, the objective in the first dual constraint can be decomposed as
\[
\sum_{(i,s) \in A}  r'_i \cdot \pi \left(  (i,s) , A  \right)
= \sum_{s \in \Scal_\even} \sum_{ i : \, (i,s) \in A } r'_i \cdot \pi( (i,s) , A )
= \sum_{s \in \Scal_\even} \sum_{ i : \, (i,s) \in A } r'_i \cdot \pi( (i,s) , A \cap \Ucal(s) ),
\]
where $\Ucal(s)$ denotes the set of products of size $s$, and the second equality follows from the independence of product demand across different sizes in $\Scal_{\even}$.

Since $\Ucal(s) \cap \Ucal(s') = \emptyset$ whenever $s \neq s'$, the optimization decouples across sizes, and hence
\[
\max_{A \subseteq \Ucal_{\even}} \left\lbrace  \sum_{(i,s) \in A}  r'_i \cdot \pi \left(  (i,s) , A  \right)  \right\rbrace
= \sum_{s \in \Scal_\even}  \max_{A_s \subseteq \Ucal(s)  }   \left\lbrace  \sum_{ i : \, (i,s) \in A_s  }  r'_i \cdot \pi \left(  (i,s) , A_s  \right)  \right\rbrace.
\]
For each $s \in \Scal_{\even}$, the maximization on the right-hand side is precisely an assortment optimization problem under the mixed-NP-MNL model. By Theorem~\ref{thm:rev_ordered}, this problem admits a revenue-ordered optimal assortment and can therefore be solved in $O(n)$ time. Summing over all $s \in \Scal_{\even}$, we conclude that the assortment optimization problem in the first dual constraint can be solved in $O(nm)$ time. This yields a polynomial-time separation oracle, and thus \ref{eqn:asymptotic-joint-LP} can be solved efficiently by constraint generation.

\subsection{Proof of Lemma~\ref{lemma:asymptotic-performance-comparison}}
\label{appendix-subsubsec:asymptotic-proof-coupling}

\paragraph{Proof of the first part of Lemma~\ref{lemma:asymptotic-performance-comparison}.} We construct a coupling between $\Pcal$ and $\Pcal_{\downarrow}$.

\paragraph{Coupling via realized utilities.} Recall that randomness in both processes arises from customer choice under the mixed-NP-MNL model :
\[
\pi(i,A) = \sum_{g \in \Gcal} \lambda_g \cdot \frac{ w_i }{ w_{0,g} + \sum_{i \in A} w_i }.
\]
Here, we simply write $\Gcal(s)$ as $\Gcal$, since the size $s$ is fixed throughout the proof. This choice model admits a random utility representation. Let $G_t$ denote the customer type at time $t$, with $\mathbb{P}(G_t = g) = \lambda_g$ for $g \in \Gcal$. For each type $g$, product $i \in [n]$ (and the no-purchase option $i=0$), and period $t$, define the realized utility
\[
U_{g,i,t} = \log(w_i) + \epsilon_{g,i,t}, \qquad U_{g,0,t} = \log(w_{0,g}) + \epsilon_{g,0,t},
\]
where $\epsilon_{g,i,t}$ are i.i.d.\ standard Gumbel random variables. Conditional on $G_t = g$, the customer chooses
\[
\arg\max_{i \in A \cup \{0\}} U_{g,i,t}.
\]
Thus, the process $\Pcal$ can be represented by the sequence $\left( G_t, (\epsilon_{g,i,t})_{g \in \Gcal,i \in [n]_0} \right)_{t \in [\theta T]}$. We use the same realization of these random variables to define $\Pcal_{\downarrow}$, thereby coupling the two processes. Under this coupling, the realized utilities at each time period are identical in both processes, although the resulting choices may differ due to different offered assortments.

\paragraph{Revenue dominance of $\Pcal$ over $\Pcal_{\downarrow}$.} We show that under this coupling, $\REV_\Pcal \geq \REV_{\Pcal_{\downarrow}}$ holds almost surely. We first use $\Omega$ and $\Omega_{\downarrow}$ to denote the random sets of ``real'' purchased units under the $\Pcal$ and $\Pcal_{\downarrow}$ processes, respectively. Recall that, under the true process $\Pcal$, every purchase is real. Therefore, $\Omega$ is precisely the set of units purchased under $\Pcal$. We have the following claim.

\begin{claim}
	\label{claim:coupling_purchased_unit_cardinality}
	Under the coupling, we have $|\Omega| \geq | \Omega_{\downarrow} |$ almost surely.
\end{claim}

Establishing this claim requires some care because the two processes $\Pcal$ and $\Pcal_{\downarrow}$ differ in both their offered assortments and inventory consumption mechanisms (in particular, due to the possibility of virtual purchases under the latter process). We defer the proof to Appendix~\ref{appendix-subsubsec:cardinality-comparison}.

We use $R(\Omega')$ to denote the revenue incurred by the set of real purchased units $\Omega'$. Therefore, we have $R(\Omega) = R( \Omega \backslash \Omega_{\downarrow} ) +R(  \Omega \cap \Omega_{\downarrow} ) $  and $R( \Omega_\downarrow) = R( \Omega_{\downarrow} \backslash \Omega ) + R(  \Omega \cap \Omega_{\downarrow} )$. We also notice that $\REV_\Pcal = R(\Omega)$ and
\[
\REV_{\Pcal_{\downarrow}} = R( \Omega_{\downarrow}) - r_{\max} \cdot \left(  || \hat{c}_\theta ||_1 - |  \Omega_{\downarrow}|   \right),
\]
where the term in the parentheses is the number of the unsold units under the $\Pcal_\downarrow$ process. Consequently, we obtain that
\begin{align*}
	\REV_\Pcal - \REV_{\Pcal_\downarrow}  = \, & R( \Omega \backslash \Omega_{\downarrow} ) - R( \Omega_{\downarrow} \backslash \Omega )  +  r_{\max} \cdot \left(   || \hat{c}_\theta ||_1 - |  \Omega_{\downarrow}  |   \right) \\ 
	\geq \, & R( \Omega \backslash \Omega_{\downarrow} ) - R( \Omega_{\downarrow} \backslash \Omega )  +  r_{\max} \cdot |  \Omega \backslash \Omega_{\downarrow}  | \\
	\geq \, & r_{\min} \cdot | \Omega \backslash \Omega_{\downarrow}| - r_{\max} \cdot | \Omega_{\downarrow} \backslash \Omega  | +  r_{\max} \cdot |  \Omega \backslash \Omega_{\downarrow}  | \\
	= \, & r_{\min} \cdot | \Omega \backslash \Omega_{\downarrow}| +  r_{\max} \cdot \left(   |  \Omega \backslash \Omega_{\downarrow}  |  -  | \Omega_{\downarrow} \backslash \Omega  |  \right) \, \geq  \, 0.
\end{align*}
In the first inequality, we use the fact that each purchased unit in set $\Omega \backslash \Omega_{\downarrow}$ corresponds to an unsold unit in the $\Pcal_{\downarrow}$ process. The second inequality follows since each product's unit revenue is between $r_{\min}$ and $r_{\max}$. The third inequality follows Claim~\ref{claim:coupling_purchased_unit_cardinality}: since $|  \Omega  | \geq | \Omega_{\downarrow}|$, we have $ |\Omega \backslash \Omega_{\downarrow}|  \geq | \Omega \backslash \Omega_{\downarrow}|  $.

\paragraph{Proof of the second part of Lemma~\ref{lemma:asymptotic-performance-comparison}.} For notational simplicity, let $c_i = \hat{c}_{\theta,i}$. Recall that
\begin{align*}
	\REV_{\Pcal_{\downarrow},i}
	&= r_i \cdot \min \{ c_i , Y_i \} - r_{\max} \cdot (c_i - Y_i)^+ \\
	&= r_i Y_i - r_i  \cdot (Y_i - c_i)^+ - r_{\max} \cdot  (c_i - Y_i)^+  \geq r_i Y_i - r_{\max} \cdot |c_i - Y_i|.
\end{align*}
Taking expectations, we first note that
\begin{align*}
	\Ebb[Y_i]
	= \sum_{\ell=1}^n \Ebb[Y_{i,\ell}]
	= \sum_{\ell=1}^n \floor{\tau_\ell - \tau_{\ell-1}} \cdot \pi(i,\Acal_\ell)
	\geq \left( \sum_{\ell=1}^n (\tau_\ell - \tau_{\ell-1}) \cdot \pi(i,\Acal_\ell) \right) - n.
\end{align*}
The term in parentheses equals the total fluid sales of product $i$, denoted as $\REV_{\Pcal_F,i}$, by the definition of $\tau_\ell$'s. Therefore, $\Ebb [ Y_i ] \geq c_i - n$. Similarly, we have
\[
\Ebb[Y_i] = \sum_{\ell=1}^n \floor{\tau_\ell - \tau_{\ell-1}} \cdot \pi(i,\Acal_\ell) \leq \sum_{\ell=1}^n \left({\tau_\ell - \tau_{\ell-1}} \right) \cdot \pi(i,\Acal_\ell) \leq c_i.
\]
Hence $c_i - n \leq \Ebb[Y_i] \leq c_i$. It follows that
\[
\Ebb[\REV_{\Pcal_{\downarrow},i}] \geq \REV_{\Pcal_F,i} - r_{\max} \cdot \left( n + \Ebb[|c_i - Y_i|] \right).
\]
Dividing by $\REV_{\Pcal_F,i} = r_i c_i \geq r_{\min} c_i$, we obtain
\[
\frac{\Ebb[\REV_{\Pcal_{\downarrow},i}]}{\REV_{\Pcal_F,i}} \geq 1 - \frac{r_{\max}}{r_{\min}} \cdot \frac{ n + \Ebb[|c_i - Y_i|]}{c_i}.
\]

\paragraph{Concentration.} It remains to show that $\left(n +\Ebb[|c_i - Y_i|]\right)/c_i \to 0$ as $\theta \to \infty$. We first note that the pre-rounded inventory vector $\bar{c}$ scales linearly with $\theta$ because it is the optimal solution to \ref{eqn:asymptotic-joint-LP}. Particularly, if $\bar{c}_{\theta = 1}$ is the optimal solution of \ref{eqn:asymptotic-joint-LP} when $\theta = 1$, then we have $\bar{c} = \theta \cdot \bar{c}_{\theta = 1}$. This immediately implies that $c_i = \floor{ \bar{c}_i } = \Theta \left(  \theta \right)$. Moreover, we also have
\[
\Ebb[|c_i - Y_i|] \leq \Ebb[|Y_i - \Ebb[Y_i]|] + |\Ebb[Y_i] - c_i| \leq \Ebb[|Y_i - \Ebb[Y_i]|] + n.
\]
Therefore, to establish that $\left(n +\Ebb[|c_i - Y_i|]\right)/c_i \to 0$ as $\theta \rightarrow \infty$, it suffices to show that $\Ebb\big[|Y_i - \Ebb[Y_i]|\big] = o(\theta)$. This can be shown by Cauchy--Schwarz inequality
\[
\Ebb[|Y_i - \Ebb[Y_i]|]   \leq \sqrt{ \Ebb \left[  (Y_i - \Ebb[Y_i])^2     \right]   } = \sqrt{\text{Var}(Y_i)} = \sqrt{ \sum_{\ell=1}^n \text{Var}(Y_{i,\ell}) },
\]
where in the last inequality we use the fact that $Y_i = \sum_{\ell = 1}^n Y_{i,\ell}$ and each $Y_{i,\ell}$ for $\ell \in [n]$ is independent from each other. Since $\text{Var}(Y_{i,\ell}) \leq \tau_\ell - \tau_{\ell-1}$, we further obtain
\[
\Ebb[|Y_i - \Ebb[Y_i]|] \leq \sqrt{ \sum_{\ell=1}^n (\tau_\ell - \tau_{\ell-1}) } \leq \sqrt{\theta T}.
\]
Thus, $\Ebb[|Y_i - \Ebb[Y_i]|] = O(\sqrt{\theta})$, while $c_i = \Theta(\theta)$, implying that $ (n + \Ebb[|c_i - Y_i|] )/c_i \to 0$. %

\subsection{Proof of Claim~\ref{claim:coupling_purchased_unit_cardinality}}
\label{appendix-subsubsec:cardinality-comparison}

We first introduce additional notation for the penalized process $\Pcal_{\downarrow}$. Let $\Acal_{(t)}$ denote the predetermined assortment associated with time period $t$ under $\Pcal_{\downarrow}$. Specifically, if $t \in \{ \sum_{k=1}^{\ell-1} \floor{\tau_k-\tau_{k-1}}+1,\ldots,\sum_{k=1}^{\ell}\floor{\tau_k-\tau_{k-1}}\}$, then $\Acal_{(t)}=\Acal_\ell$.

We next define an auxiliary stochastic process $\bar{\Pcal}$ that bridges the true process $\Pcal$ and the penalized process $\Pcal_{\downarrow}$. Conceptually, at time $t$, process $\bar{\Pcal}$ offers only the products in $\Acal_{(t)}$ that have positive inventory. Thus, $\bar{\Pcal}$ resembles $\Pcal$ in that customers can purchase only products that are in stock, while it resembles $\Pcal_{\downarrow}$ in that its offered assortment is restricted by $\Acal_{(t)}$. Formally, let $\bar{c}_i(t)$ denote the inventory level of product $i \in [n]$ at the beginning of time $t$ under $\bar{\Pcal}$, and similarly let $c_i(t)$ and $c_i^{\downarrow}(t)$ denote the corresponding inventory levels under $\Pcal$ and $\Pcal_{\downarrow}$, respectively. The assortment offered by $\bar{\Pcal}$ at time $t$ is therefore $\bar{A}_t \equiv \{ i \in \Acal_{(t)} \, : \, \bar{c}_i(t) \geq 1 \}$. All purchases under $\bar{\Pcal}$ are thus real.

The stochasticity of all three processes is determined by the arriving customer types and the random utilities, represented by $\left( G_t, (\epsilon_{g,i,t})_{g \in \Gcal, i \in [n]_0} \right)_{t \in [\theta T]}$. We couple $\bar{\Pcal}$ with $\Pcal$ and $\Pcal_{\downarrow}$ using the same coupling introduced in Appendix~\ref{appendix-subsubsec:asymptotic-proof-coupling}: all three processes share the same realization of $\left( G_t, (\epsilon_{g,i,t})_{g \in \Gcal, i \in [n]_0} \right)_{t \in [\theta T]}$. Moreover, all three processes start with the same initial inventory vector~$\hat{c}_\theta$. Recall that $\Omega$ and $\Omega_{\downarrow}$ denote the sets of units purchased by the end of the horizon under $\Pcal$ and $\Pcal_{\downarrow}$, respectively, and let $\bar{\Omega}$ denote the corresponding set under $\bar{\Pcal}$.

The following claim directly establishes Claim~\ref{claim:coupling_purchased_unit_cardinality}

\begin{claim}
	\label{claim:coupling-cardinality-intermediate-step}
	Under the coupling, (i) $|\Omega| \geq |\bar{\Omega}|$ almost surely; and (ii) $|\bar{\Omega}| \geq |  \Omega_{\downarrow}  |$ almost surely.
\end{claim}

We prove the claim as follows. For part (i), we first show that $c_i(t) \geq \bar{c}_i(t)$ for all $i \in \Acal_{(t)}$ at the beginning of every time period $t$; equivalently, $c(t) \geq \bar{c}(t)$ on $\Acal_{(t)}$. We proceed by induction. The statement holds at $t=1$ because $c_i(1)=\bar{c}_i(1)$ for all $i \in [n]$. Suppose that $c_i(t') \geq \bar{c}_i(t')$ for all $i \in \Acal_{(t')}$. Let $\bar{j} \in \bar{A}_{t'} \cup {0}$ denote the customer choice under $\bar{\Pcal}$ at time $t'$, and let $j$ denote the corresponding choice under $\Pcal$. We consider three cases.
\begin{enumerate}
	\item \underline{$j=\bar{j}$}. The inventory of the chosen product decreases by one under both processes (unless $j =\bar{j}=0$, which is a trivial case), while all other inventory levels remain unchanged. Hence, $c_i(t'+1) \geq \bar{c}_i(t'+1)$ for all $i \in \Acal_{(t')}$.
	
	\item \underline{$j\neq\bar{j}$ and $j\notin\Acal_{(t')}$}. No inventory in $\Acal_{(t')}$ is consumed under $\Pcal$, whereas $\bar{c}(t'+1) \leq \bar{c}(t')$ element-wise. Therefore, on $\Acal_{(t')}$, $c(t'+1) = c(t') \geq \bar{c}(t')\geq\bar{c}(t'+1)$.
	
	\item \underline{$j\neq\bar{j}$ and $j\in\Acal_{(t')}$}. By the induction hypothesis, $c_i(t')\geq\bar{c}_i(t')$ for all $i\in\Acal_{(t')}$. Recall that $\bar{A}_{t'}$ and ${A}_{t'}$ are the offered assortments under $\bar{\Pcal}$ and $\Pcal$, respectively. Since $ \bar{A}_{t'} \subseteq \Acal_{(t')}$, every product in $\bar{A}_{t'}$ is also in stock under $\Pcal$, and hence $\bar{A}_{t'}\subseteq A_{t'}$. Under the coupling, if $\Pcal$ chooses $j\in A_{t'}\cap\Acal_{(t')}$ while $\bar{\Pcal}$ does not, product $j$ must be out of stock under $\bar{\Pcal}$ at the beginning of time $t'$; otherwise, since the two processes share the same realized utilities and $\bar{A}_{t'}\subseteq A_{t'}$, product $j$ would also be chosen under $\bar{\Pcal}$.
	
	\noindent Hence, $\bar{c}_j(t')=0$ while $c_j(t')\geq1$. It follows that $c_j(t'+1) = c_j(t')-1 \geq 0 = \bar{c}_j(t') = \bar{c}_j(t'+1)$, while for every $k\in\Acal_{(t')}\setminus\{j\}$, $c_k(t'+1)=c_k(t')\geq\bar{c}_k(t')\geq\bar{c}_k(t'+1)$. Thus, $c(t'+1)\geq\bar{c}(t'+1)$ on $\Acal_{(t')}$.
	
\end{enumerate}
Since $\Acal_{(t'+1)}\subseteq\Acal_{(t')}$, the three cases imply that $c(t'+1)\geq\bar{c}(t'+1)$ on $\Acal_{(t'+1)}$, completing the induction.

Consequently, we have established that $c_i(t) \geq \bar{c}_i(t)$ for all $i \in \Acal_{(t)}$ at the beginning of every time period $t$. Since $\bar{A}_t \subseteq \Acal_{(t)}$, every product in $\bar{A}_t$, which is in stock under $\bar{\Pcal}$, is also in stock under $\Pcal$, and hence $\bar{A}_t \subseteq A_t$ for every $t$. Under the coupling, whenever a purchase occurs under $\bar{\Pcal}$, a purchase must also occur under $\Pcal$. Therefore, $|\Omega| \geq |\bar{\Omega}|$ almost surely, proving part (i).

For part (ii) in Claim~\ref{claim:coupling-cardinality-intermediate-step}, we show that $\bar{c}_i(t)\leq c_i^{\downarrow}(t)$ for every $i \in [n]$ and every time period $t$. This immediately yields $|\bar{\Omega}| \geq |\Omega_{\downarrow}|$, because the two processes have the same initial inventory and hence
\[|\bar{\Omega}|=\sum_{i \in [n]}(\bar{c}_i(1)-\bar{c}_i(\theta T+1)) \geq \sum_{i\in[n]}(c_i^{\downarrow}(1) - c_i^{\downarrow} (\theta T+1)) = | \Omega_{\downarrow}|.
\]

We establish the inventory inequality by induction. It holds at $t=1$ because the two processes have identical initial inventories. Suppose that $\bar{c}_i(t') \leq c_i^{\downarrow}(t')$ for all $i \in [n]$. Let $\bar{j}$ and $j^*$ denote the customer choices at time $t'$ under $\bar{\Pcal}$ and $\Pcal_{\downarrow}$, respectively. By construction, $\bar{j},j^* \in \Acal_{(t')} \cup \{0 \}$, and $j^*$ maximizes realized utility over $\Acal_{(t')}\cup{0}$. We consider two cases.
\begin{enumerate}
	\item \underline{$j^* = 0$, or the purchase of $j^* \in \Acal_{(t')}$ under $\Pcal_{\downarrow}$ is virtual}. No inventory is consumed under $\Pcal_{\downarrow}$, so $c^{\downarrow}(t'+1)=c^{\downarrow}(t')$. Since inventory under $\bar{\Pcal}$ can only decrease, $\bar{c}(t'+1) \leq \bar{c} (t')\leq c^{\downarrow}(t')=c^{\downarrow}(t'+1)$ element-wise.
	
	\item \underline{The purchase of $j^*\in\Acal_{(t')}$ under $\Pcal_{\downarrow}$ is real}. In this case, $c_{j^*}^{\downarrow}(t')\geq1$. If $\bar{j}=j^*$, the inventory of $j^*$ decreases by one under both processes, while all other inventory levels remain unchanged, so $\bar{c}(t'+1)\leq c^{\downarrow}(t'+1)$. If instead $\bar{j}\neq j^*$, then, because $j^*$ maximizes realized utility over $\Acal_{(t')}\cup\{0\}$, the only reason it is not chosen under $\bar{\Pcal}$ is that it is out of stock, so $\bar{c}_{j^*}(t')=0$. Therefore, $\bar{c}_{j^*}(t'+1) = \bar{c}_{j^*}(t') = 0 \leq c_{j^*}^{\downarrow}(t')-1=c_{j^*}^{\downarrow}(t'+1)$. For every $k\neq j^*$, we similarly have $\bar{c}_k(t'+1)\leq\bar{c}_k(t')\leq c_k^{\downarrow}(t')=c_k^{\downarrow}(t'+1)$.
	
\end{enumerate}
Thus, $\bar{c}(t'+1)\leq c^{\downarrow}(t'+1)$ in all cases, completing the induction. 

Hence, $\bar{c}_i(t)\leq c_i^{\downarrow}(t)$ for all $i\in[n]$ and $t=1,\ldots,\theta T+1$, which, as argued above, implies $|\bar{\Omega}|\geq|\Omega_{\downarrow}|$ almost surely and proves part (ii).

\section{Additional Proofs, Explanations and Examples} \label{app:extras}

\subsection{Regularity of the CFTC model}
\label{appendix-sec:regularity_property}

\begin{definition}
	\label{def:substitutability}
	A choice model $\Pbb$ over choices in $\Ucal \cup \{ \nopurchase  \}$ satisfies the \emph{regularity} property if $\Pbb ( a \mid A \cup \{ a' \}) \leq \Pbb (a \mid A)$ for all assortments $A$ and choices $a$ and $a'$ such that $a' \in \Ncal \backslash A$.
\end{definition}
Regularity states that the probability of choosing any particular product cannot increase when the offered assortment expands. It is commonly viewed as the weakest form of rational choice and is sometimes referred to as \emph{weak rationality} \citep{rieskamp2006extending,jagabathula2019limit,chen2019decision}. This property is satisfied by most standard choice models, including the MNL, the LC-MNL \citep{train2009discrete}, and the ranking-based model \citep{farias2013nonparametric}.

Our choice model also satisfies the regularity property under Assumption~\ref{assumption:disutility_is_monotonic}. \citet{akchen2023size} establish regularity for their style-size choice model, which is a special case of our model, and their proof strategy extends directly to our setting. For brevity, we omit the proof.

\subsection{Fluid regularity}\label{app:proof_fluid_regular}

Here, we establish a natural structural property of the fluid demand process: reducing the inventory of one product can only increase the sales of all other products. 

\paragraph{Fluid demand regularity.} In what follows, we establish a basic regularity property of the fluid demand process, which holds for any real-valued starting inventory vector. 
\begin{claim}
    \label{claim:fluid_regular}
    For any real-valued starting inventory vector $c \in \mathbb{R}^{n\times m}_+$, product $(i,s) \in \Ucal$, and perturbation $\delta \in (0,c_{i,s}]$, let $c_{\delta} = c - \delta\cdot e_{i,s}$ denote the inventory vector obtained by lowering the inventory of product $(i,s)$ by $\delta$. Then:
\begin{enumerate}[label=(\roman*)]
\item $x_{(i,s)}(c_{\delta}) \geq x_{(i,s)}(c) -(\delta - (c_{i,s} - x_{(i,s)}(c)))^+$,
\item $x_{(j,\sigma)}(c) \geq x_{(j,\sigma)}(c_{\delta}) \quad \forall~(j,\sigma) \in \Ucal\setminus\{(i,s)\}$.
\end{enumerate}
\end{claim}

\proof{Proof.}
We will prove the two properties in sequence, but first note the following straightforward fact: the fluid demand trajectories under $c$ and $c_{\delta}$ coincide for every product until product $(i,s)$ stocks out under $c_{\delta}$. Thus, if product $(i,s)$ never stocks out under $c_{\delta}$, then the claim is immediate. Accordingly, we proceed under the assumption that product $(i,s)$ does stock out under $c_{\delta}$ at moment $\tau^{(0)}$, at which time it has not yet stocked out under $c$.

\paragraph{Proof of property (i).}
Since product $(i,s)$ stocks out under $c_{\delta}$, we have
\[
x_{(i,s)}(c_{\delta}) = c_{i,s}-\delta \geq x_{(i,s)}(c) -(\delta - (c_{i,s} - x_{(i,s)}(c)))^+
\]
since $x_{(i,s)}(c) \leq c_{i,s}$. 

\paragraph{Proof of property (ii).}
Up to moment $\tau^{(0)}$, the fluid sales of all products are identical under $c$ and $c_{\delta}$. Over the interval $[\tau^{(0)},T]$, however, the fluid sales under the two starting inventory vectors may differ. To analyze this difference, let $\tau^{(1)},\ldots,\tau^{(L)}$ denote the stockout moments under $c$. For each $\ell \in [L]_0$, let $A^{(\ell)}$ denote the displayed assortment over the interval $[\tau^{(\ell)},\tau^{(\ell+1)})$, where $\tau^{(L+1)}=T$, and let $A^{(\ell)}_{\delta}$ denote the largest-cardinality assortment displayed under $c_{\delta}$ over the same interval.

We will prove by induction on $\ell$ that
$
A^{(\ell)}_{\delta} \subseteq A^{(\ell)}
~\text{for each } \ell \in [L]_0.
$
Since the CFTC model is regular, this inclusion implies that every product in $\Ucal\setminus\{(i,s)\}$ receives weakly greater fluid demand under $c_{\delta}$ than under $c$ over $[\tau^{(0)},T]$. Combined with the fact that fluid sales coincide up to $\tau^{(0)}$, this establishes property (ii).

\begin{itemize}
    \item \emph{Base case: $\ell=0$.} Since product $(i,s)$ stocks out at moment $\tau^{(0)}$ under $c_{\delta}$, while no product stocks out over $[\tau^{(0)},\tau^{(1)})$ under $c$, we have
    \[
    A^{(0)}_{\delta} = A^{(0)} \setminus \{(i,s)\} \subset A^{(0)}.
    \]

    \item \emph{Inductive step.} For each $k \in [L]_0$, let $(i^{(k)},s^{(k)})$ denote the product that stocks out under $c$ at moment $\tau^{(k)}$. Suppose that for all $k<\ell$ we have
    $
    A^{(k)}_{\delta} \subseteq A^{(k)}.
    $
    It suffices to show that product $(i^{(\ell)},s^{(\ell)})$ has already stocked out under $c_{\delta}$ by time $\tau^{(\ell)}$. Indeed, this would imply
    \[
    A^{(\ell)}
    =
    A^{(\ell-1)} \setminus \{(i^{(\ell)},s^{(\ell)})\}
    \supseteq
    A^{(\ell-1)}_{\delta} \YC{ \setminus \{(i^{(\ell)},s^{(\ell)})\} }
    \supseteq
    A^{(\ell)}_{\delta},
    \]
    where the first inclusion follows from the induction hypothesis and the second from the fact that products can only be removed from the displayed assortment over time.

    To prove that $(i^{(\ell)},s^{(\ell)})$ must have stocked out under $c_{\delta}$ by time $\tau^{(\ell)}$, suppose for contradiction that it has not. Then
    $
    x_{(i^{(\ell)},s^{(\ell)})}(c_{\delta},\tau^{(\ell)}) < c_{i^{(\ell)},s^{(\ell)}}.
    $
    On the other hand,
    \begin{eqnarray*}
        x_{(i^{(\ell)},s^{(\ell)})}(c_{\delta},\tau^{(\ell)})
        &\geq&
        x_{(i^{(\ell)},s^{(\ell)})}(c_{\delta},\tau^{(0)})
        +
        \sum_{k=1}^{\ell}
        (\tau^{(k)}-\tau^{(k-1)})
        \cdot
        \pi((i^{(\ell)},s^{(\ell)}),A^{(k-1)}_{\delta}) \\
        &\geq&
        x_{(i^{(\ell)},s^{(\ell)})}(c_{\delta},\tau^{(0)})
        +
        \sum_{k=1}^{\ell}
        (\tau^{(k)}-\tau^{(k-1)})
        \cdot
        \pi((i^{(\ell)},s^{(\ell)}),A^{(k-1)}) \\
        &=&
        x_{(i^{(\ell)},s^{(\ell)})}(c,\tau^{(\ell)}) \\
        &=&
        c_{i^{(\ell)},s^{(\ell)}},
    \end{eqnarray*}
    which is a contradiction. Here, the first inequality follows because $A^{(k-1)}_{\delta}$ is the largest-cardinality assortment displayed under $c_{\delta}$ over the interval $[\tau^{(k-1)},\tau^{(k)})$. The second inequality follows from regularity together with the induction hypothesis, which implies that
    $
    A^{(k-1)}_{\delta} \subseteq A^{(k-1)}
    ~\text{for all } k \in [\ell].
    $
\end{itemize}
\endproof

\subsection{No reduction to a Markov chain model}\label{app:MC_not_NP_MNL}

Consider three products $\Ncal = \{1,2,3\}$. Let the mixed-NP-MNL model have two customer types, each with the same product weights
$
w_1 = w_2 = w_3 = 1,
$
but different no-purchase weights $w_{0,1} = \frac12$ and $w_{0,2} = 5$, and let the two types arrive with probabilities $\frac12$ and $\frac12$. For any offered assortment $A$, the purchase probability of product $i \in A$ is therefore
\[
\pi(i,A)
=
\frac12 \cdot \frac{1}{w_{0,1} + |A|}
+
\frac12 \cdot \frac{1}{w_{0,2} + |A|}.
\]
Hence, by symmetry, if $|A|=k$, each offered product is chosen with probability
$
\pi_k
=
\frac12 \cdot \frac{1}{\frac12 + k}
+
\frac12 \cdot \frac{1}{5+k}.
$
In particular,
\[
\pi_1 = \frac12\left(\frac{1}{3/2} + \frac{1}{6}\right) = \frac{5}{12},
\qquad
\pi_2 = \frac12\left(\frac{1}{5/2} + \frac{1}{7}\right) = \frac{19}{70},
\qquad
\pi_3 = \frac12\left(\frac{1}{7/2} + \frac{1}{8}\right) = \frac{23}{112}.
\]

We now show that these probabilities cannot be generated by any Markov chain choice model. Suppose, toward a contradiction, that there exists a Markov chain choice model representing these probabilities. Let $\lambda_i$ denote the initial probability of starting at product $i$, and let $\rho_{ij}$ denote the transition probability from state $i$ to state $j$, where state $0$ is the no-purchase state. Since under the full assortment $\{1,2,3\}$ every product state is absorbing, we must have
$
\lambda_1 = \lambda_2 = \lambda_3 = \pi_3.
$
Now consider the assortment $\{1,2\}$. Starting from state $3$, the chain is absorbed in one step into either $1$, $2$, or $0$. Therefore,
$
\pi(1, \{1,2\}) = \lambda_1 + \lambda_3 \rho_{31}.
$
Since the left-hand side equals $\pi_2$, we obtain
$\
\pi_2 = \pi_3 + \pi_3 \rho_{31}$, so $\rho_{31} = \frac{\pi_2-\pi_3}{\pi_3}$. By symmetry, the same argument applied to the assortments $\{1,3\}$ and $\{2,3\}$ yields
\[
\rho_{ij} = \frac{\pi_2-\pi_3}{\pi_3}
\qquad \forall i \neq j, \; i,j \in \{1,2,3\}.
\]
Thus every product state has the same transition probability to each of the other two product states. Define
$
q := \frac{\pi_2-\pi_3}{\pi_3}.
$
Then each product state transitions to each other product state with probability $q$, and hence transitions to the no-purchase state with probability $1-2q$.

Next consider the singleton assortment $\{1\}$. Let
\[
a := \Pr(\text{hit product 1 before state 0 }\mid \text{start from state 2, assortment }\{1\}),
\]
and similarly let
\[
b := \Pr(\text{hit product 1 before state 0 }\mid \text{start from state 3, assortment }\{1\}).
\]
By symmetry, $a=b$. Moreover, under the singleton assortment $\{1\}$, starting from state $2$ one moves to state $1$ with probability $q$, to state $3$ with probability $q$, and to state $0$ with probability $1-2q$. Therefore
$
a = q + q a.
$
Hence
$
a = \frac{q}{1-q}.
$
Since $a=b$, the probability of choosing product $1$ from the singleton assortment is
$
p_1 = \lambda_1 + \lambda_2 a + \lambda_3 a = \pi_3 + 2\pi_3 a.
$
Substituting $a = q/(1-q)$ gives
$
\pi_1 = \pi_3\left(1 + \frac{2q}{1-q}\right) = \pi_3 \cdot \frac{1+q}{1-q}.
$
Using $q = (p_2-\pi_3)/\pi_3$, this simplifies to
$
\pi_1 = \frac{\pi_2 \pi_3}{2\pi_3-\pi_2}.
$
Thus any Markov chain choice model consistent with the symmetric pairwise and full-assortment probabilities must satisfy
$
\pi_1 = \frac{\pi_2 \pi_3}{2\pi_3-\pi_2}.
$
However, in our mixed-NP-MNL example,
\[
\frac{\pi_2 \pi_3}{2\pi_3-\pi_2}
=
\frac{\frac{19}{70}\cdot \frac{23}{112}}{2\cdot\frac{23}{112}-\frac{19}{70}}
=
\frac{437}{1092},
\]
whereas
\[
\pi_1 = \frac{5}{12} = \frac{455}{1092},
\]
and thus we obtain a contradiction.

\subsection{Proof of Theorem~\ref{thm:rev_ordered}}\label{app:proof-of-rev-order}

We first claim the following result regarding the curvature at stationary points of a rational sum.
	
	\begin{claim}
		\label{claim:RO_proof_convexity}
		Define a function $H(x): (0,\infty) \rightarrow \mathbb{R}$ such that $H(x) = \sum_{\ell = 1}^L \mu_\ell \cdot \frac{C+x}{D_\ell + x}$, where $C$, $D_\ell$, and $\mu_\ell$ are positive for all $\ell \in [L]$. We further assume that $D_\ell$ are not all equal.
		For any $x^* > 0$ such that $H'(x^*) = 0$, we have $H''(x^*) > 0$.
	\end{claim}

	To prove Theorem~\ref{thm:rev_ordered}, we first use a binary vector $x \in \{ 0,1 \}^n$ to denote the assortment decision $A$. Then, the assortment optimization problem can expressed as an optimization over ${x} \in \{ 0,1 \}^n$ and further upper bounded by relaxing ${x}$:
	\begin{align*}
	 \max_{{x} \in \{ 0,1 \}^n  } \left\lbrace \sum_{g \in \Gcal}  \frac{  \lambda_g \cdot \left(\sum_{i \in [n]} r_i w_i x_i \right) }{   w_{0,g} + \left( \sum_{i \in [n]} w_i x_i \right) }  \right\rbrace  \leq \max_{{x} \in [0,1]^n  } \left\lbrace \sum_{g \in \Gcal} \frac{  \lambda_g \cdot \left(\sum_{i \in [n]} r_i w_i x_i \right) }{   w_{0,g} + \left( \sum_{i \in [n]} w_i x_i \right) } \right\rbrace = \max_{\epsilon \in [0, \sum_{i \in [n]} w_i ]}  \left\lbrace \sum_{ g \in \Gcal} \frac{ \lambda_g \cdot \xi(\epsilon) }{   w_{0,g} + \epsilon } \right\rbrace
	\end{align*}
	where we define $\xi(\epsilon) = \max \left\lbrace  \sum_{i \in [n]} r_i w_i x_i \, : \, x \in [0,1]^n, \, \sum_{i=1}^n w_i x_i = \epsilon \right\rbrace$. Here, $\xi(\epsilon)$ is the optimal objective value of a continuous knapsack problem and the optimal value can be obtained by a greedy procedure according to the revenue $r_i$. We first increase $x_1$ from zero until either the capacity constraint is tight or $x_1 = 1$. If the former case occurs, the greedy procedure terminates; otherwise, we proceed with variable $x_2$ and increase its value from zero. We repeat until we exhaust the capacity or until all components of ${x}$ are one. Under this greedy procedure, for each $k \in [n]$, $\sum_{i= 1}^k w_i$ corresponds to the total preference weight of the revenue-ordered assortment $A^{(k)} \equiv \{ 1,\ldots, k \}$ and $\xi \left( \sum_{i = 1}^k w_i \right) = \sum_{i \in A^{(k)} } r_i w_i$. Therefore, we complete the proof of this proposition if the following claim holds:
	\[
	\max_{\epsilon \in [0, \sum_{i \in [n]} w_i ]}  \left\lbrace \sum_{ g \in \Gcal} \frac{  \lambda_g \cdot \xi(\epsilon) }{   w_{0,g} + \epsilon } \right\rbrace = \max_{k \in [n]} \left\lbrace \sum_{g \in \Gcal} \frac{  \lambda_g \cdot \xi \left(  \sum_{i=1}^k w_i \right)  }{ w_{0,g} + \sum_{i=1}^k w_i } \right\rbrace.
	\]
	
	Define $\mathcal{H}(\epsilon) \equiv \sum_{g \in \Gcal} \frac{  \lambda_g \cdot \xi(\epsilon) }{   w_{0,g} + \epsilon } $. We let $\epsilon^* \in [0,\sum_{i \in [n]} w_i]$ be an optimizer of $\mathcal{H}(\epsilon)$. Assume that $\sum_{i \in [k-1]} w_i < \epsilon^* <  \sum_{i \in [k]} w_i$ for an integer $k \in [n]$. We use $\delta = \epsilon^* - \sum_{i \in [k-1]} w_i > 0 $ to denote how much $\epsilon^*$ exceeds a breakpoint. Then, we have $\xi(\epsilon^*) = \sum_{i = 1}^{k-1} r_i w_i + \delta \cdot r_k$ by the greedy procedure described above. We also define
	\begin{align*}
		H(\delta) \equiv \mathcal{  H }(\epsilon^* ) = \sum_{ g \in \Gcal} \frac{  \lambda_g \cdot \xi(\epsilon^*) }{   w_{0,g} + \epsilon^* } = \sum_{g \in \Gcal} \frac{  \lambda_g \cdot \left(  \sum_{i=1}^{k-1} r_i w_i + \delta \cdot r_k  \right)}{   w_{0,g} + \sum_{i=1}^{k-1} w_i + \delta } =   r_k \cdot \sum_{g \in \Gcal} \frac{ \lambda_g \cdot \left(  C + \delta  \right) }{D_g + \delta},
	\end{align*}
	where $C \equiv \left(\sum_{i =1}^{k-1} w_i \right) / r_k > 0$ and $D_g \equiv w_{0,g} + \left(\sum_{i =1}^{k-1} w_i \right)  > 0$. If $\epsilon^*$ is an maximizer, we must have $H'(\delta) = 0$ and $H''(\delta) \leq 0$. However, according to Claim~\ref{claim:RO_proof_convexity}, if $H'(\delta) = 0$, then $H''(\delta) > 0$, i.e., $H(\cdot)$ is strictly convex around $\delta$ and thus any deviation would increase the value of $H(\cdot)$. This contradicts with the fact that $\mathcal{H}(\epsilon)$ is maximized at $\epsilon^*$. 
	
	\subsubsection{Proof of Claim~\ref{claim:RO_proof_convexity}}
    
    We begin by differentiating $H(x)$ and obtain that
	\[
	H'(x) = \sum_{\ell=1}^L \mu_\ell \frac{D_\ell - C}{(D_\ell + x)^2}
	\quad \text{and} \quad
	H''(x) = (-2) \cdot \sum_{\ell=1}^L \mu_\ell \frac{D_\ell - C}{(D_\ell + x)^3}.
	\]
	
	Fix a $x^*>0$ such that $H'(x^*) = 0$. We define $B_\ell \equiv D_\ell + x^* > 0$ and $t \equiv C + x^* > 0$, implying that $D_\ell - C = (D_\ell + x^*) - (C + x^*) = B_\ell - t$.
	Then,
	\[
	H'(x^*) = \sum_{\ell =1}^L  \frac{\mu_\ell}{B_\ell^2} \cdot (B_\ell - t)  \quad \text{and} \quad
	H''(x^*) = (-2) \cdot \sum_{\ell=1}^L  \mu_\ell  \cdot \frac{B_\ell - t}{B_\ell^3} = (-2) \cdot \sum_{\ell=1}^k  \frac{\mu_\ell}{B_\ell^2}  \cdot \left(1 - \frac{  t }{ B_\ell }\right)  .
	\]
	
	Define $M = \sum_{\ell = 1}^L \mu_\ell / B^2_\ell$ and construct a probability distribution ${\mu'}$ over $[L]$ such that $\mu'_\ell = (1/M) \cdot (\mu_\ell / B^2_\ell)$. Obviously, we have $\sum_{\ell=1}^L \mu'_\ell = 1$. We further define a random variable $B$ such that $B = B_\ell$ with probability $\mu'_\ell$, for $\ell \in [L]$. Under this definition, the condition that $H'(x^*) = 0$ can now be expressed as $t = \mathbb{E}\left[ B \right]$ since
	\[
	0 = \frac{ H'(x^*)  }{M} = \sum_{\ell=1}^L \mu'_\ell (B_\ell - t)  = \mathbb{E}\left[ B - t \right] = \mathbb{E}\left[ B \right] - t.
	\]
	
	To show that $H''(x^*) > 0$, it suffices to show
	\[
	0 > \sum_{\ell = 1}^L \mu'_\ell \cdot \left( 1 - \frac{t}{B_\ell} \right) = 1 -  t \cdot \mathbb{E} \left[  \frac{1}{B}   \right] = 1- \mathbb{E} \left[ B \right] \cdot \mathbb{E} \left[ \frac{1}{B} \right] \quad \Leftrightarrow  \quad \mathbb{E} \left[ B \right] \cdot \mathbb{E} \left[ \frac{1}{B} \right] > 1
	\]
	which holds due to the Cauchy--Schwarz inequality and the fact that $B_\ell$ are not all equal (thus the inequality sign holds strictly ).  

\end{appendices}

\end{document}